\documentclass[10pt]{amsart}
\usepackage{latexsym}
\usepackage{indentfirst}
\usepackage{amssymb,amsfonts,amsmath,amsthm}
\usepackage{graphicx}
\usepackage{mathrsfs}
\usepackage{array}
\usepackage{color}
\usepackage{epstopdf}
\usepackage[all]{xy}
\usepackage{hyperref}
\usepackage{url}

\newtheorem{thm}{Theorem}[section]
\newtheorem{prop}[thm]{Proposition}
\newtheorem{lem}[thm]{Lemma}

\newtheorem{cor}[thm]{Corollary}
\newtheorem{exam}[thm]{Example}

\newtheorem{rmk}[thm]{Remark}
\newtheorem{dfn}[thm]{Definition}

\numberwithin{equation}{section}

\newcommand{\frakm}{{\mathfrak m}}

\newcommand{\frakX}{{\mathfrak X}}
\newcommand{\frakY}{{\mathfrak Y}}

\newcommand{\bC}{{\mathbb C}}

\newcommand{\bG}{{\mathbb G}}

\newcommand{\bL}{{\mathbb L}}

\newcommand{\bN}{{\mathbb N}}

\newcommand{\bQ}{{\mathbb Q}}

\newcommand{\bZ}{{\mathbb Z}}

\newcommand{\calC}{{\mathcal C}}
\newcommand{\calD}{{\mathcal D}}
\newcommand{\calE}{{\mathcal E}}
\newcommand{\calF}{{\mathcal F}}

\newcommand{\calH}{{\mathcal H}}

\newcommand{\calL}{{\mathcal L}}
\newcommand{\calM}{{\mathcal M}}

\newcommand{\calO}{{\mathcal O}}
\newcommand{\calP}{{\mathcal P}}

\newcommand{\rA}{{\mathrm A}}

\newcommand{\rC}{{\mathrm C}}

\newcommand{\rE}{{\mathrm E}}

\newcommand{\rH}{{\mathrm H}}

\newcommand{\rK}{{\mathrm K}}
\newcommand{\rL}{{\mathrm L}}
\newcommand{\rM}{{\mathrm M}}

\newcommand{\rR}{{\mathrm R}}

\newcommand{\rW}{{\mathrm W}}

\newcommand{\Zp}{{\bZ_p}}

\newcommand{\Qp}{{\bQ_p}}

\newcommand{\Cp}{{\bC_p}}

\newcommand{\BdRp}{{\mathrm{B_{dR}^+}}}

\newcommand{\Rinf}{{\widehat R_{\infty}}}
\newcommand{\Rinfp}{{\widehat R_{\infty}^+}}

\newcommand{\OXp}{{\widehat \calO_X^+}}
\newcommand{\OX}{{\widehat \calO_X}}
\newcommand{\OC}{{\calO\bC}}                           

\newcommand{\Coker}{{\mathrm{Coker}}}       
\newcommand{\dlog}{{\mathrm{dlog}}}         
\newcommand{\Ext}{{\mathrm{Ext}}}           
\newcommand{\Frac}{{\mathrm{Frac}}}         
\newcommand{\Gal}{{\mathrm{Gal}}}           
\newcommand{\Hom}{{\mathrm{Hom}}}           
\newcommand{\id}{{\mathrm{id}}}             
\newcommand{\Ima}{{\mathrm{Im}}}            
\newcommand{\Ker}{{\mathrm{Ker}}}           
\newcommand{\RGamma}{{\mathrm{R\Gamma}}}    
\newcommand{\Rlim}{{\mathrm{R}\underleftarrow{\lim}}} 
\newcommand{\Spa}{{\mathrm{Spa}}}           
\newcommand{\Spf}{{\mathrm{Spf}}}           
\newcommand{\Sym}{{\mathrm{Sym}}}           
\newcommand{\Tor}{{\mathrm{{Tor}}}}         

\newcommand{\GL}{{\mathrm{GL}}}             

\newcommand{\cl}{{\mathrm{cl}}}             
\newcommand{\cts}{{\mathrm{cts}}}           
\newcommand{\et}{{\mathrm{\acute{e}t}}}    
\newcommand{\proet}{\mathrm{pro\acute{e}t}} 

\newcommand{\ya}{{\rangle}}
\newcommand{\za}{{\langle}}

\begin{document}
\title{A $p$-adic Simpson correspondence for rigid analytic varieties}

\author{Yupeng Wang}

\address{Morningside Center of Mathematics No.505, Chinese Academy of Sciences, Zhongguancun East Road 55, Haidian, Beijing, 100190, China.}
\email{wangyupeng@amss.ac.cn}



\maketitle

\begin{abstract}
  In this paper, we establish a $p$-adic Simpson correspondence on the arena of Liu-Zhu for rigid analytic varieties $X$ over $\Cp$ with a liftable good reduction by constructing a new period sheaf on $X_{\proet}$. To do so, we use the theory of cotangent complex after Beilinson and Bhatt. Then we give an integral decompletion theorem and complete the proof by local calculations. Our construction is compatible with the previous works of Faltings and Liu-Zhu.
\end{abstract}

\tableofcontents

\section{Introduction}\label{Sec-Introduction}
 In the theory of complex geometry, for a compact K\"{a}hler manifold $X$, in \cite{Sim} Simpson established a tensor equivalence between the category of semisimple flat vector bundles on $X$ and the category of polystable Higgs bundles with vanishing Chern classes. Nowadays, such a correspondence is known as the non-abelian Hogde theory or the Simpson correspondence. There is a good theory of Simpson correspondence for smooth varieties in characteristic $p>0$ admitting a lifting modulo $p^2$ (cf. \cite{OV}). So we ask for a $p$-adic analogue of Simpson's correspondence. 
 
 The first step is due to Deninger-Werner \cite{DW}. They gave a partial analogue of classical Narasimhan-Seshadri theory by studying parallel transport for vector bundles for curves.  At the same time, Faltings \cite{Fal2} constructed an equivalence between the category of small generalised representations and the category of small Higgs bundles for schemes $\frakX_0$ with toroidal singularities over $\calO_k$, the ring of integers of some $p$-adic local field $k$, under a certain deformation assumption. His method was elaborated and generalized by Abbes-Gros-Tsuji \cite{AGT} and related with the integral $p$-adic Hodge theory by Morrow-Tsuji \cite{MT} recently. When $X$ is a rigid analytic space over $k$, Liu-Zhu \cite{LZ} related a Higgs bundle on $X_{\hat{\bar k},\et}$ to each $\Qp$-local system on $X_{\et}$ and proved that the resulting Higgs field must be nilpotent (cf. \cite[Theorem 2.1]{LZ}). Their work was generalized to the logarithmic case in \cite{DLLZ22}. However, their Higgs functor is not an equivalence, so it is still open to classify Higgs bundles coming from representations. In \cite{Heu}, for smooth rigid spaces $X$ over $\hat{\bar k}$, Heuer established an equivalence between the category of one-dimensional $\hat{\bar k}$-representation of the fundamental group $\pi_1(X)$ and the category of pro-finite-\'etale Higgs bundles.  Using his method, Heuer-Mann-Werner \cite{HMW} constructed a Simpson correspondence for abeloids over $\hat{\bar k}$.
 
 In this paper, we establish an equivalence between the category of small generalised representations (Definition \ref{Dfn-small generalised representation}) and the category of small Higgs bundles (Definition \ref{Dfn-small Higgs bundle}) for rigid analytic varieties $X$ with liftable (see Notations) good reductions $\frakX$ over $\calO_{\Cp}$ in the arena of the work of Liu-Zhu. Our construction is global and the main ingredient is a new overconvergent period sheaf $\OC^{\dag}$ endowed with a canonical Higgs field $\Theta$ on $X_{\proet}$, which can be viewed as a kind of $p$-adic complete version of the peroid sheaf $\OC$ due to Hyodo \cite{Hy}. The main theorem is stated as follows:

\begin{thm}[Theorem \ref{p-adic Simpson}]\label{main theorem}  Assume $a\geq \frac{1}{p-1}$. Let $\frakX$ be a liftable smooth formal scheme over $\calO_{\Cp}$ of relative dimension $d$ with the rigid generic fibre $X$ and $\nu:X_{\proet}\to\frakX_{\et}$ be the natural projection of sites. Then there is an overconvergent period sheaf $\OC^{\dagger}$ endowed with a canonical Higgs field $\Theta$ such that the following assertions are true:
  \begin{enumerate}
      \item For any $a$-small generalised representation $\calL$ of rank $l$ on $X_{\proet}$, let $\Theta_{\calL}:=\id_{\calL}\otimes\Theta$ be the induced Higgs field on $\calL\otimes_{\widehat \calO_X}\calO\bC^{\dagger}$,
      then $\rR\nu_*(\calL\otimes_{\OX}\OC^{\dagger})$ is discrete. Denote $\calH(\calL):=\nu_*(\calL\otimes_{\OX}\OC^{\dagger})$ and $\theta_{\calH(\calL)} = \nu_*\Theta_{\calL}$. Then $(\calH(\calL),\theta_{\calH(\calL)})$ is an $a$-small Higgs bundle of rank $l$.
      
      \item For any $a$-small Higgs bundle $(\calH,\theta_{\calH})$ of rank $l$ on $\frakX_{\et}$, let $\Theta_{\calH} := \id_{\calH}\otimes\Theta+\theta_{\calH}\otimes\id_{\calO\bC^{\dagger}}$ be the induced Higgs field on $\calH\otimes_{\calO_{\frakX}}\calO\bC^{\dagger}$ and denote 
      \[
      \calL(\calH,\theta_{\calH}) = (\calH\otimes_{\calO_{\frakX}}\OC^{\dagger})^{\Theta_{\calH}=0}.
      \]
      Then $\calL(\calH,\theta_{\calH})$ is an $a$-small generalised representation of rank $l$.
      
      \item The functor $\calL\mapsto(\calH(\calL),\theta_{\calH(\calL)})$ induces an equivalence from the category of $a$-small generalised representations to the category of $a$-small Higgs bundles, whose quasi-inverse is given by $(\calH,\theta_{\calH})\mapsto \calL(\calH,\theta_{\calH})$. The equivalence preserves tensor products and dualities and identifies the Higgs complexes
      \[
      {\rm HIG}(\calL\otimes_{\OX}\OC^{\dagger},\Theta_{\calL})\simeq {\rm HIG}(\calH(\calL)\otimes_{\calO_{\frakX}}\OC^{\dagger},\Theta_{\calH(\calL)}).
      \]
      
      \item Let $\calL$ be an $a$-small generalised representation with associated Higgs bundle $(\calH,\theta_{\calH})$. Then there is a canonical quasi-isomorphism
      \[\rR\nu_*(\calL)\simeq{\rm HIG}(\calH,\theta_{\calH}),\]
      where ${\rm HIG}(\calH,\theta_{\calH})$ is the Higgs complex induced by $(\calH,\theta_{\calH})$. In particular, $\rR\nu_*(\calL)$ is a perfect complex of $\calO_{\frakX}[\frac{1}{p}]$-modules concentrated in degree $[0,d]$.
      
      \item Assume $f:\frakX\to\frakY$ is a smooth morphism between liftable smooth formal schemes over $\calO_{\Cp}$. Let $\widetilde \frakX$ and $\widetilde \frakY$ be the fixed $A_2$-liftings of $\frakX$ and $\frakY$, respectively. Assume $f$ lifts to an $A_2$-morphism $\widetilde f: \widetilde \frakX\to\widetilde \frakY$, then the equivalence in (3) is compatible with the pull-back along $f$.
  \end{enumerate}
\end{thm}

Note that when $\calL = \OX$, we get $(\calH(\OX),\theta_{\calH(\OX)}) = (\calO_{\frakX}[\frac{1}{p}],0)$. So our result can be viewed as a generalization of \cite[Proposition 3.23]{Sch3}. Theorem \ref{main theorem} (3) also provides a way to compute the pro-\'etale cohomology for a small generalised representation $\calL$. More precisely, we get a quasi-isomorphism
\[\RGamma(X_{\proet},\calL)\simeq \RGamma(\frakX_{\et},{\rm HIG}(\calH(\calL),\theta_{\calH(\calL)})).\]
If moreover, $\frakX$ is proper, then we get a finiteness result on pro-\'etale cohomology of small generalised representations.
\begin{cor}
  Keep notations as Theorem \ref{main theorem} and assume furthermore $\frakX$ is proper. Then for any $a$-small generalised representation $\calL$, $\RGamma(X_{\proet},\calL)$ is concentrated in degree $[0,2d]$ and has cohomologies as finite dimensional $\Cp$-spaces.
\end{cor}
 
 The overconvergent period sheaf $\OC^{\dagger}$ (with respect to a certain lifting of $\frakX$) has $\OC$ as a subsheaf. Indeed, it is a direct limit of certain $p$-adic completions of $\OC$.
 In particular, when $\frakX$ comes from a scheme $\frakX_0$ over $\calO_k$ and the generalised representation $\calL$ comes from a $\Zp$-local system on the rigid generic fibre $X_0$ of $\frakX_0$, our construction coincides with the work of Liu-Zhu (Remark \ref{compare with Liu-Zhu}). On the other hand, $\OC^{\dagger}$ is related with an obstruction class $\cl(\calE^+)$ solving a certain deformation problem (Remark \ref{Faltings obstruction} and Proposition \ref{Compare obstruction}). Since the class $\cl(\calE^+)$ is exactly the one used to establish the Simpson correspondence in \cite{Fal2}, our construction is compatible with the works of Faltings and Abbes-Gros-Tsuji (Remark \ref{compare with Faltings}). These answer a question appearing in \cite[Remark 2.5]{LZ}. Another answer was announced in \cite{YZ} in a different way.
 
 Since we need to take $p$-adic completions of $\OC$, we have to find its integral models. Note that $\OC$ is a direct limit of symmetric products of Faltings' extension, which was constructed for varieties by Faltings \cite{Fal1} at first and revisited by Scholze \cite{Sch2} in the rigid analytic case. So we are reduced to finding an integral version of Faltings' extension.
 To do so, we use the method of cotangent complex which was established and developed in \cite{Qui},\cite{Ill1},\cite{Ill2},\cite{GR} etc., and was systematically used in the $p$-adic theory by \cite{Sch1},\cite{Bei}, \cite{Bha} etc.. Finally, the proof of Theorem \ref{main theorem} is based on some explicit local calculations, especially an integral decompletion theorem  (Theorem \ref{Strong Decompletion}) for small representations, which can be regarded as a generalization of \cite[Appendix A]{DLLZ22}.
 
 \subsection{Notations}
 
 Let $k$ be a complete discrete valuation field of mixed characteristics $(0,p)$ with ring of integers $\calO_k$ and perfect residue field $\kappa$. We normalise the valuation on $k$ by setting $\nu_p(p)=1$ and the associated norm is given by $\|\cdot\| = p^{-\nu_p(\cdot)}$. We denote $k_0 = \Frac(\rW(\kappa))$ the maximal absolutely unramified subfield of $k$. Put $\calD_k = \calD_{k/k_0}$ the relative differential ideal of $\calO_k$ over $\rW(\kappa)$.

 Let $\overline{k}$ be a fixed algebraic closure of $k$ and $\Cp = \widehat{\overline{k}}$ be its $p$-adic completion. We denote by $\calO_{\Cp}$ (resp. $\frakm_{\Cp}$) the ring of integers of $\Cp$ (resp, the maximal ideal of $\calO_{\Cp}$). In this paper, when we write $p^aA$ for some $\calO_{\Cp}$-module $A$, we always assume $a\in\bQ$. An $\calO_{\Cp}$-module $M$ is called {\bf almost vanishing} if it is $\frakm_{\Cp}$-torsion and in this case, we write $M^{\rm al}=0$. A morphism $f:M\rightarrow N$ of $\calO_{\Cp}$-modules is {\bf almost injective} (resp. {\bf almost surjective}) if $\Ker(f)^{\rm al} = 0$ (resp. $\Coker(f)^{\rm al}=0$). A morphism is an {\bf almost isomorphism} if it is both almost injective and almost surjective.

 We choose a sequence $\{1, \zeta_p, \dots, \zeta_{p^n}, \dots\}$ such that $\zeta_{p^n}$ is a primitive $p^n$-th root of unity in $\overline k$ satisfying $\zeta_{p^{n+1}}^p = \zeta_{p^n}$ for every $n\geq 0$. For every $\alpha\in\bZ[\frac{1}{p}]\cap(0,1)$, one can (uniquely) write $\alpha = \frac{t(\alpha)}{p^{n(\alpha)}}$ with $\gcd(t(\alpha),p) = 1$ and $n(\alpha)\geq 1$. Then we define that $\zeta^{\alpha} := \zeta_{p^{n(\alpha)}}^{t(\alpha)}$ when $\alpha\neq 0$ and that $\zeta^{\alpha}=1$ when $\alpha = 0$.

 We always fix an element $\rho_k\in \Cp$ with $\nu_p(\rho_k) = \nu_p(\calD_k)+\frac{1}{p-1}$.
 Let $\rA_{\inf,k} = \rW(\calO_{\bC_p^{\flat}})\otimes_{\rW(\kappa)}\calO_k$ be the period ring of Fontaine. Then there is a surjective homomorphism $\theta_k: \rA_{\inf,k}\rightarrow \calO_{\Cp}$ whose kernel is a principal ideal by \cite[Proposition 3.1.9]{FF}. We fix a generator $\xi_k$ of $\Ker(\theta_k)$. For instance, when $k = k_0$ is absolutely unramified, then we choose $\rho_{k} = \zeta_p-1$ and $\xi_k = \frac{[\epsilon]-1}{[\epsilon]^{\frac{1}{p}}-1}$ for $\epsilon = (1, \zeta_p, \zeta_{p^2}, \cdots)\in \calO_{\Cp}^{\flat}$. Put $\rA_2 = \rA_{\inf,k}/\xi_k^2$ and denote Fontaine's $p$-adic analogue of $2\pi i$ by $t=\log[\epsilon]$.
 
 For a $p$-adic formal scheme $\frakX$ over $\calO_{\Cp}$, we say it is {\bf smooth} if it is formally smooth and locally of topologically finite type. We say $\frakX$ is {\bf liftable} if it admits a lifting $\widetilde \frakX$ to $\Spf(\rA_2)$. In this paper, we always assume $\frakX$ is liftable. Let $X$ be the rigid analytic generic fibre of $\frakX$ and denote by $\nu: X_{\proet}\rightarrow \frakX_{\et}$ the natural projection of sites. Let $\OXp$ ane $\OX$ be the completed structure sheaves on $X_{\proet}$ in the sense of \cite[Definitiuon 4.1]{Sch2}. Both of them can be viewed as $\calO_{\frakX}$-algebras via the projection $\nu$.
 
 Let $K$ be an object in the derived category of complexes of $\Zp$-modules. We denote by $\hat K$ the derived $p$-adic completion $\Rlim_nK\otimes_{\Zp}\Zp/p^n$. In particular, for a morphism $A\to B$ of $\Zp$-algebras, we denote the derived $p$-adic completion of cotangent complex $\rL_{B/A}$ by $\widehat \rL_{B/A}$. In this paper, for two complexes $K_1$ and $K_2$ of (sheaves of) modules, we write $K_1\simeq K_2$ if they are quasi-isomorphic. For two modules or sheaves $M_1$ and $M_2$, we write $M_1\cong M_2$ if they are isomorphic.
 
 \subsection{Organization}
 In Section \ref{Sec 2}, we construct the integral Faltings' extension by using $p$-complete cotangent complexes and explain how it is related to the deformation theory. At the end of this section we construct the desired overconvergent sheaf. In Section \ref{Sec 3}, we prove an integral decompletion theorem for small representations. In Section \ref{Sec 4}, we establish a local version of Simpson correspondence. We first consider the trivial representation and then reduce the general case to this special case. Finally, in Section \ref{Sec 5}, we state  and prove our main theorem. The appendix specifies some notations and includes some elementary facts that were used in previous sections.

\section*{Acknowledgments}
 The paper consists of main results of the author's Ph.D. Thesis in Peking University. The author expresses his deepest gratitude to his advisor, Ruochuan Liu, for suggesting this topic, for useful advice on this paper, and for his warm encouragement and generous and consistent help during the whole time of the author's Ph.D. study. The author thanks David Hansen, Gal Porat and Mao Sheng for their comments on the earlier draft. The auhor also thanks anonymous referees for their careful reading, professional comments and valuable suggestions to improve this paper.

\section{Integral Faltings' extension and period sheaves}\label{Sec 2}
  We construct the overconvergent period sheaf $\OC^{\dagger}$ in this section. In order to do so, we have to construct an integral version of Faltings' extension at first.
  \subsection{Integral Faltings' extension}
  We first discuss the properties of the cotangent complex.
  The following Lemmas are well-known, but for the convenience of readers, we include their proofs here.
  \begin{lem}\label{regular sequence}
    Let $A$ be a ring. Suppose that $(f_1, \dots, f_n)$ is a regular sequence in $A$ and generates the ideal $I = (f_1,... ,f_n)$, then $\rL_{(A/I)/A}\simeq (I/I^2)[1]$.
  \end{lem}
  \begin{proof}
  Regard $A$ as a $\bZ[X_1, \dots, X_n]$-algebra by mapping $X_i$ to $f_i$ for every $i$. Since $f_1, \dots, f_n$ is a regular sequence in $A$, for any $i\geq 1$, we have
  \[
  \Tor^{\bZ[X_1, \cdots, X_n]}_i(\bZ,A)=0.
  \]
  It follows from \cite[8.8.4]{Wei} that 
  \[
  \rL_{(A/I)/A}\simeq \rL_{\bZ/\bZ[X_1, \cdots, X_n]}\otimes^L_{\bZ[X_1, \cdots, X_n]}
  A.\]
  So we may assume $A = \bZ[X_1, \cdots, X_n]$ and $I = (X_1, \cdots, X_n)$. From homomorphisms $\bZ\rightarrow A \rightarrow A/I$ of rings, we get an exact triangle
  \[\xymatrix@C=0.4cm{
    \rL_{A/\bZ}\otimes^LA/I\ar[r] & \rL_{(A/I)/\bZ}\ar[r] & \rL_{(A/I)/A}\ar[r]&
  }.\]
  The middle term is trivial since $A/I=\bZ$ and hence we deduce that
  \[
  \rL_{(A/I)/A}\simeq(\rL_{A/\bZ}\otimes^{L}_A\bZ)[1]\simeq (I/I^2)[1]
  \]
  as desired.
\end{proof}
\begin{lem}\label{Fontaine theorem}
  \begin{enumerate}
      \item The map $\dlog: \mu_{p^{\infty}} \to \Omega^1_{\calO_{\bar k}/\calO_k}, \zeta_{p^n}\mapsto \frac{d\zeta_{p^n}}{\zeta_{p^n}}$ induces an isomorphism
    \[
      \dlog: \overline k/\rho_k^{-1}\calO_{\bar k}\otimes\Zp(1)\rightarrow\Omega^1_{\calO_{\bar k}/\calO_k},
    \]
    where $\Zp(1)$ denotes the Tate twist.
    \item $\rL_{\calO_{\overline k}/\calO_k}\simeq \Omega^1_{\calO_{\overline k}/\calO_k}[0]$.
    
    \item $\widehat \rL_{\calO_{\Cp}/\calO_k}\simeq \frac{1}{\rho_k}\calO_{\Cp}(1)[1]$.
  \end{enumerate}
\end{lem}
\begin{proof}
  \begin{enumerate}
      \item This is \cite[Th\'eor\`eme 1']{Fon}.
      \item This is \cite[Theorem 1.3]{Bei}.
      \item This follows from $(1)$ and $(2)$ after taking derived $p$-completions on both sides.
  \end{enumerate}
\end{proof}
\begin{cor}\label{base coefficients}
  \begin{enumerate}
      \item $\widehat \rL_{\calO_{\Cp}/A_{\inf,k}}[-1]\simeq \frac{1}{\rho_k}\calO_{\Cp}(1)[0]\simeq \xi_k A_{\inf,k}/\xi_k^2A_{\inf,k}[0]$.
      \item $\widehat \rL_{\calO_{\Cp}/A_2}\simeq \frac{1}{\rho_k}\calO_{\Cp}(1)[1]\oplus\frac{1}{\rho_k^2}\calO_{\Cp}(2)[2].$
  \end{enumerate}
\end{cor}
\begin{proof}
  \begin{enumerate}
       \item Considering the morphisms $\calO_k\to A_{\inf,k}\to\calO_{\Cp}$ of rings, we have an exact triangle
      \[
        \rL_{A_{\inf,k}/\calO_k}\widehat \otimes^L_{A_{\inf,k}}\calO_{\Cp} \to\widehat \rL_{\calO_{\Cp}/\calO_k}\to \widehat \rL_{\calO_{\Cp}/A_{\inf,k}}\to .
      \]
      Since
      \[
      \widehat \rL_{A_{\inf,k}/\calO_k}\simeq  \rL_{A_{\inf}/\rW(\kappa)}\widehat \otimes_{\rW(\kappa)}^L\calO_k = 0,
      \]
      the first quasi-isomorphism follows from Lemma \ref{Fontaine theorem} (3). Now, the second quasi-isomorphism is straightforward from Lemma \ref{regular sequence}.
      \item Considering the morphisms $A_{\inf,k}\to A_2\to\calO_{\Cp}$ of rings, we have the exact triangle
      \[
        \rL_{A_2/A_{\inf,k}}\widehat \otimes^L_{A_2}\calO_{\Cp}\to \widehat \rL_{\calO_{\Cp}/A_{\inf,k}} \to \widehat \rL_{\calO_{\Cp}/A_2}\to .
      \]
      Combining Lemma \ref{regular sequence} with $(1)$, the above exact triangle reduces to 
      \[
        \xi_k^2A_{\inf,k}/\xi_k^4A_{\inf,k}\otimes_{A_2}\calO_{\Cp}[1]\to \xi_k A_{\inf,k}/\xi_k^2A_{\inf,k}[1] \to \widehat \rL_{\calO_{\Cp}/A_2}\to .
      \]
      Now we complete the proof by noting that the first arrow is trivial.
  \end{enumerate}
\end{proof}
   We identify $\calO_{\Cp}(1)$ with $\calO_{\Cp}t$, where $t$ is Fontaine's $p$-adic analogue of $2\pi i$. It follows from Lemma \ref{Fontaine theorem} (1) that the sequence $\{\dlog(\zeta_{p^n})\}_{n\geq 0}$ can be identified with the element $t\in \frac{1}{\rho_k}\calO_{\Cp}(1)$. If we regard $A_{\inf,k}$ as a subring of $\BdRp$ and identify $t\BdRp/t^2\BdRp$ with $\Cp(1)$, then Corollary \ref{base coefficients} says that $t$ and $\rho_k\xi_k$ in $\Cp(1)$ differ by a $p$-adic unit in $\calO_{\Cp}^{\times}$.
\begin{rmk}
  The corollary is still true if one replaces $\Cp$ by any closed subfield $K\subset \Cp$ containing $\mu_{p^{\infty}}$. All results in this paper hold for $K$ instead of $\Cp$.
\end{rmk}

Now we construct the integral Faltings' extension in the local case. We fix some notations as follows:

Let $\frakX = \Spf(R^+)$ be a smooth formal scheme over $\Spf(\calO_{\Cp})$ endowed with an \'etale morphism 
  \[\Box: \frakX\to \hat \bG_m^d = \Spf(\calO_{\Cp}\za \underline T^{\pm 1}\ya),\] 
  where $\calO_{\Cp}\za\underline T^{\pm 1}\ya = \calO_{\Cp}\za T_1^{\pm 1},\dots, T_d^{\pm 1}\ya.$ We say $\frakX$ is {\bf small} in this case.
  Let $X = \Spa(R,R^+)$ be the rigid analytic generic fibre of $\frakX$ and $X_{\infty} = \Spa(\Rinf,\Rinfp)$ be the affinoid perfectoid space associated to the base-change of $X$ along the Galois cover
  \[
  \bG_{m,\infty}^d = \Spa(\Cp\za\underline T^{\pm \frac{1}{p^{\infty}}}\ya,\calO_{\Cp}\za\underline T^{\pm\frac{1}{p^{\infty}}}\ya) \to \bG_m^d = \Spa(\Cp\za\underline T^{\pm 1}\ya,\calO_{\Cp}\za\underline T^{\pm 1}\ya).
  \]
  Denote by $\Gamma$ the Galois group of the cover $X_{\infty}\to X$ and let $\gamma_i$ be in $\Gamma$ satisfying 
  \begin{equation}\label{gamma action on chart}
        \gamma_i(T_j^{\frac{1}{p^n}}) = \zeta_{p^n}^{\delta_{ij}}T_j^{\frac{1}{p^n}}
  \end{equation}
  for any $1\leq i,j\leq d$ and $n\geq 0$. Here, $\delta_{ij}$ denotes the Kronecker's delta. Then $\Gamma\cong \Zp\gamma_1\oplus\dots\oplus\Zp\gamma_d$.
Let $\widetilde R^+$ be a lifting of $R^+$ along $A_2\to \calO_{\Cp}$. Then the morphisms $\widetilde R^+\to R^+\to \widehat R_{\infty}^+$ of rings give an exact triangle of $p$-complete cotangent complexes
\begin{equation}\label{exact triangle for samll case}
     \rL_{R^+/\widetilde R^+} \widehat \otimes^L_{R^+} \Rinfp\to\widehat \rL_{\Rinfp/\widetilde R^+}\to\widehat \rL_{\Rinfp/R^+}\to.
\end{equation}
The first term is easy to handle. Indeed, combining \cite[8.8.4]{Wei} with Corollary \ref{base coefficients} (2), we deduce that 
\[
  \rL_{R^+/\widetilde R^+}\widehat \otimes^L_{R^+} \Rinfp\simeq \frac{1}{\rho_k}\Rinfp(1)[1]\oplus\frac{1}{\rho_k^2}\Rinfp(2)[2].
\]
Now we compute the third term of (\ref{exact triangle for samll case}).
\begin{lem}\label{third term}
  We have $\widehat \rL_{\Rinfp/R^+}\simeq \widehat \Omega^1_{R^+}\otimes_{R^+}\Rinfp[1]$, where $\widehat \Omega^1_{R^+}$ denotes the module of formal differentials of $R^+$ over $\calO_{\Cp}$.
\end{lem}
\begin{proof}
  Since $R^+$ is \'etale over $\calO_{\Cp}\za\underline T^{\pm 1}\ya$, thanks to \cite[Lemma 3.14]{BMS1}, we are reduced to the case $R^+=\calO_{\Cp}\za\underline T^{\pm 1}\ya$. For any $n\geq 0$, put $A^+_n = \calO_{\Cp}[\underline T^{\pm \frac{1}{p^n}}]$ and denote $A_{\infty}^+ = \varinjlim_nA_n^+$. Since all rings involved are $p$-torsion free, we get
  \[
  \widehat \rL_{\widehat R_{\infty}^+/R^+} \simeq \widehat \rL_{A_{\infty}^+/A_0^+}.
  \]
  By \cite[Chapitre II(1.2.3.4)]{Ill1}, we see that 
  \[
  \rL_{A_{\infty}^+/A_0^+} = \varinjlim_n\rL_{A_n^+/A_0^+}.
  \]
  Since all $A_n^+$'s are smooth over $\calO_{\Cp}$, from the exact triangle
  \[
  \rL_{A_0^+/\calO_{\Cp}}\otimes^L_{A_0^+}A_n^+\to\rL_{A_n^+/\calO_{\Cp}}\to\rL_{A_n^+/A_0^+}\to,
  \]
  we deduce that 
  \[
    \rL_{A_n^+/A_0^+}\simeq A_n^+\otimes_{A_0^+}\frac{1}{p^n}\Omega^1_{A_0^+}/\Omega^1_{A_0^+}[0],
  \]
  where we identify $\Omega^1_{A_n^+}$ with $A_n^+\otimes_{A_0^+}\frac{1}{p^n}\Omega^1_{A_0^+}$. Therefore, we get 
  \[
    \rL_{A_{\infty}^+/A_0^+} \simeq A_{\infty}^+\otimes_{A_0^+}\Omega^1_{A_0^+}\otimes_{\Zp}(\Qp/\Zp)[0].
  \]
  Now the result follows by taking $p$-completions.
\end{proof}
Since $R^+$ admits a lifting $\widetilde R^+$ to $A_2$, the composition 
\[
  \widehat \rL_{\Rinfp/R^+}\simeq \widehat \rL_{\rA_2(\widehat R_{\infty}^{+})/\widetilde R^+}\widehat \otimes^L_{\rA_2(\widehat R_{\infty}^{+})}\Rinfp \to \widehat \rL_{\Rinfp/\widetilde R^+}
\]
defines a section of $\widehat \rL_{\Rinfp/\widetilde R^+}\to\widehat \rL_{\Rinfp/R^+}$. Since the exact triangle $(\ref{exact triangle for samll case})$ is $\Gamma$-equivariant, by taking cohomologies along (\ref{exact triangle for samll case}), we get the following proposition.
\begin{prop}\label{local Faltings extension}
  There exists a $\Gamma$-equivariant short exact sequence of $\Rinfp$-modules
  \begin{equation}\label{exact sequence for small case}
      0\to \frac{1}{\rho_k}\Rinfp(1)\to E^+\to \Rinfp\otimes_{R^+}\widehat \Omega_{R^+}^1\to 0,
  \end{equation}
  where $E^+ = \rH^{-1}(\widehat \rL_{\Rinfp/\widetilde R^+})$. Moreover, the above exact sequence admits a (non-$\Gamma$-equivariant) section such that $E^+\cong \frac{1}{\rho_k}\Rinfp(1)\oplus \Rinfp\otimes_{R^+}\widehat \Omega_{R^+}^1$ as $\widehat R_{\infty}^+$-modules.
\end{prop}
\begin{rmk}
  When $R^+$ is the base-change of some formal smooth $\calO_k$-algebra $R_0^+$ of topologically finite type along $\calO_k\to\calO_{\Cp}$, then it admits a canonical lifting $\widetilde R^+ = R_0^+\widehat \otimes_{\calO_k}A_2$. After inverting $p$, the resulting $E^+$ becomes the usual Faltings' extension and the corresponding sequence (\ref{exact sequence for small case}) is even $\Gal(\bar k/k)$-equivariant.
\end{rmk}

  We describe the $\Gamma$-action on $E^+$. For any $1\leq i\leq d$, by the proof of Lemma \ref{third term}, the compatible sequence $\{\dlog(T_i^{\frac{1}{p^n}})\}_{n\geq 0}$ defines an element $x_i\in E^+$, which goes to $\dlog T_i$ via the projection $E^+\to \Rinfp\otimes_{R^+}\widehat \Omega^1_{R^+}$. Since $\Gamma$ acts on $T_i$'s via (\ref{gamma action on chart}), we deduce that for any $1\leq i,j\leq d$,
  \[
  \gamma_i(x_j) = x_j+\delta_{ij}.
  \]
  In summary, we have the following proposition.
  \begin{prop}\label{good basis}
    The $\Rinfp$-module $E^+$ is free of rank $d+1$ and has a basis $\frac{t}{\rho_k}, x_1, \dots, x_d$ such that
    \begin{enumerate}
        \item for any $1\leq i\leq d$, $x_i$ is a lifting of $\dlog(T_i)\in \Rinfp\otimes_{R^+}\widehat \Omega^1_{R^+}$ and that
        \item for any $1\leq i,j\leq d$, $\gamma_i(x_j) = x_j+\delta_{ij}t$.
    \end{enumerate}
    Moreover, let $c: \Gamma \to \Hom_{R^+}(\widehat \Omega_{R^+}^1,\frac{1}{\rho_k}\Rinfp(1))$ be the map carrying $\gamma_i$ to $c(\gamma_i)$, which sends $\dlog(T_j)$ to $\delta_{ij}t$. Then the cocycle determined by $c$ in $\rH^1(\Gamma,\Hom_{R^+}(\widehat \Omega^1_{R^+},\frac{1}{\rho_k}\Rinfp(1)))$ coincides with the extension class represented by $E^+$ in $\Ext_{\Gamma}^1(\Rinfp\otimes_{R^+}\widehat \Omega^1_{R^+},\frac{1}{\rho_k}\Rinfp(1))$ via the canonical isomorphism
    \[
      \rH^1(\Gamma,\Hom_{R^+}(\widehat \Omega^1_{R^+},\frac{1}{\rho_k}\Rinfp(1)))\cong \Ext_{\Gamma}^1(\Rinfp\otimes_{R^+}\widehat \Omega^1_{R^+},\frac{1}{\rho_k}\Rinfp(1)).
    \]
  \end{prop}
  \begin{proof}
    It remains to prove the ``moreover'' part. By (1), the extension class of $E^+$ is represented by the cocycle 
    \[
    f:\Gamma\to \Hom_{\Rinfp}(\Rinfp\otimes_{R^+}\widehat \Omega_{R^+}^1,\frac{1}{\rho_k}\Rinfp(1))\cong \Hom_{R^+}(\widehat \Omega_{R^+}^1,\frac{1}{\rho_k}\Rinfp(1))
    \]
    such that $f(\gamma)(\dlog(T_i)) = \gamma(x_i)-x_i$ for any $\gamma\in \Gamma$ and any $i$. However, by (2), $f$ is exactly $c$. We are done.
  \end{proof}

 Now we extend the above construction to the global case. 
 Let $\frakX$ be a smooth formal scheme over $\calO_{\Cp}$ with a fixed lifting $\widetilde \frakX$ to $A_2$. Denote by $X$ its rigid analytic generic fibre over $\Cp$. We regard both $\calO_{\frakX}$ and $\calO_{\widetilde \frakX}$ as sheaves on $X_{\proet}$ via the projection $\nu: X_{\proet}\to\frakX_{\et}$ (note that $\frakX$ and $\widetilde \frakX$ has the same \'etale site).
 Considering morphisms of sheaves of rings $\calO_{\widetilde \frakX}\to\calO_{\frakX}\to\OXp$, we get an exact triangle 
 \begin{equation}\label{exact triangle in global case}
   \rL_{\calO_{\frakX}/\calO_{\widetilde \frakX}}\widehat \otimes_{\calO_{\frakX}}^L\OXp\to\widehat \rL_{\OXp/\calO_{\widetilde \frakX}}\to\rL_{\OXp/\calO_{\frakX}}\to.
 \end{equation}
 Similar to the local case, the first term becomes
 \[
   \rL_{\calO_{\frakX}/\calO_{\widetilde \frakX}}\widehat \otimes_{\calO_{\frakX}}^L\OXp\simeq \rL_{\calO_{\Cp}/A_2}\otimes^L_{\calO_{\Cp}}\OXp
 \]
 and the composition 
 \[
 \widehat \rL_{\OXp/\calO_{\widetilde \frakX}}\simeq \widehat \rL_{\rA_2(\widehat \calO_X^{+})/\calO_{\widetilde \frakX}}\widehat \otimes^L_{\rA_2(\widehat \calO_X^{+})}\OXp\to \widehat \rL_{\OXp/\calO_{\widetilde \frakX}}
 \]
 defines a section of $\widehat \rL_{\OXp/\calO_{\widetilde \frakX}}\to\rL_{\OXp/\calO_{\frakX}}$.
 
 We claim that 
 \begin{equation}\label{claim on third term}
   \widehat \rL_{\OXp/\calO_{\frakX}}\simeq \OXp\otimes_{\calO_{\frakX}}\widehat \Omega_{\frakX}^1[1].
 \end{equation}
 Granting this, taking cohomologies along (\ref{exact triangle in global case}), we get the following theorem.
\begin{thm}\label{Integral Faltings extension}
  There is an exact sequence of sheaves of $\OXp$-modules
  \begin{equation}\label{exact sequence in global case}
      0\to\frac{1}{\rho_k}\OXp(1)\to\calE^+\to\OXp\otimes_{\calO_{\frakX}}\widehat \Omega^1_{\frakX}\to 0,
  \end{equation}
  where $\calE^+ =\rH^{-1}(\widehat \rL_{\OXp/\calO_{\widetilde \frakX}})$.
\end{thm}
\begin{rmk}\label{Faltings obstruction}
   Apply $\rR\Hom(-,\frac{1}{\rho_k}\OXp(1))$ to the exact triangle (\ref{exact triangle in global case}) and consider the induced long exact sequence
   \[
    \cdots\to \Ext^1(\widehat \rL_{\calO_{\frakX}/\calO_{\widetilde \frakX}}\widehat \otimes_{\calO_{\frakX}}\OXp,\frac{1}{\rho_k}\OXp(1))\xrightarrow{\partial}\Ext^2(\widehat \rL_{\OXp/\calO_{\frakX}},\frac{1}{\rho_k}\OXp(1))\to \cdots
   \]
   and the commutative diagram
    \[
   \xymatrix@C=0.4cm{
   \Ext^1(\widehat \rL_{\calO_{\frakX}/\calO_{\widetilde \frakX}}\widehat \otimes_{\calO_{\frakX}}\OXp,\frac{1}{\rho_k}\OXp(1))\ar[d]^{\cong}\ar[r]^{\partial}& \Ext^2(\widehat \rL_{\OXp/\calO_{\frakX}},\frac{1}{\rho_k}\OXp(1))\ar[d]^{\cong}\\
   \Hom(\frac{1}{\rho_k}\calO_{\frakX}(1),\frac{1}{\rho_k}\OXp(1))\ar[r]^{\partial}& \Ext^1(\OXp\otimes_{\calO_{\frakX}}\widehat \Omega^1_{\calO_{\frakX}},\frac{1}{\rho_k}\OXp(1)).
   }
   \]
   Then the extension class $[\calE^+]$ associated to $\calE^+$ is the image of the natural inclusion
   $\frac{1}{\rho_k}\calO_{\frakX}(1)\to\frac{1}{\rho_k}\OXp(1)$ via the connecting map $\partial$. By construction, it is the obstruction class to lift $\OXp$ (as a sheaf of $\calO_{\frakX}$-algebras) to a sheaf of $\calO_{\widetilde \frakX}$-algebras in the sense of \cite[III Proposition 2.1.2.3]{Ill1}. In particular, $\calE^+$ depends on the choice of $\widetilde \frakX$. When $\frakX$ comes from a smooth formal scheme $\frakX_0$ over $\calO_k$ and $\widetilde \frakX $ is the base-change of $\frakX_0$ along $\calO_k\to A_2$, the $\calE^+$ coincides with the usual Faltings' extension after inverting $p$. So we call $\calE^+$ the {\bf integral Faltings's extension} (with respect to the lifting $\widetilde \frakX$).
\end{rmk}
 It remains to prove the claim (\ref{claim on third term}). 
\begin{lem}
  With notations as above, we have 
  \[
  \widehat \rL_{\OXp/\calO_{\frakX}}\simeq \OXp\otimes_{\calO_{\frakX}}\widehat \Omega^1_{\frakX}.
  \]
\end{lem}
\begin{proof}
  
  Since the problem is local on $X_{\proet}$, by the proof of \cite[Corollary 4.7]{Sch2}, we may assume $\frakX = \Spf(R)$ is small and are reduced to showing for any perfectoid affinoid space $U = \Spa(S,S^+)\in X_{\proet}/X_{\infty}$, 
  \begin{equation}\label{equ-third term}
    \widehat \rL_{S^+/R^+}\simeq S^+\otimes_{R^+}\widehat \Omega^1_{R^+}.
  \end{equation}
  Since both $S^+$ and $\Rinfp$ are perfectoid rings, by \cite[Lemma 3.14]{BMS1}, we have a quasi-isomorphism
  \[
    \widehat \rL_{\Rinfp/R^+}\widehat \otimes_{\Rinfp}S^+\to \widehat \rL_{S^+/R^+}.
  \]
  Combining this with Lemma \ref{third term}, we get (\ref{equ-third term}) as desired. 
\end{proof}
\subsection{Faltings' extension as obstruction class}
  In this subsection, we shall give another description of the integral Faltings' extension from the perspective of deformation theory.
  To make notations clear, in this subsection, for a sheaf $S$ of $A_2$-algebras, we always identify $\xi_k A_2$ with $\frac{1}{\rho_k}S(1)$. 
  Before moving on, we recall some basic results due to Illusie. Although their statements are given in terms of rings, all results still hold for ring topoi.
  
  Let $A$ be a ring with an ideal $I\triangleleft A$ satisfying $I^2=0$. Put $\overline A = A/I$ and fix a flat $\overline A$-algebra $\overline B$. A natural question is whether there exists a flat $A$-algebra $B$ whose reduction modulo $I$ is $\overline B$.

\begin{thm}[\emph{\cite[III Proposition 2.1.2.3]{Ill1}}]\label{Deformation I}
  There exists an obstruction class $\cl\in \Ext^2(\rL_{\overline B/\overline A},\overline B\otimes_{\overline A}I)$ such that $\overline B$ lifts to some flat $A$-algebra $B$ if and only if $\cl=0$. In this case, the set of isomorphism classes of such deformations forms a torsor under $\Ext^1(\rL_{\overline B/\overline A},\overline B\otimes_{\overline A}I)$ and the group of automorphisms of a fixed deformation is $\Hom(\rL_{\overline B/\overline A},\overline B\otimes_{\overline A}I)$.
\end{thm}

 If $B$ and $C$ are flat $A$-algebras with reductions $\overline B$ and $\overline C$ respectively and if $\overline f: \overline B\rightarrow\overline C$ is a morphism of $\overline A$-algebras, then one can ask whether there exists an deformation $f: B\rightarrow C$ of $\overline f$ along $A\to \overline A$.

\begin{thm}[\emph{\cite[III Proposition 2.2.2]{Ill1}}]\label{Deformation II}
  There is an obstruction class $\cl\in \Ext^1(\rL_{\overline B/\overline A},\overline C\otimes_{\overline A}I)$ such that $\overline f$ lifts to a morphism $f: B\rightarrow C$ if and only if $\cl=0$. In this case, the set of all lifts forms a torsor under $\Hom(\rL_{\overline B/\overline A},\overline C\otimes_{\overline A}I)$.
\end{thm}
We only focus on the case where $(A,I) = (A_2,(\xi))$. Let $\frakX$ be a smooth formal scheme over $\calO_{\Cp}$ and denote 
\[
  {\rm ob}(\frakX)\in \Ext^2(\widehat \rL_{\calO_{\frakX}/\calO_{\Cp}},\frac{1}{\rho_k}\calO_{\frakX}(1))
\]
the obstruction class to lift $\frakX$ to a flat $A_2$-scheme (e.g. \cite[III Th\'eor\`eme 2.1.7]{Ill1}). Consider the exact triangle
\[
  \rL_{\calO_{\Cp}/A_2}\widehat \otimes^L_{\calO_{\Cp}}\calO_{\frakX}\to\widehat \rL_{\calO_{\frakX}/A_2}\to \widehat \rL_{\calO_{\frakX}/\calO_{\Cp}}
\]
and the induced long exact sequence
\begin{equation*}
\begin{split}\cdots&\to\Ext^1(\widehat \rL_{\calO_{\frakX}/A_2},\frac{1}{\rho_k}\calO_{\frakX}(1))\to\Ext^1( \rL_{\calO_{\Cp}/A_2}\widehat \otimes^L_{\calO_{\Cp}}\calO_{\frakX},\frac{1}{\rho_k}\calO_{\frakX}(1))\\
&\xrightarrow{\partial}\Ext^2(\widehat \rL_{\calO_{\frakX}/\calO_{\Cp}},\frac{1}{\rho_k}\calO_{\frakX}(1))\to\cdots.
 \end{split}
\end{equation*}
 The ${\rm ob(\frakX)}$ is the image of identity morphism of $\frac{1}{\rho_k}\calO_{\frakX}(1)$ under $\partial$ via the canonical isomorphism 
 \[
   \Ext^1(\rL_{\calO_{\Cp}/A_2}\widehat \otimes^L_{\calO_{\Cp}}\calO_{\frakX},\frac{1}{\rho_k}\calO_{\frakX}(1))\cong\Hom(\frac{1}{\rho_k}\calO_{\frakX}(1),\frac{1}{\rho_k}\calO_{\frakX}(1)).
 \]
 If moreover, $\frakX$ is liftable and $\widetilde \frakX$ is such a lifting, then ${\rm ob}(\frakX)=0$ and $\widetilde \frakX$ defines a class 
 \[[\widetilde \frakX]\in \Ext^1(\widehat \rL_{\calO_{\frakX}/A_2},\frac{1}{\rho_k}\calO_{\frakX}(1))\] which goes to the identity map of $\frac{1}{\rho_k}\calO_{\frakX}(1)$. Indeed, $[\widetilde \frakX]$ is the image of the identity map of $\frac{1}{\rho_k}\calO_{\frakX}(1)$ via the morphism
 \[
   \Ext^1(\widehat \rL_{\calO_{\frakX}/\calO_{\widetilde \frakX}},\frac{1}{\rho_k}\calO_{\frakX}(1))\to\Ext^1(\widehat \rL_{\calO_{\frakX}/A_2},\frac{1}{\rho_k}\calO_{\frakX}(1)).
 \]
 
 We also consider the similar deformation problem for $\OXp$. Since $\OXp$ is locally perfectoid, thanks to \cite[Lemma 3.14]{BMS1}, $\widehat \rL_{\OXp/\calO_{\Cp}} = 0$ and hence we get a quasi-isomorphism
 \[
   \rL_{\calO_{\Cp}/A_2}\widehat \otimes^L_{\calO_{\Cp}}\OXp\simeq \widehat \rL_{\OXp/A_2}.
 \]
 In particular, we have an isomorphism 
 \[\Ext^1(\widehat \rL_{\OXp/A_2},\frac{1}{\rho_k}\OXp(1))\cong\Hom(\frac{1}{\rho_k}\OXp(1),\frac{1}{\rho_k}\OXp(1)).\]
 Therefore, $\OXp$ admits a canonical lifting, which turns out to be $\rA_2(\widehat \calO_X^{+})$ and there is a unique class
 \[[X]\in \Ext^1(\widehat \rL_{\OXp/A_2},\frac{1}{\rho_k}\OXp(1))\] 
 corresponding to the identity map of $\frac{1}{\rho_k}\OXp(1)$.
 
 Regard $[\widetilde \frakX]$ and $[X]$ as classes in $\Ext^1(\widehat \rL_{\calO_{\frakX}/A_2},\frac{1}{\rho_k}\OXp(1))$ via the canonical morphisms induced by
$
 \frac{1}{\rho_k}\calO_{\frakX}(1)\to \frac{1}{\rho_k}\OXp(1)
$
 and 
$
   \widehat \rL_{\calO_{\frakX}/A_2} \to \widehat \rL_{\OXp/A_2},
$
 respectively. Then as shown in \cite[III Proposition 2.2.4]{Ill1}, the difference 
 \[
 \cl(\calE^+): = [\widetilde \frakX]-[X]
 \]
 belongs to
 \[
  \Ext^1(\widehat \rL_{\calO_{\frakX}/\calO_{\Cp}},\frac{1}{\rho_k}\OXp(1))\cong  \Ext^1(\widehat \Omega^1_{\calO_{\frakX}/\calO_{\Cp}}\otimes_{\calO_{\frakX}}\OXp,\frac{1}{\rho_k}\OXp(1))
 \]
 via the injection 
 \[
   \Ext^1(\widehat \rL_{\calO_{\frakX}/\calO_{\Cp}},\frac{1}{\rho_k}\OXp(1))\to \Ext^1(\widehat \rL_{\calO_{\frakX}/A_2},\frac{1}{\rho_k}\OXp(1))
 \]
 and $\cl(\calE^+)$ is the obstruction answering whether there is an $A_2$-morphism from $\calO_{\widetilde \frakX}$ to $\rA_2(\widehat \calO_X^{+})$ which lifts the $\calO_{\Cp}$-morphism $\calO_{\frakX}\to\OXp$ as described in Theorem \ref{Deformation II}. 
 
 Recall we have another obstruction class $[\calE^+]$ described in Remark \ref{Faltings obstruction}. We claim that it coincides with the class $\cl(\calE^+)$ constructed above.
 \begin{prop}\label{Compare obstruction}
   $\cl(\calE^+) = [\calE^+]$.
 \end{prop}
 \begin{proof}
   Note that we have a commutative diagram of morphisms of cotangent complexes
  \begin{equation}\label{Big diagram}
  \xymatrix@C=0.4cm{
   & \rL_{\calO_{\widetilde \frakX}/\rA_2}\widehat \otimes^{L}_{\calO_{\widetilde \frakX}}\widehat \calO^+_{X}\ar@{=}[d]\ar[r]&\rL_{\calO_{\frakX}/\rA_2}\widehat \otimes^{L}_{\calO_{\frakX}}\widehat \calO^+_{X}\ar[d]^{\beta}\ar[r]^{\alpha}&\rL_{\calO_{\frakX}/\calO_{\widetilde \frakX}}\widehat \otimes_{\calO_{\frakX}}^{L}\widehat \calO^+_{X}\ar[r]^{\qquad\qquad+1}\ar[d]& \\
   & \rL_{\calO_{\widetilde \frakX}/\rA_2}\widehat \otimes^{L}_{\calO_{\widetilde \frakX}}\widehat \calO^+_{X}\ar[r]\ar[dr]^{\simeq}_{-1}&\widehat \rL_{\widehat \calO^+_{X}/\rA_2}\ar[d]\ar[r]& \widehat \rL_{\widehat \calO^+_{X}/\calO_{\widetilde \frakX}}\ar[d]\ar[r]^{\qquad+1}&\\
   &&\widehat \rL_{\widehat \calO^+_{X}/\calO_{\frakX}}\ar[d]^{+1}\ar@{=}[r]&\widehat \rL_{\widehat \calO^+_{X}/\calO_{\frakX}}\ar[d]^{+1}\\
   &&&&
  }
  \end{equation}
  where the notations ``$+1$'' and ``$-1$'' denote the shifts of dimensions.
  
  Consider the resulting diagram from applying $\rR\Hom(-,\frac{1}{\rho_k}\OXp(1))$ to (\ref{Big diagram}).
  Denote the identity map of $\frac{1}{\rho_k}\OXp(1)$ by $\id$. By construction, $[\calE^+]$ is the image of $\id$ via the connecting map induced by the triangle
  \[\rL_{\calO_{\frakX}/\calO_{\widetilde \frakX}}\widehat \otimes_{\calO_{\frakX}}^{L}\OXp\to\widehat \rL_{\OXp/\calO_{\widetilde \frakX}}\to\widehat \rL_{\OXp/\calO_{\frakX}}.\]
  By the commutativity of diagram (\ref{Big diagram}), $[\calE^+]$ is also the image of $\alpha^*(\id)$ via the connecting map $\partial$ induced by the triangle
  \[\rL_{\calO_{\frakX}/A_2}\widehat \otimes_{\calO_{\frakX}}^{L}\OXp\to\widehat \rL_{\OXp/A_2}\to\widehat \rL_{\OXp/\calO_{\frakX}}.\]
  
  On the other hand, by the constructions of $[\widetilde \frakX]$ and $[X]$, as elements in 
  \[
  \Ext^1(\rL_{\calO_{\frakX}/\rA_2}\widehat \otimes^{L}_{\calO_{\frakX}}\OXp,\frac{1}{\rho_k}\OXp(1)),
  \]
  we have $[\widetilde \frakX] = \alpha^*(\id)$ and $[X] = \beta^*(\id)$; here, for the second equality, we identify 
  \[
    \Hom(\frac{1}{\rho_k}\OXp(1),\frac{1}{\rho_k}\OXp(1)) = \Ext^1(\widehat \rL_{\calO_{\Cp}/A_2}\widehat \otimes_{\calO_{\Cp}}^L\OXp,\frac{1}{\rho_k}\OXp(1))
  \]
  with $\Ext^1(\widehat \rL_{\OXp/A_2}\widehat \otimes_{\calO_{\Cp}}^L\OXp,\frac{1}{\rho_k}\OXp(1))$. So we have 
  \[\cl(\calE^+) = \alpha^*(\id)-\beta^*(\id)\in \Ext^1(\rL_{\calO_{\frakX}/A_2}\widehat \otimes^L_{\calO_{\frakX}}\OXp,\frac{1}{\rho_k}\OXp(1)).\]
  However, the diagram
  \[
  \xymatrix@C=0.45cm{
    \widehat \rL_{\rA_2(\widehat \calO_X^{+})/\calO_{\widetilde \frakX}}\ar[d]\ar[rr]^{+1}&&\rL_{\calO_{\widetilde \frakX}/\rA_2}\widehat \otimes^{L}_{\calO_{\widetilde \frakX}}\OXp\ar[d]\\
    \widehat \rL_{\OXp/\calO_{\frakX}}\ar[r]^{+1\quad}&\rL_{\calO_{\frakX}/\rA_2}\widehat \otimes^{L}_{\calO_{\widetilde \frakX}}\OXp\ar[r]&\rL_{\calO_{\frakX}/\calO_{\Cp}}\widehat \otimes^L_{\calO_{\frakX}}\OXp
    }
  \]
  induces a commutative diagram
  \[
    \xymatrix@C=0.45cm{
      \Ext^1(\rL_{\calO_{\frakX}/\calO_{\Cp}}\widehat \otimes^L_{\calO_{\frakX}}\OXp,\frac{1}{\rho_k}\OXp(1))\ar[r]^{\subset}\ar[dr]^{\cong}&\Ext^1(\rL_{\calO_{\frakX}/A_2}\widehat \otimes^L_{\calO_{\frakX}}\OXp,\frac{1}{\rho_k}\OXp(1))\ar[d]^{\partial}\\
      &\Ext^2(\widehat \rL_{\OXp/\calO_{\frakX}},\frac{1}{\rho_k}\OXp(1)).
    }
  \]
  In particular, as elements in $\Ext^1(\rL_{\calO_{\frakX}/\calO_{\Cp}}\widehat \otimes^L_{\calO_{\frakX}}\OXp,\frac{1}{\rho_k}\OXp(1))$, we have
  \[
  \cl(\calE^+) = \partial(\alpha^*(\id)-\beta^*(\id)) = \partial(\alpha^*(\id)) = [\calE^+]
  \]
  as desired. So we are done.
 \end{proof}
\begin{rmk}
  When $\frakX$ is small affine and comes from a formal scheme over $\calO_k$, the obstruction class $\cl(\calE^+)$ was considered as \emph{Higgs-Tate extension associated to $\widetilde \frakX$} in \cite[I. 4.3]{AGT}.
\end{rmk}
 \begin{exam}\label{Exam-Faltings obstruction}
   Let $R^+ = \calO_{\Cp}\za\underline T^{\pm 1}\ya$ and $\widetilde R^+ = A_2\za\underline T^{\pm 1}\ya$ for simplicity. Consider the $A_2$-morphism $\widetilde \psi:\widetilde R^+\to\rA_2(\widehat R_{\infty}^{+})$, which sends $T_i$ to $[T_i^{\flat}]$ for all $i$, where $T_i^{\flat}\in \widehat R_{\infty}^{\flat+}$ is determined by the compatible sequence $(T_i^{\frac{1}{p^n}})_{n\geq 0}$. Then $\widetilde \psi$ is a lifting of the inclusion $R^+\to\Rinfp$, but is not $\Gamma$-equivariant. For any $\gamma\in\Gamma$, $\gamma\circ\widetilde \psi$ is another lifting. By Theorem \ref{Deformation II}, their difference
   $c(\gamma): = \gamma\circ\widetilde \psi-\widetilde \psi$
   belongs to $\Hom_{R^+}(\widehat \Omega^1_{R^+},\frac{1}{\rho_k}\Rinfp(1))$. One can check that for any $1\leq i,j\leq 1$, 
   \[c(\gamma_i)(\dlog(T_j))=\frac{(\gamma_i-1)([T_j^{\flat}])}{T_j} = \delta_{ij}([\epsilon]-1) = \delta_{ij}t,\]
   where the last equality follows from the fact that $[\epsilon]-1-t\in t^2\BdRp$. By construction, the cocycle $c:\Gamma\to\Hom_{R^+}(\widehat \Omega^1_{R^+},\frac{1}{\rho_k}\Rinfp(1))$ is exactly the class $\cl(\calE^+)$. Comparing this with Proposition \ref{good basis}, we deduce that $\cl(\calE^+)=[\calE^+]$ in this case.
 \end{exam}
  As an application of Proposition \ref{Compare obstruction}, we study the behavior of integral Faltings' extension under the pull-back. 
  
  Let $f:\frakX\to\frakY$ be a formally smooth morphism of liftable smooth formal schemes. Fix liftings $\widetilde \frakX$ and $\widetilde \frakY$ of $\frakX$ and $\frakY$, respectively. Denote by $\calE_X^+$ and $\calE_Y^+$ the corresponding integral Faltings' extensions. Then the pull-back of $\calE_X^+$ along the injection 
  \[
  f^*\widehat \Omega^1_{\frakY}\otimes_{\calO_{\frakX}}\OXp \to \widehat \Omega^1_{\frakX}\otimes_{\calO_{\frakX}}\OXp
  \]
  defines an extension $\calE_1^+$ of 
  $\widehat \Omega^1_{\frakY}\otimes_{\calO_{\frakY}}\OXp\footnote{Here, the tensor product $\widehat \Omega^1_{\frakY}\otimes_{\calO_{\frakY}}\OXp$ should be understood as $f^{-1}\widehat \Omega^1_{\frakY}\otimes_{f^{-1}\calO_{\frakY}}\OXp$. The same thing also applies to sheaves like $\calO_X^+\otimes_{\widehat \calO_Y^+}\calE_Y^+$, $\widehat \calO_X^+\otimes_{\widehat \calO_Y^+}\calO\bC_{Y,\rho}^+$, $\widehat \calO_X^+\otimes_{\widehat \calO_Y^+}\calO\widehat \bC_{Y,\rho}^+$, $\widehat \calO_X^+\otimes_{\widehat \calO_Y^+}\calO\bC^{\dagger,+}_{Y,\rho}$, etc..}\cong f^*\widehat \Omega^1_{\frakY}\otimes_{\calO_{\frakX}}\OXp$
  by $\frac{1}{\rho_k}\OXp(1)$. We denote its extension class by 
  \[\cl_1\in \Ext^1(\widehat \Omega^1_{\frakY}\otimes_{\calO_{\frakY}}\OXp,\frac{1}{\rho_k}\OXp(1)).\]
  On the other hand, the tensor product $\calE_2^+ = \calE_Y^+\otimes_{\widehat \calO^+_Y}\OXp$ induced by applying $-\otimes_{\widehat \calO_Y^+}\OXp$ to 
  \[0\to\frac{1}{\rho_k}\widehat \calO_Y^+(1)\to\calE_Y^+\to\widehat \calO_Y^+\otimes_{\calO_{\frakY}}\widehat \Omega^1_{\frakY}\to 0\]
  is also an extension of 
  $\widehat \Omega^1_{\frakY}\otimes_{\calO_{\frakY}}\OXp$
  by $\frac{1}{\rho_k}\OXp(1)$ and we denote the associated extension class by 
  \[\cl_2\in \Ext^1(\widehat \Omega^1_{\frakY}\otimes_{\calO_{\frakY}}\OXp,\frac{1}{\rho_k}\OXp(1)).\]
  Then it is natural to ask whether $\calE_1^+\cong\calE_2^+$ (equivalently, $\cl_1 = \cl_2$).
  \begin{prop}\label{compatible under pull-back}
    Keep notations as above. If $f:\frakX\to\frakY$ lifts to an $A_2$-morphism $\widetilde f:\widetilde \frakX\to \widetilde \frakY$, then $\cl_1 = \cl_2$.
  \end{prop}
  We are going to prove this proposition in the rest of this subsection.
  
  By Theorem \ref{Deformation II}, there exists an obstruction class
  \[
  \cl(f)\in \Ext^1(\widehat \rL_{\calO_{\frakY}/\calO_{\Cp}},\frac{1}{\rho_k}\calO_{\frakX}(1))
  \]
  to lift $f$ along the surjection $A_2\to\calO_{\Cp}$. Before moving on, let us recall the definition of $\cl(f)$.
  
  Let $[\widetilde \frakX]$ and $[\widetilde \frakY]$ be classes similarly defined as before and regard them as elements in $\Ext^1(\widehat \rL_{\calO_{\frakY}/A_2},\frac{1}{\rho_k}\calO_{\frakX}(1))$ via the obvious morphisms. Then similar to the construction of $\cl(\calE^+)$, one can check that 
  \[
  \cl(f) = [\widetilde \frakX]-[\widetilde \frakY]
  \]
  via the injection 
  \[
    \Ext^1(\widehat \rL_{\calO_{\frakY}/\calO_{\Cp}},\frac{1}{\rho_k}\calO_{\frakX}(1)) \to \Ext^1(\widehat \rL_{\calO_{\frakY}/A_2},\frac{1}{\rho_k}\calO_{\frakX}(1)).
  \]
  For simplicity, we still denote by $\cl(f)$ its image in 
  \[ \Ext^1(\widehat \rL_{\calO_{\frakY}/\calO_{\Cp}},\frac{1}{\rho_k}\OXp(1))\cong \Ext^1(\widehat \Omega^1_{\frakY}\otimes_{\calO_{\frakY}}\OXp,\frac{1}{\rho_k}\OXp(1)) \]
  via the natural map $\frac{1}{\rho_k}\calO_{\frakX}(1)\to \frac{1}{\rho_k}\OXp(1)$. Then the following proposition is true.
  \begin{prop}\label{twist compatibility}
    $\cl(f) = \cl_1-\cl_2$.
  \end{prop}
  \begin{proof}
    By the constructions of $\calE_1^+$ and $\calE_2^+$, we see that $\cl_1$ is the image of $\cl(\calE_X^+)$ via the morphism 
    \[
    \Ext^1(\widehat \Omega^1_{\frakX},\frac{1}{\rho_k}\OXp(1))\to \Ext^1(\widehat \Omega^1_{\frakY}\otimes_{\calO_{\frakY}}\calO_{\frakX},\frac{1}{\rho_k}\OXp(1))
    \]
    induced by 
    \[
    \rL_{\calO_{\frakY}/\calO_{\Cp}}\widehat \otimes^L_{\calO_{\frakY}}\calO_{\frakX} \to \widehat \rL_{\calO_{\frakX}/\calO_{\Cp}},
    \]
    and that $\cl_2$ is the image of $\cl(\calE_Y^+)$ via the morphism
    \[
    \Ext^1(\widehat \Omega^1_{\frakY}\otimes_{\calO_{\frakY}}\widehat \calO_Y^+,\frac{1}{\rho_k}\widehat \calO_Y^+(1))\to \Ext^1(\widehat \Omega^1_{\frakY}\otimes_{\calO_{\frakY}}\OXp,\frac{1}{\rho_k}\OXp(1))
    \]
    induced by the inclusion $\frac{1}{\rho_k}\widehat \calO_Y^+(1) \to\frac{1}{\rho_k}\OXp(1)$. 
    
    Now by Proposition \ref{Compare obstruction}, we have 
    \[\cl_1-\cl_2 = \cl(\calE_X^+)-\cl(\calE_Y^+) = ([\widetilde \frakX]-[\widetilde \frakY]) - ([X]-[Y]).\]
    However, the inclusion $\widehat \calO_Y^+\to \OXp$ admits a canonical $A_2$-lifting, namely $\rA_2(\widehat \calO_Y^{+})\to \rA_2(\widehat \calO_X^{+})$. So we deduce that $[X]-[Y] = 0$, which completes the proof.
  \end{proof}
  Now, Proposition \ref{compatible under pull-back} is a special case of Proposition \ref{twist compatibility}.
  
  \begin{cor}\label{relative Faltings extension}
    Assume $f:\frakX\to\frakY$ admits a lifting along $A_2\to\calO_{\Cp}$, then there is an exact sequence of sheaves of $\OXp$-modules
    \begin{equation}\label{relative exact sequence}
      0\to\OXp\otimes_{\widehat \calO_Y^+}\calE_Y^+\to\calE_X^+\to \OXp\otimes_{\calO_{\frakX}}\widehat \Omega^1_{\frakX/\frakY}\to 0,
    \end{equation}
    where $\widehat \Omega^1_{\frakX/\frakY}$ is the module of relative differentials.
  \end{cor}
  \begin{proof} 
  This follows from the Proposition \ref{compatible under pull-back} combined with the definitions of $\calE_1^+$ and $\calE_2^+$.
  \end{proof}
\subsection{Period sheaves}
  Now, we define the desired period sheaf $\calO\bC^{\dag}$ as mentioned in Introduction. The construction generalizes the previous work of Hyodo \cite{Hy}. 
  
  Let $\frakX = \Spf(R^+)$ be a small smooth formal scheme and $\widetilde \frakX = \Spf(\widetilde R^+)$ be a fixed $A_2$-lifting. Let $E^+$ be the integral Faltings' extension introduced in Proposition \ref{local Faltings extension}. Define $E_{\rho_k}^+ = \rho_k E^+(-1)$. Then it fits into the following exact sequence
  \[0\to\Rinfp\to E^+_{\rho_k}\to\rho_k\Rinfp\otimes_{R^+}\widehat \Omega^1_{R^+}(-1)\to 0.\]
  For any $\rho\in \rho_k\calO_{\Cp}$, denote by $E_{\rho}^+$ the pull-back of $E_{\rho_k}^+$ along the inclusion 
  \[
  \rho\Rinfp\otimes_{R^+}\widehat \Omega^1_{R^+}(-1)\to\rho_k\Rinfp\otimes_{R^+}\widehat \Omega^1_{R^+}(-1),
  \]
  then it fits into the following $\Gamma$-equivariant exact sequence
  \begin{equation}\label{value exact sequence}
  0\to\Rinfp\to E^+_{\rho}\to\rho\Rinfp\otimes_{R^+}\widehat \Omega^1_{R^+}(-1)\to 0.
  \end{equation}
  By Proposition \ref{good basis}, $E_{\rho}^+$ admits an $\Rinfp$-basis $1, \frac{\rho x_1}{t}, \dots, \frac{\rho x_d}{t}$. Let $E = E^+_{\rho}[\frac{1}{p}]$, which fits into the induced exact sequence
  \[
    0\to\Rinf\to E\to\Rinf\otimes_{R^+}\widehat \Omega^1_{R^+}(-1)\to 0.
  \]
  Then it is independent of the choice of $\rho$ and has $E_{\rho}^+$ as a sub-$\Rinfp$-module. Moreover, it admits an $\Rinf$-basis
  \[1, y_1 = \frac{x_1}{t}, \dots, y_d = \frac{x_d}{t}
  \]
  such that $\gamma_i(y_j)=y_j+\delta_{ij}$ for any $1\leq i,j\leq d$. Define $S_{\infty} = \varinjlim_n\Sym^n_{\Rinf}E$. Then by similar arguments used in \cite[Section I]{Hy}, we have the following result.
  \begin{prop}\label{Hyodo's result} 
    There exists a canonical Higgs field 
    \[
    \Theta:S_{\infty}\to S_{\infty}\otimes_{\Rinf}\widehat \Omega_{R^+}^1(-1)
    \]
    on $S_{\infty}$ such that the induced Higgs complex is a resolution of $\Rinf$. The $\Theta$ is induced by taking alternative sum along the projection $E\to \Rinf\otimes_{R^+}\widehat \Omega^1_{R^+}(-1)$ and if we denote by $Y_i$ the image of $y_i$ in $S_{\infty}$, then there is a $\Gamma$-equivariant isomorphism 
    \[
    \iota: S_{\infty}\xrightarrow{\cong}\Rinf[Y_1,\dots,Y_d]
    \]
    such that $\Theta = \sum_{i=1}^d\frac{\partial}{\partial Y_i}\otimes\frac{\dlog T_i}{t}$ via this isomorphism, where $\Rinf[Y_1,\dots,Y_d]$ is the polynomial ring on free variables $Y_i$'s over $\Rinf$.
  \end{prop}
   Since we have $\Rinfp$-lattices $E_{\rho}^+$'s of $E$, inspired by Proposition \ref{Hyodo's result}, we make the following definition.
 \begin{dfn}\label{local period sheaf} 
 For any $\rho\in\rho_k\calO_{\Cp}$, define
 \begin{enumerate}
     \item $S_{\infty,\rho}^+ = \varinjlim_n\Sym^n_{\Rinfp}E_{\rho}^+$;
     \item $\widehat S_{\infty,\rho}^+ = \varprojlim_n S_{\infty,\rho}^+/p^n$;
     \item $S_{\infty}^{\dagger,+} = \varinjlim_{\nu_p(\rho)>\nu_p(\rho_k)}\widehat S_{\infty,\rho}^+$ and $S_{\infty}^{\dagger} = S_{\infty}^{\dagger,+}[\frac{1}{p}]$.
 \end{enumerate}
 \end{dfn}
 For any $\rho_1,\rho_2\in\rho_k\calO_{\Cp}$ satisfying $\nu_p(\rho_1)\geq \nu_p(\rho_2)$, we have $E_{\rho_1}^+\subset E_{\rho_2}^+\subset E$. So Proposition \ref{Hyodo's result} implies that $S_{\infty,\rho_1}^+\subset S_{\infty,\rho_2}^+\subset S_{\infty}$. Moreover, the restriction of $\Theta$ to $S_{\infty,\rho}^+$ (for $\rho\in\rho_k\calO_{\Cp}$) induces a Higgs field on $S_{\infty,\rho}^+$, which is identified with $\Rinfp[\rho Y_1,\dots,\rho Y_d]$ via the canonical isomorphism $\iota$. In this case, we still have $\Theta=\sum_{i=1}^d\frac{\partial}{\partial Y_i}\otimes\frac{\dlog T_i}{t}$. Since $\Theta$ is continuous, it extends to $\widehat S_{\infty,\rho}^+$ and thus we have the following corollary.
 
 \begin{cor}\label{complete Hyodo}
   For any $\rho\in\rho_k\calO_{\Cp}$, there exists a canonical Higgs field 
    \[
    \Theta:\widehat S_{\infty,\rho}^+\to \widehat S_{\infty,\rho}^+\otimes_{\Rinfp}\widehat \Omega_{R^+}^1(-1)
    \]
    on $\widehat S_{\infty,\rho}^+$. Moreover, there is a $\Gamma$-equivariant isomorphism 
    \[
    \iota: \widehat S_{\infty,\rho}^+\xrightarrow{\cong}\Rinfp\za \rho Y_1, \dots, \rho Y_d\ya
    \]
    such that $\Theta = \sum_{i=1}^d\frac{\partial}{\partial Y_i}\otimes\frac{\dlog T_i}{t}$ via this isomorphism, where $\Rinfp\za \rho Y_1,\dots,\rho Y_d\ya$ is the $p$-adic completion of $\Rinfp[\rho Y_1,\dots ,\rho Y_d]$.
 \end{cor}
 After taking inductive limit among $\{\rho\in\rho_k\calO_{\Cp}|\nu_p(\rho)>\nu_p(\rho_k)\}$, we get the following corollary. 
 \begin{cor}\label{overconvergent resolution}
   There exists a canonical Higgs field 
    \[
    \Theta:S_{\infty}^{\dag,+}\to S_{\infty}^{\dag,+}\otimes_{\Rinfp}\widehat \Omega_{R^+}^1(-1)
    \]
    on $S_{\infty}^{\dag,+}$. Moreover, there is a $\Gamma$-equivariant isomorphism 
    \[
    \iota: S_{\infty}^{\dag,+}\xrightarrow{\cong}\varinjlim_{\nu_p(\rho)>\nu_p(\rho_k)}\Rinfp\za \rho Y_1,\dots,\rho Y_d\ya
    \]
    such that $\Theta = \sum_{i=1}^d\frac{\partial}{\partial Y_i}\otimes\frac{\dlog T_i}{t}$ via this isomorphism. After inverting $p$, the induced Higgs complex ${\rm HIG}(S_{\infty}^{\dagger},\Theta)$
    \begin{equation}\label{local resolution}
    S_{\infty}^{\dagger}\xrightarrow{\Theta}S_{\infty}^{\dagger}\otimes_{R^+}\widehat \Omega^1_{R^+}(-1)\xrightarrow{\Theta}S_{\infty}^{\dagger}\otimes_{R^+}\widehat \Omega^2_{R^+}(-2)\to\cdots
    \end{equation}
    is a resolution of $\Rinf$.
 \end{cor}
 \begin{proof}
   It remains to prove the Higgs complex ${\rm HIG}(S_{\infty}^{\dagger},\Theta)$ is a resolution of $\Rinf$.
   For any $\rho\in\rho_k\calO_{\Cp}$, consider the Higgs complexes 
   \[
   {\rm HIG}(\widehat S_{\infty,\rho}^{+},\Theta): 
   \widehat S_{\infty,\rho}^{+}\xrightarrow{\Theta}\widehat S_{\infty,\rho}^{+}\otimes_{R^+}\widehat \Omega^1_{R^+}(-1)\xrightarrow{\Theta}\widehat S_{\infty,\rho}^{+}\otimes_{R^+}\widehat \Omega^2_{R^+}(-2)\to\cdots
   \]
   and
   \[
     {\rm HIG}(S_{\infty}^{\dagger,+},\Theta): 
   S_{\infty}^{\dagger,+}\xrightarrow{\Theta}\widehat S_{\infty}^{\dagger,+}\otimes_{R^+}\widehat \Omega^1_{R^+}(-1)\xrightarrow{\Theta}S_{\infty}^{\dagger,+}\otimes_{R^+}\widehat \Omega^2_{R^+}(-2)\to\cdots.
   \]
   Then we have
   \[{\rm HIG}(S_{\infty}^{\dagger},\Theta)={\rm HIG}(S_{\infty}^{\dagger,+},\Theta)[\frac{1}{p}]=\varinjlim_{\nu_p(\rho)>\nu_p(\rho_k)}{\rm HIG}(\widehat S_{\infty,\rho}^{+},\Theta)[\frac{1}{p}]. \]
   By Corollary \ref{complete Hyodo}, ${\rm HIG}(\widehat S_{\infty,\rho}^{+},\Theta)$ is computed by the Koszul complex
   \[ \rK(\Rinfp\za\rho Y_1,\dots,\rho Y_d\ya; \frac{\partial}{\partial Y_1}, \dots,\frac{\partial}{\partial Y_d})\simeq \rK(\Rinfp\za\rho Y_1\ya;\frac{\partial}{\partial Y_1})\widehat \otimes_{\Rinfp}^L\dots\widehat \otimes_{\Rinfp}^L\rK(\Rinfp\za\rho Y_d\ya;\frac{\partial}{\partial Y_d}),\]
   via the canonical isomorphism $ \iota $. Note that for any $j$,
\[
  \rH^i(\rK(\Rinfp\za\rho Y_j\ya; \frac{\partial}{\partial Y_j})) = 
  \left\{\begin{array}{rcl}
   \Rinfp,  &i=0\\
   \Rinfp\za \Lambda_{j,\rho}\ya/\Rinfp\za\Lambda_{j,\rho},I,+\ya,  &i=1\\
   0,& i\geq 2
  \end{array}\right.
\]
is derived $p$-complete by Proposition \ref{derived vs classical}, where $\Rinfp\za \Lambda_{j,\rho}\ya$ and $\Rinfp\za\Lambda_{j,\rho},I,+\ya$ are defined as in Definition \ref{p-complete module} for $\Lambda_{j,\rho} = \{\rho^nY_j^n\}_{n\geq 0}$ and $I = \{\nu_p(n+1)\}_{n\geq 0}$. We deduce that for any $i\geq 0$,
\[
\rH^i( \rK(\Rinfp\za\rho Y_1,\dots,\rho Y_d\ya; \frac{\partial}{\partial Y_1}, \dots,\frac{\partial}{\partial Y_d})) = \wedge^i_{\Rinfp}(\oplus_{j=1}^d\Rinfp\za \Lambda_{j,\rho}\ya/\Rinfp\za\Lambda_{j,\rho},I,+\ya).
\]
In particular, we get 
\[\rH^0({\rm HIG}(S_{\infty}^{\dagger,+},\Theta))=\varinjlim_{\nu_p(\rho)>\nu_p(\rho_k)}\rH^0({\rm HIG}(\widehat S_{\infty,\rho}^{+},\Theta)) = \Rinfp.\]

It remains to show that for any $i\geq 1$, 
\[\varinjlim_{\nu_p(\rho)>\nu_p(\rho_k)}\rH^i({\rm HIG}(\widehat S_{\infty,\rho}^{+},\Theta))\cong\varinjlim_{\nu_p(\rho)>\nu_p(\rho_k)} \wedge^i_{\Rinfp}(\oplus_{j=1}^d\Rinfp\za \Lambda_{j,\rho}\ya/\Rinfp\za\Lambda_{j,\rho},I,+\ya)\]
is $p^{\infty}$-torsion. To do so, it suffices to prove that for any $\nu_p(\rho_1)>\nu_p(\rho_2)>\nu_p(\rho_k)$, there is an $N\geq 0$ such that 
\[
p^N\Rinfp\za \Lambda_{j,\rho_1}\ya\subset \Rinfp\za\Lambda_{j,\rho_2},I,+\ya.
\]
By Remark \ref{describe of completion}, we only need to find an $N$ such that the following conditions hold:
\begin{enumerate}
    \item for any $i\geq 0$, $N+i\nu_p(\rho_1)-i\nu_p(\rho_2)-\nu_p(i+1)\geq 0$;
    \item $\lim_{i\to+\infty}(N+i\nu_p(\rho_1)-i\nu_p(\rho_2)-\nu_p(i+1) )= +\infty$.
\end{enumerate}
Since $\nu_p(\rho_1)>\nu_p(\rho_2)$, such an $N$ exists. This completes the proof.
 \end{proof}
\begin{rmk}
\begin{enumerate}
  \item In the proof of Corollary \ref{overconvergent resolution}, we have seen that for any $\rho\in\rho_k\calO_{\Cp}$, the Higgs complex ${\rm HIG}(S_{\infty,\rho}^+[\frac{1}{p}],\Theta)$ is not a resolution of $\Rinf$.
  \item For any $1\leq i\leq d$, the $p^{\infty}$-torsion of $\rH^i({\rm HIG}(S_{\infty}^{\dagger,+},\Theta))$ is unbounded.
\end{enumerate}
\end{rmk}
\begin{rmk}\label{Lie algebra}
  Since for any $1\leq i,j\leq d$, $\gamma_i(Y_j)=Y_j+\delta_{ij}$, one can check that
  $\frac{\partial}{\partial Y_i} = \log\gamma_i$ on $S_{\infty}^{\dagger}$.
  So the Higgs field is $\Theta = \sum_{i=1}^d\log\gamma_i\otimes\frac{\dlog T_i}{t}$.
\end{rmk}

\begin{rmk}
  A similar local construction of $S_{\infty}^{\dagger,+}$ also appeared in \cite[I.4.7]{AGT}.
\end{rmk}
  There is a global story by using Theorem \ref{Integral Faltings extension} instead of Proposition \ref{local Faltings extension}.
 Put $\calE_{\rho_k}^+ = \rho_k\calE^+(-1)$ and for any $\rho\in \rho_k\calO_{\Cp}$, denote by $\calE_{\rho}^+$ the pull-back of $\calE_{\rho_k}^+$ along the inclusion 
\[
  \rho\OXp\otimes_{\calO_{\frakX}}\widehat \Omega^1_{\frakX}(-1)\to \rho_k\OXp\otimes_{\calO_{\frakX}}\widehat \Omega^1_{\frakX}(-1).
\]
Then it fits into the following exact sequence
\begin{equation}\label{twist exact sequence}
0\to\OXp\to\calE_{\rho}^+\to\rho\OXp\otimes_{\calO_{\frakX}}\widehat \Omega^1_{\frakX}(-1)\to 0.
\end{equation}
As an analogue of Definition \ref{local period sheaf} in the local case, we define period sheaves as follows:
 \begin{dfn}\label{global period sheaf} 
 For any $\rho\in\rho_k\calO_{\Cp}$, define
 \begin{enumerate}
     \item $\OC_{\rho}^+ = \varinjlim_n\Sym^n_{\OXp}\calE_{\rho}^+$;
     \item $\calO\widehat \bC_{\rho}^+ = \varprojlim_n \OC_{\rho}^+/p^n$;
     \item $\OC^{\dagger,+} = \varinjlim_{\nu_p(\rho)>\nu_p(\rho_k)}\calO\widehat \bC_{\rho}^+$ and $\OC^{\dagger} = \OC^{\dagger,+}[\frac{1}{p}]$.
 \end{enumerate}
 \end{dfn}
\begin{thm}\label{period sheaf}
  There is a canonical Higgs field $\Theta$ on $\OC^{\dagger,+}$ such that the induced Higgs complex ${\rm HIG}(\OC^{\dagger},\Theta)$: 
  \begin{equation}
    \OC^{\dagger}\xrightarrow{\Theta}\OC^{\dagger}\otimes_{\calO_{\frakX}}\widehat \Omega^1_{\frakX}(-1)\xrightarrow{\Theta}\OC^{\dagger}\otimes_{\calO_{\frakX}}\widehat \Omega^2_{\frakX}(-2)\to \cdots
  \end{equation}
  is a resolution of $\OX$. Moreover, when $\frakX = \Spf(R^+)$ is small affine, there is an isomorphism 
  \[\iota: \OC^{\dagger,+}_{\mid X_{\infty}}\to\varinjlim_{\nu_p(\rho)>\nu_p(\rho_k)}\OXp\za\rho Y_1, \dots, \rho Y_d\ya_{\mid X_{\infty}}\]
  such that the Higgs field $\Theta = \sum_{i=1}^d\frac{\partial}{\partial Y_i}\otimes\frac{\dlog T_i}{t}$.
\end{thm}
\begin{proof}
  Since the problem is local, we are reduced to Corollary \ref{overconvergent resolution}.
\end{proof}
 Finally, we describe the relative version of above constructions. We assume that $f:\frakX\to\frakY$ is a morphism of liftable smooth formal schemes and lifts to an $A_2$-morphism $\widetilde f:\widetilde \frakX\to\widetilde \frakY$. Then by Corollary \ref{relative Faltings extension}, for any $\rho\in\rho_k\calO_{\Cp}$, we have the following exact sequence
 \[
 0\to\OXp\otimes_{\widehat \calO_Y^+}\calE_{\rho,Y}^+\to\calE_{\rho,X}^+\to\OXp\otimes_{\calO_{\frakX}}\widehat \Omega^1_{\frakX/\frakY}(-1)\to 0.
 \]
 By constructions of period sheaves in Definition \ref{global period sheaf}, we get morphisms of sheaves $\OXp\otimes_{\widehat \calO_Y^+}\calF_Y\to\calF_X$ for $\calF\in\{\OC^+_{\rho},\calO\widehat \bC^+_{\rho}, \OC^{\dagger,+}\}$.
 Also, the natural projection $\calE_{\rho,X}^+\to\OXp\otimes_{\calO_{\frakX}}\widehat \Omega^1_{\frakX/\frakY}(-1)$ induces relative Higgs fields
 \[\Theta_{X/Y}:\calF_X\to\calF_X\otimes_{\calO_{\frakX}}\widehat \Omega^1_{\frakX/\frakY}(-1)\]
 for $\calF\in\{\OC^+_{\rho},\calO\widehat \bC^+_{\rho}, \OC^{\dagger,+}\}$.
 Using similar arguments as above, we get the following proposition.
 \begin{prop}\label{relative resolution}
   Assume that $f:\frakX\to\frakY$ is a morphism of liftable smooth formal schemes and lifts to an $A_2$-morphism $\widetilde f:\widetilde \frakX\to\widetilde \frakY$.
   The induced relative Higgs complex ${\rm HIG}(\OC_X^{\dagger},\Theta_{X/Y})$:
   \[\OC^{\dagger}_X\xrightarrow{\Theta_{X/Y}}\OC_X^{\dagger}\otimes_{\calO_{\frakX}}\widehat \Omega^1_{\frakX/\frakY}(-1)\xrightarrow{\Theta_{X/Y}}\OC_X^{\dagger}\otimes_{\calO_{\frakX}}\widehat \Omega^2_{\frakX/\frakY}(-2)\to\cdots\]
   is a resolution of $\varinjlim_{\rho,\nu_p(\rho)>\nu_p(\rho_k)}(\OXp\widehat \otimes_{\widehat \calO_Y^+}\calO\widehat \bC_{\rho,Y}^+)[\frac{1}{p}]$ and makes the following diagram
   \begin{equation}\label{Diag-commutative diagram}
       \xymatrix@C=0.45cm{
         f^*\OC_Y^{\dagger}\ar[r]^{f^*\Theta_Y}\ar[d]&f^*\OC_Y^{\dagger}\otimes_{\calO_{\frakY}}\widehat \Omega^1_{\frakY}(-1)\ar[d]\ar[r]&\cdots\\
         \OC_X^{\dagger}\ar[r]^{\Theta_X}\ar[d]^{\Theta_{X/Y}}&\OC_X^{\dagger}\otimes_{\calO_{\frakX}}\widehat \Omega^1_{\frakX}(-1)\ar[r]\ar[d]^{\Theta_{X/Y}}&\cdots\\
         \OC_X^{\dagger}\otimes_{\calO_{\frakX}}\widehat \Omega^1_{\frakX/\frakY}(-1)\ar[d]\ar[r]^{\Theta_{X/Y}}&\OC_X^{\dagger}\otimes_{\calO_{\frakX}}\widehat \Omega^2_{\frakX/\frakY}(-2)\ar[r]\ar[d]&\cdots\\
         \vdots & \vdots & 
       }
   \end{equation}
   commute, where $f^*\OC_Y^{\dagger} = \OX\otimes_{\widehat \calO_Y}\OC_Y^{\dagger}$ and $f^*\Theta_Y = \id\otimes\Theta_Y$.
 \end{prop}
 \begin{proof}
   Put $\calC:=\varinjlim_{\rho,\nu_p(\rho)>\nu_p(\rho_k)}(\OXp\widehat \otimes_{\widehat \calO_Y^+}\calO\widehat \bC_{\rho,Y}^+)[\frac{1}{p}]$. Since $f$ admits a lifting $\widetilde f$, for any $\rho\in\rho_k\calO_{\Cp}$, we have a morphism $\OXp\otimes_{\widehat \calO_Y^+}\calO\bC_{\rho,Y}^+\to \calO\bC_{\rho,X}^+$ and hence morphisms $f^*\OC_Y^{\dagger}\to\calC\to\OC_X^{\dagger}$. It remains to show the relative Higgs complex ${\rm HIG}(\OC_X^{\dagger},\Theta_{X/Y})$ is a resolution of $\calC$ and that the diagram (\ref{Diag-commutative diagram}) commutes.
   Since the problem is local, we may assume $\frakY = \Spf(S^+)$ and $\frakX = \Spf(R^+)$ are both small affine such that the morphism $f:\frakX\to\frakY$ is induced by a morphism $S^+\to R^+$ which makes the following diagram
   \[\xymatrix@C=0.45cm{\calO_{\Cp}\za T_1^{\pm 1},\dots,T_d^{\pm 1}\ya\ar[r]^{\subset\qquad\qquad}\ar[d]&\calO_{\Cp}\za  T_1^{\pm 1},\dots,T_d^{\pm 1},T_{d+1}^{\pm 1},\dots, T_{d+r}^{\pm 1}\ya\ar[d]\\
   S^+\ar[r]&R^+}
   \]
   commute, where $d$ is the dimension of $\frakY$ over $\calO_{\Cp}$, $r$ is the dimension of $\frakX$ over $\frakY$ and both vertical maps are \'etale. Let $\widehat S_{\infty}^+$ and $\widehat R_{\infty}^+$ be the perfectoid rings corresponding to the base-changes of $S^+$ and $R^+$ along morphisms 
   \[\calO_{\Cp}\za T_1^{\pm 1},\dots,T_d^{\pm 1}\ya\to \calO_{\Cp}\za T_1^{\pm \frac{1}{p^{\infty}}},\dots,T_d^{\pm \frac{1}{p^{\infty}}}\ya\]
   and 
   \[\calO_{\Cp}\za  T_1^{\pm 1},\dots,T_d^{\pm 1},T_{d+1}^{\pm 1},\dots, T_{d+r}^{\pm 1}\ya\to \calO_{\Cp}\za  T_1^{\pm \frac{1}{p^{\infty}}},\dots,T_d^{\pm \frac{1}{p^{\infty}}},T_{d+1}^{\pm \frac{1}{p^{\infty}}},\dots, T_{d+r}^{\pm \frac{1}{p^{\infty}}}\ya,\]
   respectively. Put $Y_{\infty} = \Spa(\widehat S_{\infty},\widehat S_{\infty}^+)$ and $X_{\infty} = \Spa(\widehat R_{\infty},\widehat R_{\infty}^+)$ with $\widehat S_{\infty}=\widehat S_{\infty}^+[\frac{1}{p}]$ and $\widehat R_{\infty}=\widehat R_{\infty}^+[\frac{1}{p}]$. For any $\rho\in\rho_k\calO_{\Cp}$, since $\calE_{\rho,Y}^+$ fits into the exact sequence
   \[0\to \widehat \calO_X^+\to\OXp\otimes_{\widehat \calO_Y^+}\calE_{\rho,Y}^+\to\rho\widehat \Omega^1_{\frakY}\otimes_{\calO_{\frakY}}\widehat \calO_X^+(-1)\to 0,\]
   we see that $(\OXp\otimes_{\widehat \calO_Y^+}\calE_Y^+)(X_{\infty})(\subset \calE_{\rho,X}^+(X_{\infty}))$ coincides with $\widehat R_{\infty}^+\otimes_{\widehat S_{\infty}^+}\calE_{\rho,Y}^+(Y_{\infty})$.
   This implies that 
   \[(\OXp\otimes_{\widehat \calO_Y^+}\calO\bC_{\rho,Y}^+)(X_{\infty}) \cong \widehat R_{\infty}^+[\rho Y_1,\dots,\rho Y_d]\] 
   such that the induced Higgs field is given by $\sum_{i=0}^d\frac{\partial}{\partial Y_i}\otimes\frac{\dlog T_i}{t}$. On the other hand, we have 
   \[\calO\bC_{\rho,X}^+(X_{\infty}) \cong \widehat R_{\infty}^+[\rho Y_1,\dots,\rho Y_{d+r}]\] 
   such that the induced Higgs field is given by $\sum_{i=0}^{d+r}\frac{\partial}{\partial Y_i}\otimes\frac{\dlog T_i}{t}$. So the morphism $\OXp\otimes_{\widehat \calO_Y^+}\calO\bC_{\rho,Y}^+\to \calO\bC_{\rho,X}^+$ is compatible with Higgs fields for any $\rho\in\rho_k\calO_{\Cp}$. Therefore, for any $\rho\in\rho_k\calO_{\Cp}$, we have morphisms of sheaves
   \[\OXp\otimes_{\widehat \calO_Y^+}\calO\bC_{\rho,Y}^+\to\OXp\otimes_{\widehat \calO_Y^+}\calO\widehat \bC_{\rho,Y}^+\to\OXp\widehat \otimes_{\widehat \calO_Y^+}\calO\widehat \bC_{\rho,Y}^+\to \calO\widehat \bC_{\rho,X}^+\]
   which are all compatible with Higgs fields. After taking direct limits and inverting $p$, we get morphisms
   \[f^*\calO\bC_Y^{\dagger}\to\calC\to\calO\bC_X^{\dagger}\]
   of sheaves which are compatible with Higgs fields. In particular, the top two rows of (\ref{Diag-commutative diagram}) form a commutative diagram.
   
   To complete the proof, we have to show that ${\rm HIG}(\OC_X^{\dagger},\Theta_{X/Y})$ is a resolution of $\calC$. Since we do have a morphism $\calC\to {\rm HIG}(\OC_X^{\dagger},\Theta_{X/Y})$, we can conclude by checking the exactness locally:
   
   By the ``moreover'' part of Corollary \ref{overconvergent resolution}, we obtain that 
   \[\OC_X^{\dagger}(X_{\infty}) = (\varinjlim_{\rho,\nu_p(\rho)>\nu_p(\rho_k)}\widehat R_{\infty}^+\za \rho Y_1,\dots,\rho Y_{d+r}\ya)[\frac{1}{p}]\]
   with $\Theta_X = \sum_{i=1}^{d+r}\frac{\partial}{\partial Y_i}\otimes\frac{\dlog T_i}{t}$. A similar argument also shows that $\Theta_{X/Y} = \sum_{i=d+1}^{d+r}\frac{\partial}{\partial Y_i}\otimes\frac{\dlog T_i}{t}$. So the rest part of (\ref{Diag-commutative diagram}) commutes.
   Note that $\calC(X_{\infty}) = (\varinjlim_{\rho,\nu_p(\rho)>\nu_p(\rho_k)}\widehat R_{\infty}^+\za \rho Y_1,\dots,\rho Y_{d}\ya)[\frac{1}{p}]$. By a similar argument in the proof of  Corollary \ref{overconvergent resolution}, we see that ${\rm HIG}(\OC_X^{\dagger},\Theta_{X/Y})$ is a resolution of $\calC$ as desired.
 \end{proof}

\section{An integral decompletion theorem}\label{Sec 3}
In this section, we generalize results in \cite[Appendix A]{DLLZ22} to an integral case which will be used to simplify local calculations. Let $\frakX = \Spf(R^+)$, $\Rinfp$ and $\Gamma$ be as in the previous section. Throughout this section, we put $\pi = \zeta_p-1$, $r = \nu_p(\pi) = \frac{1}{p-1}$ and $c = p^r$. Recall $\nu_p(\rho_k)\geq r$.
  We begin with some definitions.

\begin{dfn}\label{Integral Banach Algebras}
\begin{enumerate}
  \item By a {\bf Banach $\calO_{\Cp}$-algebra}, we mean a flat $\calO_{\Cp}$-algebra $A$ such that
$A[\frac{1}{p}]$ is a Banach $\Cp$-algebra, and that
$A = \{a\in A[\frac{1}{p}] \mid \|a\|\leq 1\}$.

  \item Assume $A$ is a Banach $\calO_{\Cp}$-algebra. For an $A$-module $M$, we say it is a {\bf Banach $A$-module} if
$M[\frac{1}{p}]$ is a Banach $A[\frac{1}{p}]$-module, and
$M=\{m\in M[\frac{1}{p}]\mid \|m\|\leq 1\}$.
 \end{enumerate}
\end{dfn}
 There are some typical examples.
\begin{exam}
\begin{enumerate}
  \item If $A$ is a Banach $\calO_{\Cp}$-algebra, then any topologically free $A$-module endowed with the supreme norm is a Banach $A$-module.

  \item The rings $R^+$ and $\Rinfp$ are Banach $\calO_{\Cp}$-algebras.

  \item The $\Rinfp/R^+$ is a Banach $R^+$-module.
 \end{enumerate}
\end{exam}

 Now, we make the definition of ($a$-trivial) $\Gamma$-representations.

\begin{dfn}\label{Local generalized rep}
  Assume $a>r$ and $A\in \{R^+,\Rinfp\}$.
  \begin{enumerate}
  \item By an {\bf $A$-representation of $\Gamma$} of rank $l$, we mean a finite free $A$-module $M$ of rank $l$ endowed with a continuous semi-linear $\Gamma$-action.
  
  \item Let $M$ be a representation of $\Gamma$ of rank $l$ over $A$. We say $M$ is {\bf $a$-trivial}, if $M/p^a\cong (A/p^a)^l$ as representations of $\Gamma$ over $A/p^a$.
  
  \item Let $M$ be a representation of $\Gamma$ of rank $l$ over $R^+$. We say $M$ is {\bf essentially $(a+r)$-trivial} if $M$ is $a$-trivial and $M\otimes_{R^+}\widehat R_{\infty}^+$ is $(a+r)$-trivial.
  \end{enumerate}
\end{dfn}
 The goal of this section is to prove the following integral decompletion theorem.
\begin{thm}\label{Strong Decompletion}
  Assume $a>r$. Then the functor $M\mapsto M\otimes_{R^+}\Rinfp$ induces an equivalence from the category of $(a+r)$-trivial $R^+$-representations of $\Gamma$ to the category of $(a+r)$-trivial $\Rinfp$-representations of $\Gamma$. The equivalence preserves tensor products and dualities.
\end{thm}
 The first difficulty is to construct the quasi-inverse, namely the decompletion functor, of the functor in Theorem \ref{Strong Decompletion}. To do so, we need to generalize the method adapted in \cite{DLLZ22} to the small integral case. However, their method only shows the decompletion functor takes values in the category of essentially $(a+r)$-trivial representations. So, the second difficulty is to show the resulting representation is actually $(a+r)$-trivial. The trivialness condition is crucial to overcome both difficulties.

\subsection{Construction of decompletion functor}
  Now we construct the decompletion functor at first. From now on, we use $\rR\Gamma(\Gamma,M)$ to denote the continuous group cohomology of a $p$-adically completed $R^+$-module endowed with a continuous $\Gamma$-action. By virtues of \cite[Lemma 7.3]{BMS1}, $\RGamma(\Gamma,M)=\Rlim_k\RGamma(\Gamma,M/p^k)$ can be calculated by Koszul complex $\rK(M;\gamma_1-1,\dots,\gamma_d-1)$:
  \[M\xrightarrow{(\gamma_1-1,\dots,\gamma_d-1)}M^d\to\cdots.\]
\begin{prop}\label{Prop-Weak Decompletion}
  Assume $a>r$. Let $M_{\infty}$ be an $(a+r)$-trivial $\Rinfp$-representation of $\Gamma$. Then there exists a finite free $R^+$-submodule $M\subset M_{\infty}$ such that the following assertions are true:
  \begin{enumerate}

  \item The $M$ is an essentially $(a+r)$-trivial $R^+$-representation of $\Gamma$ such that the natural inclusion $M\hookrightarrow M_{\infty}$ induces an isomorphism $M\otimes_{R^+}\Rinfp\cong M_{\infty}$ of $\Rinfp$-representations of $\Gamma$.

  \item The induced morphism $\RGamma(\Gamma,M)\rightarrow \RGamma(\Gamma,M_{\infty})$ identifies the former as a direct summand of the latter, whose complement is concentrated in positive degrees and killed by $\pi$.
  \end{enumerate}
\end{prop}
\begin{rmk}
  The $M$ is unique up to isomorphism and the functor $M_{\infty}\mapsto M$ turns out to be the quasi-inverse of the functor $M\mapsto M\otimes R^+_{\infty}$ described in Theorem \ref{Strong Decompletion}.
\end{rmk}
 Now we prove Proposition \ref{Prop-Weak Decompletion} by using similar arguments in \cite{DLLZ22}. Since we work on the integral level, so we need to control ($p$-adic) norms carefully. We start with the following result.
 \begin{lem}\label{USE for DLLZ}
   For any cocycle $f\in\rC^{\bullet}(\Gamma,\Rinf/R)$, there exists a cochain $g\in \rC^{\bullet-1}(\Gamma,\Rinf/R)$ such that $dg = f$ and $\|g\|\leq c\|f\|.$
 \end{lem}
 \begin{proof}
   The result follows from the same argument used in the proof of \cite[Proposition A.2.2.1]{DLLZ22}, especially the part for checking the condition (3) of \cite[Definition A.1.6]{DLLZ22}, by using \cite[Lemma 5.5]{Sch2} instead of \cite[Lemma 6.1.7]{DLLZ19} there.
 \end{proof}
 
 Since the norm on $R$ (resp. $\widehat R_{\infty}$) is induced by that on $R^+$ (resp. $\widehat R_{\infty}^+$), there exists a norm-preserving embedding of complexes 
 \[\rC^{\bullet}(\Gamma,\widehat R_{\infty}^+/R^+)\to \rC^{\bullet}(\Gamma,\widehat R_{\infty}/R).\]
 We shall apply Lemma \ref{USE for DLLZ} via this embedding.
 
\begin{lem}\label{Lem-Uniform Strict Exactness}
  For any cocycle $f\in\rC^{\bullet}(\Gamma,\Rinfp/R^+)$, there is a cochain $g\in\rC^{\bullet-1}(\Gamma,\Rinfp/R^+)$ such that $\|g\|\leq \|f\|$ and $dg = \pi f$.
\end{lem}
\begin{proof}
  Regard $\rC^{\bullet}(\Gamma,\Rinfp/R^+)$ as a subcomplex of $\rC^{\bullet}(\Gamma,\Rinf/R)$ as above. Applying Lemma \ref{USE for DLLZ} to $\pi f$, we get a cochain $g\in \rC^{\bullet-1}(\Gamma,\Rinf/R)$ such that $\|g\|\leq c\|\pi f\|$ and $dg = \pi f$. But $c\|\pi f\| = \|f\|\leq 1$, so we see $g\in \rC^{\bullet-1}(\Gamma,\Rinfp/R^+)$ as desired.
\end{proof}

\begin{lem}\label{Lem-An Approximation Lemma}
  Let $(\rC^{\bullet},d)$ be a complex of Banach modules over a Banach $\calO_{\Cp}$-algebra $A$. Suppose that for every degree $s$ and every cocycle $f\in \rC^s$, there exists a $g\in \rC^{s-1}$ such that $\|g\|\leq \|f\|$ and $dg = \pi f$. Then, for any cochain $f\in \rC^s$, there exists an $h\in \rC^{s-1}$ such that $\|h\|\leq \max(\frac{\|f\|}{c}, \|df\|)$ and $\|\pi^2 f-dh\|\leq \frac{\|df\|}{c}$.
\end{lem}
\begin{proof}
  By assumption, one can choose a $g\in \rC^s$ such that $dg = \pi df$ and that $\|g\|\leq \|df\|$. Then $(g-\pi f)\in \rC^{s}$ is a cocycle. Using assumption again, there is an $h\in \rC^{s-1}$ satisfying $\|h\|\leq \|g-\pi f\|$ and $dh = \pi(g-\pi f)$. Since $\|\pi^2 f - dh\|\leq \frac{\|g\|}{c} \leq \frac{\|df\|}{c}$ and $\|h\|\leq \max(\|df\|,\frac{\|f\|}{c})$, this $h$ is desired.
\end{proof}

The following lemma is a consequence of Lemma \ref{Lem-Uniform Strict Exactness} and Lemma \ref{Lem-An Approximation Lemma}.

\begin{lem}\label{Lem-USE For Toric Chart}
  For any cochain $f\in\rC^{\bullet}(\Gamma,\Rinfp/R^+)$, there is a cochain $h\in\rC^{\bullet-1}(\Gamma,\Rinfp/R^+)$ such that $\|h\|\leq \max(\frac{\|f\|}{c}, \|df\|)$ and $\|\pi^2 f-dh\|\leq \frac{\|df\|}{c}$.
\end{lem}

The following lemma can be viewed as an integral version of \cite[Lemma A.1.12]{DLLZ22}.

\begin{lem}\label{Lem-USE For Small Rep}
  We denote $(R^+,\Rinfp/R^+)$ by $(A,M)$ for simplicity.

  Let $L = \bigoplus_{i=1}^n Ae_i$ be a Banach $A$-module (with the supreme norm) endowed with a continuous $\Gamma$-action. Assume there exists an $R>1$ such that, for each $\gamma\in\Gamma$ and each $i$, $\|(\gamma-1)(e_i)\|\leq \frac{1}{Rc}$. Then the following assertions are true:
\begin{enumerate}
  \item For any cocycle $f\in\rC^{\bullet}(\Gamma, L\otimes_A M)$, there is a cochain $g\in\rC^{\bullet-1}(\Gamma, L\otimes_A M)$ such that $\|g\|\leq \|f\|$ and $dg = \pi f$.

  \item For any cochain $f\in\rC^{\bullet}(\Gamma, L\otimes_A M)$, there exists an $h\in \rC^{\bullet}(\Gamma, L\otimes_A M)$ such that $\|h\|\leq \max(\frac{\|f\|}{c}, \|df\|)$ and $\|\pi^2 f-dh\|\leq \frac{\|df\|}{c}$.
 \end{enumerate}
\end{lem}
\begin{proof}
  We only prove $(1)$ and then $(2)$ follows from Lemma \ref{Lem-An Approximation Lemma} directly.

  Now, let $f = \sum_{i=1}^n e_i\otimes f_i$ be a cocycle with $f_j\in \rC^s(\Gamma,M)$ for all $1\leq j\leq n$. Then $\|f\|\leq 1$. For any $\gamma_1, \gamma_2, ..., \gamma_{s+1}\in \Gamma$, we have
  \begin{equation*}
  \begin{split}
  (\sum_{i=1}^n e_i\otimes df_i)(\gamma_1, ..., \gamma_{s+1})& = (\sum_{i=1}^n e_i\otimes df_i)(\gamma_1, ..., \gamma_{s+1})-df(\gamma_1, ..., \gamma_{s+1})\\
  & = \sum_{i=1}^n (1-\gamma_1)(e_i)\otimes f_i(\gamma_2, ..., \gamma_{s+1}).
  \end{split}
  \end{equation*}
  It follows that $\|\sum_{i=1}^n e_i\otimes df_i\| \leq \frac{\|f\|}{Rc}$. In other words, for each $1\leq j\leq n$, we have $\|df_j\|\leq \frac{\|f\|}{Rc}$. By Lemma \ref{Lem-USE For Toric Chart}, for every $j$, there is a $g_j\in \rC^{s-1}(\Gamma,M)$ such that $\|g_j\|\leq \max(\frac{\|f_j\|}{c},\|df_j\|)\leq \frac{\|f_j\|}{c}$ and $\|\pi^2f_j-dg_j\|\leq \frac{\|df_j\|}{c}\leq\frac{\|f\|}{Rc^2}$.

  Now, put $g = \sum_{i=1}^n e_i\otimes g_i$. Then $\|g\| \leq \frac{\|f\|}{c}$. On the other hand, we have
  \[\pi^2f-dg = \sum_{i=1}^ne_i\otimes(\pi^2f_i - dg_i) + (\sum_{i=1}^n e_i\otimes (dg_i - dg)).\]
  The first term on the right hand side is bounded by $\frac{\|f\|}{Rc^2}$ and the second term is bounded by $\frac{\|g\|}{Rc}\leq \frac{\|f\|}{Rc^2}$. Thus $\|\pi^2f-dg\|$ is bounded by $\frac{\|f\|}{Rc^2}$.
  Then $h_1:=\frac{g}{\pi}$ belongs to $\rC^{s-1}(\Gamma,(L\otimes_A M))$ such that $\|h_1\|\leq \|f\|$ and that $\|\pi f-dh_1\|\leq\frac{\|f\|}{Rc}$.

  Assume we have already $h_1, h_2, ..., h_t \in \rC^{s-1}(\Gamma,L\otimes_A M)$ satisfying
  \[\|h_j\| \leq \frac{\|f\|}{R^{j-1}} \quad {\rm and}\quad \|\pi f-\sum_{i=1}^j dh_i\|\leq\frac{\|f\|}{R^jc},\quad\forall 1\leq j\leq t.\]
  Then $f-\pi^{-1}\sum_{i=1}^t dh_i\in \rC^s(\Gamma,L\otimes_A M)$ with norm $\|f-\pi^{-1}\sum_{i=1}^t dh_i\|\leq \frac{\|f\|}{R^{t}}$.
  Replacing $f$ by $f-\pi^{-1}\sum_{i=1}^t dh_i$ and proceeding as above, we get an $h_{t+1}\in \rC^{s-1}(\Gamma,L\otimes_A M)$ with norm $\|h_{t+1}\|\leq \|f-\pi^{-1}\sum_{i=1}^t dh_i\|\leq \frac{\|f\|}{R^t}$ such that
  \[\|\pi f-\sum_{i=1}^t dh_i - dh_{t+1}\|\leq\frac{\|f-\pi^{-1}\sum_{i=1}^t dh_i\|}{Rc}\leq\frac{\|f\|}{R^{t+1}c}.\]
  Then $\sum_{i=1}^{+\infty}h_i$ converges to an element $h\in \rC^{s-1}(\Gamma,L\otimes_A M)$ such that $\pi f=dh$ and that $\|h\|\leq \sup_{j\geq 1}(\|h_j\|)\leq\|f\|$. This implies $(1)$.
\end{proof}
The following lemma is a generalization of \cite[Lemma A.1.14]{DLLZ22} whose proof is similar.
\begin{lem}\label{Lem-Norm Contral Lemma}
  Let $A\rightarrow B$ be an isometry of Banach $\calO_{\Cp}$-algebras. Suppose the natural projection $\mathrm{pr}:B\rightarrow B/A$ admits an isometric section $s:B/A\rightarrow B$ as Banach modules over $A$. Then, for all $b_1, b_2\in B$, we have
  \[\|\mathrm{pr}(b_1b_2)\|\leq \max(\|b_1\|\|\mathrm{pr}(b_2)\|, \|b_2\|\|\mathrm{pr}(b_1)\|)\]
\end{lem}
 We shall apply this lemma to the inclusion $R^+\rightarrow \widehat{R}^+_{\infty}$.
\begin{lem}\label{Lem-Descent Of Small Rep}
  We denote the triple $(R^+,\widehat{R}^+_{\infty})$ by $(A,B)$ for simplicity. Let $f$ be a cocycle in $\rC^1(\Gamma,\GL_n(B))$. Suppose there exists an $R>1$ such that $\|f(\gamma)-1\|\leq\frac{1}{Rc}$ for all $\gamma\in \Gamma$. Let $\overline{f}$ be the image of $f$ in $\rC^1(\Gamma,\rM_n(B/A))$ (which is not necessary a cocycle). If $\|\overline{f}\|\leq\frac{1}{Rc^2}$, then there exists a cocycle $f^{\prime}\in \rC^1(\Gamma,\GL_n(A))$ which is equivalent to $f$ such that $\|f'(\gamma)-1\|\leq\frac{1}{Rc}$ for all $\gamma\in\Gamma$.
\end{lem}
\begin{proof}
  We proceed as in the proof of \cite[Lemma A.1.15]{DLLZ22}. It is enough to show that there exists an $h\in \rM_n(B)$ with $\|h\|\leq c\|\overline{f}\|$ such that the cocycle
  \[g:\gamma\mapsto \gamma(1+h)f(\gamma)(1+h)^{-1}\] satisfies $\|g(\gamma)-1\|\leq\frac{1}{Rc}$ for all $\gamma\in\Gamma$ and $\|\overline{g}\|\leq\frac{\|\overline{f}\|}{R}$ in $\rC^1(\Gamma,\rM_n(B/A))$.

  Granting the claim, by iterating this procession, we can find a sequence $h_1, h_2, \cdots$ in $\rM_n(B)$ with $\|h_n\|\leq\frac{c\|\overline{f}\|}{R^{n-1}}\leq\frac{1}{cR^n}$ such that
  \[\overline{\gamma(\prod_{i=1}^n(1+h_i))f(\gamma)(\prod_{i=1}^n(1+h_i))^{-1}}\leq\frac{\|\overline{f}\|}{R^n}.\]
  Set $h = \prod_{i=1}^{+\infty}(1+h_i) \in \GL_n(B)$. Then we have a cocycle
  \[f^{\prime}:\gamma\mapsto\gamma(h)f(\gamma)h^{-1}\]
  taking values in $\rM_n(A)\cap\GL_n(B)$ such that $\|f^{\prime}(\gamma)-1\|\leq\frac{1}{Rc}$ for every $\gamma\in \Gamma$. Thus $f^{\prime}\in \GL_n(A)$ and we prove the lemma.

  Now, we prove the claim. Since $f\in \rC^1(\Gamma,\GL_n(B))$ is a cocycle, for all $\gamma_1,\gamma_2\in\Gamma$, we have $f(\gamma_1\gamma_2) = \gamma_1(f(\gamma_2))f(\gamma_1)$. Using Lemma \ref{Lem-Norm Contral Lemma}, we get
  \begin{equation}
  \begin{split}
  \|d\overline{f}(\gamma_1,\gamma_2)\|& =\|\overline{\gamma_1f(\gamma_2)+f(\gamma_1)-f(\gamma_1\gamma_2)}\| \\
  & = \|\overline{(\gamma_1f(\gamma_2)-1)(f(\gamma_1)-1)-1}\|\\
  &= \|\overline{(\gamma_1f(\gamma_2)-1)(f(\gamma_1)-1)}\|\leq\frac{\|\overline{f}\|}{Rc}.
  \end{split}
  \end{equation}
  Since $\|\overline{f}\|\leq\frac{1}{Rc^2}$, we can apply Lemma \ref{Lem-USE For Toric Chart} to $\pi^{-2}\overline{f}$ and get an $\overline{h}\in\rM_n(B/A)$ such that
  \[\|\overline{h}\|\leq\max(\frac{\|\pi^{-2}\overline{f}\|}{c},\|\pi^{-2}d\overline{f}\|)\leq\max(c\|\overline{f}\|,c^2\|d\overline{f}\|)\leq c\|\overline{f}\|\leq\frac{1}{Rc}.\]
  and that
  \begin{equation}\label{Equ-Descent Lemma I}
  \|\overline{f}-d\overline{h}\|\leq\frac{\|\pi^{-2}d\overline{f}\|}{c}\leq c\|d\overline{f}\|\leq\frac{\|\overline{f}\|}{R}.
  \end{equation}
  By assumption, we can lift $\overline{h}$ to an $h\in \rM_n(B)$ such that $\|h\|=\|\overline{h}\|\leq c\|\overline{f}\|$. It follows that for all $\gamma\in\Gamma$, we have
  \[\|\gamma(1+h)f(\gamma)(1+h)^{-1}-f(\gamma)\|\leq \|h\|\leq\frac{1}{Rc}\]
   and therefore,
  \[\|\gamma(1+h)f(\gamma)(1+h)^{-1}-1\|\leq \frac{1}{Rc}.\]
  Moreover, we have
  \begin{equation}\label{Equ-Descent Lemma II}
    \|\overline{\gamma(1+h)f(\gamma)(1+h)^{-1}-\gamma(1+h)f(\gamma)(1-h)}\|\leq\|\overline{h}^2\|\leq\frac{c\|\overline{f}\|}{Rc} = \frac{\|\overline{f}\|}{R}.
  \end{equation}
  By Lemma \ref{Lem-Norm Contral Lemma}, we have
  \begin{equation}\label{Equ-Descent Lemma III}
    \begin{split}
    &\|\overline{\gamma(1+h)f(\gamma)(1-h)}-\overline{f}(\gamma)-\gamma(\overline{h})+\overline{h}\|
    \\= &\|\overline{\gamma(h)(f(\gamma)-1)}-\overline{(f(\gamma)-1)h}-\overline{\gamma(h)f(\gamma)h}\|\leq \frac{\|\overline{f}\|}{R}.
    \end{split}
  \end{equation}

  Combining $(\ref{Equ-Descent Lemma I})$, $(\ref{Equ-Descent Lemma II})$ and $(\ref{Equ-Descent Lemma III})$, we conclude that
  \[\|\overline{\gamma(1+h)f(\gamma)(1+h)^{-1}}\|\leq\frac{\|\overline{f}\|}{R}\]
  which proves the claim as desired.
\end{proof}
  Now we are able to prove  Proposition \ref{Prop-Weak Decompletion}.
\begin{proof}(of Proposition \ref{Prop-Weak Decompletion})
\begin{enumerate}
  \item Since $a>r$, we may choose $s>1$ such that $\|p^{a+r}\| = \frac{1}{sc^2}$. By assumptions, for a set of basis $\{e_1, e_2, ..., e_n\}$ of $M_{\infty}$, it determines a cocycle $f\in \rC^1(\Gamma,\GL_n(\Rinfp))$ satisfying $\|f(\gamma)-1\|\leq \frac{1}{sc^2}$. In particular, $f$ satisfies the hypothesis of Lemma \ref{Lem-Descent Of Small Rep}. Thus there exists a cocycle $f^{\prime}\in \rC^1(\Gamma,R^+)$ which is equivalent to $f$ such that
  \[\|f^{\prime}(\gamma)-1\|\leq \frac{1}{sc},\quad \forall \gamma\in\Gamma.\]
  Then the cocycle $f^{\prime}$ defines a finite free sub-$R^+$-module $M$ of rank $n$ such that
  \[M\otimes_{R^+}\widehat{R}_{\infty}^+\cong M_{\infty}.\]

  \item By $(1)$, we have $M_{\infty} \cong M\oplus M\otimes_{R^+}(\widehat{R}_{\infty}^+/R^+)$. Applying Lemma \ref{Lem-USE For Small Rep} (1) to $L = M$, we deduce that $\rH^i(\Gamma, M\otimes_{R^+}\widehat{R}_{\infty}^+/R^+)$ is killed by $\pi$ for every $i\geq 0$.
 But $\rH^0(\Gamma, M_{\infty}) = M_{\infty}^{\Gamma}$ is $\pi$-torsion free, so we get
  \[\rH^0(\Gamma, M_{\infty}) = \rH^0(\Gamma, M)\]
  and complete the proof.
 \end{enumerate}
\end{proof}
 Up to now, we have constructed a decompletion functor from the category of $(a+r)$-trivial $\Rinfp$-representations of $\Gamma$ to the category of essentially $(a+r)$-trivial $R^+$-representations of $\Gamma$. Now Theorem \ref{Strong Decompletion} follows from the next proposition directly.
 \begin{prop}\label{essential small is small}
   Every essentially $(a+r)$-trivial $R^+$-representation of $\Gamma$ is $(a+r)$-trivial.
 \end{prop}
 We leave the proof of this proposition in the next subsection.
\subsection{Essentially $(a+r)$-trivial representation is $(a+r)$-trivial}
  Throughout this subsection, we always assume $a>r$. For any $R^+$-module $N$ with a continuous $\Gamma$-action, we denote $\rH^i(\Gamma,N)$ by $\rH^i(N)$ for simplicity. 
  
  Now for a fixed essentially $(a+r)$-trivial $R^+$-representation $M$ of $\Gamma$ of rank $n$, we define 
  \[
  M_{\infty} = M\otimes_{R^+}\Rinfp.
  \]
  Then it is $(a+r)$-trivial and of the form $M_{\infty} = M\oplus M_{\rm cp}$ for $M_{\rm cp} = M\otimes_{R^+}\Rinfp/R^+$. Since $M$ is $a$-trivial, by Lemma \ref{Lem-USE For Small Rep}, we see that $\RGamma(\Gamma,M_{\rm cp})$ is concentrated in positive degrees and is killed by $\pi$. As a consequence, for any $h\geq r$, we have 
  \[
  \RGamma(\Gamma,M_{\rm cp}/p^h)\simeq \RGamma(\Gamma,M_{\rm cp})[1].
  \]
  In particular, $\RGamma(\Gamma,M_{\rm cp}/p^h)$ is killed by $\pi$. So we deduce that 
\[\pi\rH^0(M_{\infty}/p^h)\cong \pi\rH^0(M/p^h).\]
  Replacing $M$ by $(\Rinfp)^l$, we get
 \[\pi\rH^0(\Rinfp/p^h)^n\cong \pi\rH^0(R^+/p^h)^n = (\pi R^+/p^h)^n.\]
 Since $M_{\infty}$ is $(a+r)$-trivial, choose $h=a+r$ and we get
 \[\pi\rH^0(M/p^{a+r})\cong\pi\rH^0(M_{\infty}/p^{a+r})\cong \pi\rH^0(\Rinfp/p^{a+r})^n\cong(\pi R^+/p^{a+r})^n\cong (R^+/p^a)^n.\]
 Thus, $\pi\rH^0(M/p^{a+r})$ is a free $R^+/p^a$-module of rank $n$.
 
 Choose $g_1, \dots g_n\in \rH^0(M/p^{a+r})$ such that $\pi g_1,\dots,\pi g_n$ is an $R^+/p^a$-basis of $\pi\rH^0(M/p^{a+r})$. We claim that the sub-$R^+/p^{a+r}$-module
 \[\sum_{i=1}^nR^+/p^{a+r}\cdot g_i\subset\rH^0(M/p^{a+r})\]
 is free. For any $i$, let $\widetilde g_i\in M$ be a lifting of $g_i$. Assume $x_1,\dots,x_n\in R^+$ such that 
 \[\sum_{i=1}^nx_i\widetilde g_i\equiv 0\mod p^{a+r}.\]
 Then 
 \[\sum_{i=1}^nx_i\pi\widetilde g_i\equiv 0\mod p^{a+r}.\]
 By the choice of $g_i$'s, we deduce that $x_i\in p^aR^+$ for any $i$. Write $x_i = \pi y_i$ for some $y_i\in R^+$. Then 
 \[\sum_{i=1}^ny_i\pi\widetilde g_i\equiv 0\mod p^{a+r}.\]
 So $y_i\in p^aR^+$ and hence $x_i\in p^{a+r}R^+$ for all $i$. This proves the claim.
 
 It remains to prove $\widetilde g_1, \dots,\widetilde g_n$ is an $R^+$-basis of $M$.
 Let $e_1,\dots,e_n$ be an $R^+$-basis of $M$. Since $M$ is $a$-trivial, we get
 \[M/p^a = \rH^0(M/p^a) = \sum_{i=1}^nR^+/p^ae_i.\]
 So $\pi e_1,\cdots, \pi e_n$ is an $R^+/p^{a-r}$-basis of $\pi M/p^a$. However, by the choice of $\widetilde g_i$'s, $\pi \widetilde g_1,\cdots, \pi \widetilde g_n$ is also an $R^+/p^{a-r}$-basis of $\pi M/p^a$. Since $a>r$, we deduce that $\widetilde g_i$'s generate $M$ as an $R^+$-module. This completes the proof.

\section{Local Simpson correspondence}\label{Sec 4}
 In this section, we establish an equivalence between the category of $a$-small representations of $\Gamma$ over $\Rinfp$ and the category of $a$-small Higgs modules over $R^+$. This is a local version of $p$-adic Simpson correspondence. Throughout this section, put $r = \frac{1}{p-1}$.
 \begin{dfn}\label{Dfn-small representation}
    Assume $a>r$ and $A\in\{R^+,\Rinfp\}$. We say a representation $M$ of $\Gamma$ over $A$ is {\bf $a$-small} if it is $(a+\nu_p(\rho_k))$-trivial in the sense of Definition \ref{Local generalized rep}.
 \end{dfn}
 \begin{dfn}\label{Dfn-Higgs module}
    By a {\bf Higgs module} over $R^+$, we mean a finite free $R^+$-module $H$ together with an $R^+$-linear morphism $\theta:H\to H\otimes_{R^+}\widehat \Omega^1_{R^+}(-1)$ such that $\theta\wedge\theta = 0$. A Higgs module $(H,\theta)$ is called {\bf $a$-small}, if $\theta$ is divided by $p^{a+\nu_p(\rho_k)}$; that is, \[\Ima(\theta)\subset p^{a+\nu_p(\rho_k)} H\otimes_{R^+}\widehat \Omega^1_{R^+}(-1).\]
 \end{dfn}
  Let $S_{\infty}^{\dagger,+}$ with the canonical Higgs field $\Theta$ be as in Corollary \ref{overconvergent resolution}.
  For an $a$-small representation $M$ over $\Rinfp$, define 
 \begin{equation}\label{Equ-Higgs field with coefficient-I}
 \Theta_M = \id_M\otimes\Theta: M\otimes_{\Rinfp}S_{\infty}^{\dagger,+}\to M\otimes_{\Rinfp}S_{\infty}^{\dagger,+}\otimes_{R^+}\widehat \Omega_{R^+}^1(-1).
 \end{equation}
  Then it is a Higgs field on $M\otimes_{\Rinfp}S_{\infty}^{\dagger,+}$ and we denote the induced Higgs complex by ${\rm HIG}(H\otimes_{R^+}S_{\infty}^{\dagger,+},\Theta_H)$. 
  For an $a$-small Higgs module $(H,\theta_H)$, define 
\begin{equation}\label{Equ-Higgs field with coefficient-II}
\Theta_H=\theta_H\otimes\id+\id_H\otimes\Theta:H\otimes_{R^+}S_{\infty}^{\dagger,+}\to H\otimes_{R^+}S_{\infty}^{\dagger,+}\otimes_{R^+}\widehat \Omega_{R^+}^1(-1).
\end{equation}
 Then $\Theta_H$ is a Higgs field on $H\otimes_{R^+}S_{\infty}^{\dagger,+}$ and we denote the induced Higgs complex by ${\rm HIG}(H\otimes_{R^+}S_{\infty}^{\dagger,+},\Theta_H)$.
 The main theorem in this section is the following local version of Simpson correspondence.
\begin{thm}[Local Simpson correspondence]\label{local Simpson}
  Assume $a>r$.
  \begin{enumerate}
  \item  Let $M$ be an $a$-small $\Rinfp$-representation of $\Gamma$ of rank $l$. Let $H(M):=(M\otimes_{\Rinfp}S_{\infty}^{\dagger,+})^{\Gamma}$ and $\theta_{H(M)}$ be the restriction of $\Theta_M$ to $H(M)$. Then $(H(M),\theta_{H(M)})$ is an $a$-small Higgs module of rank $l$.
       
  \item  Let $(H,\theta_H)$ be an $a$-small Higgs module of rank $l$ over $R^+$. Put $M(H,\theta_H) = (H\otimes_{R^+}S_{\infty}^{\dagger,+})^{\Theta_H=0}$. Then $M(H,\theta_H)$ is an $a$-small $\Rinfp$-representation of $\Gamma$ of rank $l$.
  
  \item The functor $M\mapsto(H(M),\theta_{H(M)})$ induces an equivalence from the category of $a$-small $\Rinfp$-representations of $\Gamma$ to the category of $a$-small Higgs modules over $R^+$, whose quasi-inverse is given by $(H,\theta_H)\mapsto M(H,\theta_H)$. The equivalence preserves tensor products and dualities.
  
  \item  Let $M$ be an $a$-small $\Rinfp$-representation of $\Gamma$ and $(H,\theta_H)$ be the corresponding Higgs module. Then there is a canonical $\Gamma$-equivariant isomorphism of Higgs complexes
  \[{\rm HIG}(H\otimes_{R^+}S_{\infty}^{\dagger,+},\Theta_{H})\to {\rm HIG}(M\otimes_{\Rinfp}S_{\infty}^{\dagger,+},\Theta_M).\]
  Moreover, there is a canonical quasi-isomorphism
  \[\RGamma(\Gamma,M[\frac{1}{p}])\simeq {\rm HIG}(H[\frac{1}{p}],\theta_H),\]
  where ${\rm HIG}(H[\frac{1}{p}],\theta_H)$ is the Higgs complex induced by $(H,\theta_H)$.
  \end{enumerate}
\end{thm}
  The following corollary follows from Theorem \ref{Strong Decompletion} and Theorem \ref{local Simpson} directly.
 \begin{cor}
   Assume $a>r$. The following categories are equivalent:
   \begin{enumerate}
       \item The category of $a$-small representations of $\Gamma$ over $R^+$;
       
       \item The category of $a$-small representations of $\Gamma$ over $\Rinfp$;
       
       \item The category of $a$-small Higgs modules over $R^+$.
   \end{enumerate}
 \end{cor}
 
  In order to prove the theorem, we need to compute $\RGamma(\Gamma, M\otimes_{\Rinfp}S_{\infty}^{\dagger,+})$. By Corollary \ref{overconvergent resolution}, we are reduced to computing $\RGamma(\Gamma, M\otimes_{\Rinfp}\Rinfp\za\rho Y_1, \cdots, \rho Y_d\ya)$ for any $\rho\in\rho_k\calO_{\Cp}$. So before we move on, let us fix some notations to simplify the calculation. 
  
  For any $n\geq 0$, define 
  \[F_n(Y) = n!\binom{Y}{n}= Y(Y-1)\cdots (Y-n+1)\in \bZ[Y].\]
  For any $\alpha\in\bN[\frac{1}{p}]\cap(0,1)$, define $\epsilon_{\alpha} = 1-\zeta^{-\alpha}$. Then $\nu_p(\rho_k)\geq r\geq \nu_p(\epsilon_{\alpha})$.
\subsection{Calculation in trivial representation case}
  We are going to compute $\RGamma(\Gamma,\Rinfp\za\rho Y_1,\dots,\rho Y_d\ya)$ in this subsection. We assume $d = 1$ first. In this case, $\Gamma = \Zp\gamma$ and acts on $\Rinfp\za\rho Y\ya$ via $\gamma(Y) = Y+1$. Note that $\{\rho^n F_n\}_{n\geq 0}$ is a set of topological $\Rinfp$-basis of $\Rinfp\za\rho Y\ya$ and for any $n\geq 0$, 
  \[\gamma(\rho^n F_n) = \rho^n F_n+n\rho\cdot\rho^{n-1}F_{n-1}.\]
  So we get a $\gamma$-equivariant decomposition
  \[\Rinfp\za\rho Y\ya = \widehat \oplus_{\alpha\in\bN[\frac{1}{p}]\cap [0,1)}R^+\za\rho Y\ya\cdot T^{\alpha}.\]
  So it suffices to compute $\RGamma(\Gamma,R^+\za\rho Y\ya\cdot T^{\alpha})$ for any $\alpha$. We only need to consider the Koszul complex $\rK(R^+\za\rho Y\ya\cdot T^{\alpha};\gamma-1)$:
  \[R^+\za\rho Y\ya\cdot T^{\alpha}\xrightarrow{\gamma-1}R^+\za\rho Y\ya\cdot T^{\alpha}.\]
 Note that for any $\alpha$, $\{\rho^n F_nT^{\alpha}\}_{n\geq 0}$ is a set of topological $R^+$-basis of $R^+\za\rho Y\ya T^{\alpha}$. So we have 
 \begin{equation}\label{Equ-trivial representation}
     (\gamma-1)(\rho^n F_nT^{\alpha}) = \left\{
     \begin{array}{rcl}
          n\rho\cdot\rho^{n-1}F_{n-1}, &\alpha = 0  \\
          \zeta^{\alpha}\epsilon_{\alpha}T^{\alpha}(\rho^nF_n+n\frac{\rho}{\epsilon_{\alpha}}\rho^{n-1}F_{n-1}), &\alpha \neq 0.
     \end{array}
     \right.
 \end{equation}
 Put $\Lambda_{\rho}=\{\rho^n F_n\}_{n\geq 0}$ and $I_{\rho}=\{\nu_p(\rho(n+1))\}_{n\geq 0}$. Let $R^+\za\Lambda_{\rho}\ya$ and $R^+\za\Lambda_{\rho},I_{\rho},+\ya$ be as in Definition \ref{p-complete module}. Then by (\ref{Equ-trivial representation}), we see that
 \[
 (\gamma-1)(R^+\za\rho Y\ya) = R^+\za\Lambda_{\rho},I_{\rho},+\ya
 \]
 and that 
 \[(\gamma-1)(R^+\za\rho Y\ya T^{\alpha})\sim \{\zeta^{\alpha}\epsilon_{\alpha}(\rho^nF_n+n\frac{\rho}{\epsilon_{\alpha}}\rho^{n-1}F_{n-1})\}_{n\geq 0}\]
 in the sense of Definition \ref{equivalent of Basis}. By Proposition \ref{basis of free modules}, we get 
 \[(\gamma-1)(R^+\za\rho Y\ya T^{\alpha})=\epsilon_{\alpha}(R^+\za\rho Y\ya T^{\alpha}).\]
 In summary, we see that for $\alpha\neq 0$, $\rH^1(\Zp\gamma,R^+\za\rho Y\ya T^{\alpha})$ is killed by $\epsilon_{\alpha}$ and that for $\alpha = 0$, $\rH^1(\Zp\gamma,R^+\za\rho Y\ya) = R^+\za\rho Y\ya/R^+\za\Lambda_{\rho},I_{\rho},+\ya$.
 So we have the following lemma.
\begin{lem}\label{trivial representation-I}
  Keep notations as above.
  \begin{enumerate}
  \item The inclusion $R^+\za\rho Y\ya\hookrightarrow \Rinfp\za\rho Y\ya$ identifies $\RGamma(\Gamma, R^+\za\rho Y\ya)$ with a direct summand of $\RGamma(\Zp\gamma, \Rinfp\za\rho Y\ya)$ whose complement is concentrated in degree $1$ and is killed by $\zeta_p-1$.

  \item The $\rH^0(\Gamma, R^+\za\rho Y\ya) = R^+$ is independent of $\rho$.

  \item The $\rH^1(\Gamma, R^+\za\rho Y\ya) = R^+\za\rho Y\ya/R^+\za\Lambda_{\rho},I_{\rho},+\ya$ is the derived $p$-adic completion of $\oplus_{i\geq 0}R^+/(i+1)\rho R^+$.
  \end{enumerate}
\end{lem}
\begin{proof}
  It remains to compute $\rH^0(\Gamma,R^+\za\rho Y\ya T^{\alpha})$. 
  
  When $\alpha \neq 0$, assume $\sum_{n\geq 0}a_n\rho^nF_nT^{\alpha}$ is $\gamma$-invariant, then we have 
  \[\sum_{n\geq 0}\zeta^{\alpha}\epsilon_{\alpha}(a_n+\frac{\rho}{\epsilon_{\alpha}}(n+1)a_{n+1})\rho^nF_nT^{\alpha} = 0.\]
  This implies that for any $n\geq 0$ and any $m\geq 0$,
  \[a_n = (-1)^m\prod_{j=1}^m(\frac{\rho}{\epsilon_{\alpha}}(n+j))a_{n+m}.\]
  In particular, $\nu_p(a_n)\geq \sum_{j=1}^m\nu_p(n+j)$ for any $m\geq 0$. This forces $a_n = 0$ for any $n\geq 0$.
  
  When $\alpha = 0$, assume $\sum_{n\geq 0}a_n\rho^nF_n$ is $\gamma$-invariant, then we have 
  \[\sum_{n\geq 0}(n+1)\rho a_{n+1}\rho^nF_n = 0,\]
  which implies $a_n = 0$ for any $n\geq 1$. So we have $R^+\za\rho Y\ya^{\Gamma} = R^+$.
\end{proof}
Now we are able to handle the higher dimensional case.
\begin{lem}\label{trivial representation-II}
  Identify $\widehat S_{\infty,\rho}^+$ with $\Rinfp\za\rho Y_1,\dots,\rho Y_d\ya$.
  \begin{enumerate}
  \item The inclusion $R^+\za\rho \underline Y\ya\hookrightarrow \widehat S_{\infty,\rho}^+$ identifies $\RGamma(\Gamma, R^+\za\rho \underline Y\ya)$ with a direct summand of $\RGamma(\Gamma, \widehat S_{\infty,\rho}^+)$ whose complement is concentrated in degree $\geq 1$ and is killed by $\zeta_p-1$.
  
  \item For any $i\geq 0$, we have
  \[\rH^i(\Gamma, R^+\za\rho \underline Y\ya) = \wedge^i_{R^+}(\oplus_{j=1}^dR^+\za\rho Y_j\ya/R^+\za\Lambda_{\rho,j},I_{\rho},+\ya)\]
  for $\Lambda_{\rho,j}=\{\rho^n F_n(Y_j)\}$ and $I_{\rho} =\{\nu_p((n+1)\rho)\}_{n\geq 0}$.
  \end{enumerate}
\end{lem}
\begin{proof}
  Note that $\RGamma(\Gamma,\Rinfp\za\rho Y_1,\dots, \rho Y_d\ya)$ is presented by the Koszul complex
\[\rK(\Rinfp\za\rho Y_1,\dots, \rho Y_d\ya;\gamma_1-1,\dots,\gamma_d-1)\simeq\rK(\Rinfp\za\rho Y_1\ya;\gamma_1-1)\widehat \otimes^L_{\Rinfp}\cdots\widehat \otimes^L_{\Rinfp}\rK(\Rinfp\za\rho Y_d\ya;\gamma_d-1).\]
Since $R^+\za\rho Y_j\ya/R^+\za\Lambda_{\rho,j},I_{\rho},+\ya$ is already derived $p$-complete, the lemma follows from Lemma \ref{trivial representation-I} directly.
\end{proof}
\begin{prop}\label{trivial representation}
  \begin{enumerate}
      \item $(S_{\infty}^{\dagger,+})^{\Gamma} = R^+$;
      
      \item For any $i\geq 1$, $\rH^i(\Gamma,S_{\infty}^{\dagger,+})$ is $p^{\infty}$-torsion.
  \end{enumerate}
\end{prop}
\begin{proof}
  We only need to show for any $i\geq 1$,
  \[\varinjlim_{\nu_p(\rho)>\nu_p(\rho_k)}\rH^i(\Gamma,\widehat S_{\infty,\rho}^+)\]
  is $p^{\infty}$-torsion. However, by Lemma \ref{trivial representation-II}, this follows from a similar argument as in the proof of Corollary \ref{overconvergent resolution}.
\end{proof}

\subsection{Calculation in general case}
 Now, by virtues of Theorem \ref{Strong Decompletion}, we may assume that $M$ is an $a$-small representation of $\Gamma$ over $R^+$. Let $e_1,\dots,e_l$ be an $R^+$-basis of $M$ and $A_j$ be the matrix of $\gamma_j$ with respect to the chosen basis for all $1\leq j\leq d$; that is,
 \[\gamma_j(e_1,\dots,e_l) = (e_1,\dots,e_l)A_j.\]
 Put $B_j = A_j-I$. It is the matrix of $\gamma_j-1$ and has $p$-adic valuation $\nu_p(B_j)\geq a+\nu_p(\rho_k)$ by $a$-smallness of $M$. Similar to the trivial representation case, we are reduced to computing $\RGamma(\Gamma,M\otimes_{R^+}\Rinfp\za\rho Y_1,\dots,\rho Y_d\ya)$. 
 Note that we still have a $\Gamma$-equivariant decomposition 
 \[M\otimes_{R^+}\Rinfp\za\rho Y_1,\dots,\rho Y_d\ya=\widehat \oplus_{\underline \alpha\in(\bN[\frac{1}{p}]\cap[0,1))^d}M\otimes_{R^+}R^+\za\rho Y_1,\dots,\rho Y_d\ya\underline T^{\underline \alpha},\]
 where $\underline T^{\underline \alpha}$ denotes $T_1^{\alpha_1}\cdots T_d^{\alpha_d}$ for any $\underline \alpha = (\alpha_1,\dots,\alpha_d)$.
 
 Assume $\underline \alpha\neq 0$ at first. Without loss of generality, we assume $\alpha_d\neq 0$. Note that $\{e_{i,n}:= e_i\rho^nF_n(Y_d)\underline T^{\underline \alpha}\}_{1\leq i\leq l,n\geq 0}$ is a set of topological basis of $M\otimes_{R^+}R^+\za\rho Y_1,\dots,\rho Y_d\ya\underline T^{\underline \alpha}$ over $R^+\za\rho Y_1,\dots,\rho Y_{d-1}\ya$. We have
 \[(\gamma_d-1)(e_{1,n},\dots,e_{l,n}) = \zeta^{\alpha_d}\epsilon_{\alpha_d}((e_{1,n},\dots,e_{l,n})\cdot(\epsilon_{\alpha_d}^{-1}B_d+I)+(e_{1,n-1},\dots,e_{l,n-1})\cdot n\frac{\rho}{\epsilon_{\alpha_d}}A_d).\]
 Similar to the trivial representation case, using Proposition \ref{basis of free modules II}, we deduce that 
 \[\RGamma(\Zp\gamma_d, M\otimes_{R^+}R^+\za\rho Y_1,\dots,\rho Y_d\ya\underline T^{\underline \alpha})\simeq  M\otimes_{R^+}R^+\za\rho Y_1,\dots,\rho Y_d\ya\underline T^{\underline \alpha}/\epsilon_{\alpha_d}[-1].\]
 Using the Hochschild-Serre spectral sequence, we have the following lemma.
 \begin{lem}\label{error term}
   Assume $\underline \alpha\neq 0$. Then the complex
   $\RGamma(\Gamma, M\otimes_{R^+}R^+\za\rho Y_1,\dots,\rho Y_d\ya\underline T^{\underline \alpha})$
   is concentrated in positive degrees and is killed by $\zeta_p-1$.
 \end{lem}
 
 Now, we focus on the $\underline \alpha = 0$ case and prove the following proposition.
 \begin{prop}\label{principal term}
  Keep notations as above. Assume $\nu_p(\rho)<a+\nu_p(\rho_k)-r$. Define 
  \[H(M):=(M\otimes_{R^+}R^+\za\rho Y_1,\dots,\rho Y_d\ya)^{\Gamma},\]
  then the following assertions are true:
\begin{enumerate}
  \item The $H(M)$ is a finite free $R^+$-module of rank $l$ and is independent of the choice of $\rho$. More precisely, if we define
  \begin{equation*}
    (h_1, \dots, h_l) = (e_1, \dots, e_l)\sum_{n_1,\cdots n_d\geq 0}\prod_{i=1}^d\frac{(-A_i^{-1}B_i)^{n_i}}{n_i!}F_{n_i}(Y_i),
  \end{equation*}
  then $h_1,\dots,h_l$ is an $R^+$-basis of $H(M)$.

  \item The inclusion $H(M)\hookrightarrow M\otimes_{R^+}R^+\za\rho Y_1,\dots,\rho Y_d\ya$ induces a $\Gamma$-equivariant isomorphism
  \[H(M)\otimes_{R^+}R^+\za\rho Y_1,\dots,\rho Y_d\ya\cong M\otimes_{R^+}R^+\za\rho Y_1,\dots,\rho Y_d\ya.\]
 \end{enumerate}
\end{prop}
 \begin{proof}
 We first consider the $d=1$ case. In this case, $\Gamma = \Zp\gamma$ acts on $R^+\za\rho Y\ya$ via $\gamma(Y)=Y+1$. Let $e_1,\dots,e_l$ be a basis of $M$ and $A$ be the matrix of $\gamma$ associated to the chosen basis. Put $B=A-I$ and then $\nu_p(B)\geq a+\nu_p(\rho_k)>\nu_p(\rho)+r$. Note that $\{\rho^n F_n(Y)\}_{n\geq 0}$ is a set of topological basis of $R^+\za\rho Y\ya$. 
 \begin{enumerate}
   \item Assume $x = \sum_{n\geq 0}\underline eX_n\rho^nF_n(Y)\in H(M)$, where $X_n\in(R^+)^l$ for any $n\geq 0$ and $\underline e$ denotes $(e_1,\dots,e_l)$. Since $\gamma(x) = x$, we deduce that for any $n\geq 0$,
   \[BX_n = -(n+1)\rho AX_{n+1}.\]
   In other words, we have 
   \[X_{n} = \frac{-A^{-1}B}{n\rho}X_{n-1} = \frac{(-A^{-1}B)^n}{\rho^nn!}X_0.\] Note that $\nu_p(\frac{(A^{-1}B)^n}{\rho^nn!})\geq (a+\nu_p(\rho_k)-r-\nu_p(\rho)))n$. So we get $\frac{(A^{-1}B)^n}{\rho^nn!}\in \rM_l(R^+)$ and hence $X_n$ is uniquely determined by $X_0$. In particular, we have
   \begin{equation}\label{form}
       x = \underline e\sum_{n\geq 0} \frac{(-A^{-1}B)^n}{\rho^nn!}\rho^nF_n(Y)X_0 = \underline e\sum_{n\geq 0} \frac{(-A^{-1}B)^n}{n!}F_n(Y)X_0.
   \end{equation}
   Conversely, any $x\in M\otimes_{R^+}R^+\za\rho Y_1,\dots,\rho Y_d\ya$ which is of the form (\ref{form}) for some $X_0\in (R^+)^l$ is $\gamma$-invariant. So we are done.
   
   \item From the proof of (1), we see that $\sum_{n\geq 0} \frac{(-A^{-1}B)^n}{\rho^nn!}\rho^nF_n(Y)\in\GL_l(R^+\za\rho Y\ya)$. Thus $h_i$'s form an $R^+\za\rho Y\ya$-basis of $M\otimes_{R^+}R^+\za\rho Y\ya$ as desired.
   \end{enumerate}
   
   Now, we handle the case for any $d\geq 1$. By what we have proved and by iterating, we get 
   \begin{equation*}
   \begin{split}
       &\underline e (R^+\za\rho Y_1,\dots,\rho Y_d\ya)^l\\
       = &\underline e\sum_{n_d\geq 0}\frac{(-A_d^{-1}B_d)^{n_d}}{n_d!}F_{n_d}(Y_d)(R^+\za\rho Y_1,\dots,\rho Y_d\ya)^l\\
        = &\underline e\sum_{n_{d-1},n_d\geq 0}\frac{(-A_{d-1}^{-1}B_{d-1})^{n_{d-1}}}{n_{d-1}!}F_{n_{d-1}}(Y_{d-1})\frac{(-A_d^{-1}B_d)^{n_d}}{n_d!}F_{n_d}(Y_d)(R^+\za\rho Y_1,\dots,\rho Y_d\ya)^l\\
        =&\cdots\\
        =& \underline e\sum_{n_1,\cdots n_d\geq 0}\prod_{i=1}^d\frac{(-A_i^{-1}B_i)^{n_i}}{n_i!}F_{n_i}(Y_i)(R^+\za\rho Y_1,\dots,\rho Y_d\ya)^l.
    \end{split}
   \end{equation*}
   Since $\underline e\sum_{n_1,\cdots n_d\geq 0}\prod_{i=1}^d\frac{(-A_i^{-1}B_i)^{n_i}}{n_i!}F_{n_i}(Y_i)$ forms a $\Gamma$-invariant basis, the result follows from that $(R^+\za\rho Y_1,\dots,\rho Y_d\ya)^{\Gamma}=R^+$.
 \end{proof}
 \begin{rmk}\label{simple form}
   Note that if $\nu_p(z)>r$, then $(1+z)^Y = \sum_{n\geq 0}\frac{z^n}{n!}F_n(Y)$. Therefore, for $M$ and $\rho$ as above, as $\nu_p(A_i^{-1}B_j)\geq a>r$, the operator $\prod_{i=1}^d\gamma_i^{-Y_i}$, whose matrix is given by $\sum_{n_1,\cdots n_d\geq 0}\prod_{i=1}^d\frac{(-A_i^{-1}B_i)^{n_i}}{n_i!}F_{n_i}(Y_i)$, is well-defined on $M\otimes_{R^+}R^+\za\rho Y_1,\dots,\rho Y_d\ya$. Then the above proposition says that we have $H(M) = \prod_{i=1}^d\gamma_i^{-Y_i}M$. Since $\log(1+z)(1+z)^Y = \sum_{n\geq 0}\frac{z^n}{n!}F'_n(Y)$ when $\nu_p(z)>r$, for any $\underline e\vec m\in M$ with $\vec m\in(R^+)^l$ and $1\leq j\leq d$, we get
   \begin{equation*}
       \begin{split}
           \frac{\partial}{\partial Y_j}(\prod_{i=1}^d\gamma_i^{-Y_i}\underline e\vec m) &= \underline e\frac{\partial}{\partial Y_j}(\sum_{n_1,\cdots n_d\geq 0}\prod_{i=1}^d\frac{(-A_i^{-1}B_i)^{n_i}}{n_i!}F_{n_i}(Y_i)\vec m)\\
           & = \underline e\sum_{n_1,\dots, n_d\geq 0}\frac{(-A_j^{-1}B_j)^{n_j}}{n_j!}F_{n_j}'(Y_j)\prod_{1\leq i\leq d,i\neq j}\frac{(-A_i^{-1}B_i)^{n_i}}{n_i!}F_{n_i}(Y_i)\vec m\\
           &= \underline e(-\log(A_j)\sum_{n_1,\cdots n_d\geq 0}\prod_{i=1}^d\frac{(-A_i^{-1}B_i)^{n_i}}{n_i!}F_{n_i}(Y_i)\vec m)\\
           &=-\log\gamma_j\prod_{i=1}^d\gamma_i^{-Y_i}\underline e\vec m.
       \end{split}
   \end{equation*}
 \end{rmk}
 \begin{cor}\label{local representation to Higgs}
   Keep notations as above.
   \begin{enumerate}
       \item Denote by $\theta_{H(M)}$ the restriction of $\Theta$ to $H(M)$. Then $(H(M),\theta_{H(M)})$ is an $a$-small Higgs module. Moreover, $\theta_{H(M)} = \sum_{i=1}^d-\log\gamma_i\otimes\frac{\dlog T_i}{t}$.
       
       \item The inclusion $H(M)\to M\otimes_{R^+}S_{\infty}^{\dagger,+}$ induces a $\Gamma$-equivariant isomorphism
       \[H(M)\otimes_{R^+}S_{\infty}^{\dagger,+}\cong M\otimes_{R^+}S_{\infty}^{\dagger,+}\]
       and identifies the corresponding Higgs complexes
       \[{\rm HIG}(H(M)\otimes_{R^+}S_{\infty}^{\dagger,+},\Theta_{H(M)})\cong {\rm HIG}(M\otimes_{R^+}S_{\infty}^{\dagger,+},\Theta_M).\]
   \end{enumerate}
 \end{cor}
 \begin{proof}
 \begin{enumerate}
   \item Since $\Theta = \sum_{i=1}^d\frac{\partial}{\partial Y_i}\otimes\frac{\dlog T_i}{t}$, the ``moreover'' part follows from Remark \ref{simple form}. Since $\nu_p(B_i)\geq a+\nu_p(\rho_k)$ for all $j$ and $\log\gamma_j = -\sum_{n\geq 1}\frac{(-B_j)^n}{n}$, we see the $a$-smallness of $(H(M),\theta_{H(M)})$ as $\nu_p(\frac{B_i^n}{n})\geq a+\nu_p(\rho_k)$ for all $n$.
   
   \item This follows from Proposition \ref{principal term} (2) and the definition of $\theta_{H(M)}$.
 \end{enumerate}
 \end{proof}
 
 We have seen how to achieve an $a$-small Higgs module from an $a$-small representation. It remains to construct an $a$-small representation of $\Gamma$ from an $a$-small Higgs module.
 
\begin{prop}\label{local Higgs to representation}
  Assume $a>r$. Let $(H,\theta_H)$ be an $a$-small Higgs module of rank $l$ over $R^+$. Put $M = (H\otimes_{R^+}S_{\infty}^{\dagger,+})^{\Theta_H=0}$.
  \begin{enumerate}
      \item The restricted $\Gamma$-action on $M$ makes it an $a$-small $\Rinfp$-representation of rank $l$. Moreover, if $\theta_H = \sum_{i=1}^d\theta_i\otimes\frac{\dlog T_i}{t}$, then $\gamma_i$ acts on $M$ via $\exp(-\theta_i)$.
      
      \item The inclusion $M\hookrightarrow H\otimes_{R^+}S_{\infty}^{\dagger,+}$ induces a $\Gamma$-equivariant isomorphism
      \[M\otimes_{\Rinfp}S_{\infty}^{\dagger,+}\cong H\otimes_{R^+}S_{\infty}^{\dagger,+}\]
      and identifies the corresponding Higgs complexes
      \[{\rm HIG}(M\otimes_{\Rinfp}S_{\infty}^{\dagger,+},\Theta_{M})\cong {\rm HIG}(H\otimes_{R^+}S_{\infty}^{\dagger,+},\Theta_H).\]
  \end{enumerate}
\end{prop}
\begin{proof}
\begin{enumerate}
  \item The argument is similar to the proof of Proposition \ref{principal term}.
  
  Assume $\rho\in\rho_k\frakm_{\Cp}$ such that $a+\nu_p(\rho_k)>\nu_p(\rho)+r$.
  Let $e_1,\dots,e_l$ be an $R^+$-basis of $H$. We claim that $M=(H\otimes_{R^+}\Rinfp\za\rho Y_1,\dots,\rho Y_d\ya)^{\Theta_H=0}$. 
  
  In fact, if $\vec G = (G_1,\cdots,G_l)^t\in (\Rinfp\za\rho Y_1,\dots,\rho Y_d\ya)^l$ such that $m = \sum_{i=1}^le_iG_i\in M$, then we see that for any $1\leq i\leq d$,
   \[\theta_i\vec G + \frac{\partial\vec G}{\partial Y_i} = 0.\]
   This forces $\vec G=\prod_{i=1}^d\exp(-\theta_iY_i)\vec a$ for some $\vec a\in (\Rinfp)^l$. Since $\nu_p(\theta_j)\geq a+\nu_p(\rho_k)$, the matrix $\prod_{i=1}^d\exp(-\theta_iY_i)$ is well-defined in $\GL_l(\Rinfp\za\rho Y_1,\dots,\rho Y_d\ya)$. This shows that $M$ is finite free of rank $l$ and is independent of the choice of $\rho$. 
   
   Note that $\gamma_i(Y_j) = Y_j+\delta_{ij}$. We see $\gamma_i$ acts on $M$ via $\exp(-\theta_i)$. Since $\nu_p(\theta_i)\geq a+\nu_p(\rho_k)$, using $\exp(-\theta_iY_i) = \sum_{n\geq 0}\frac{(-\theta_i)^n}{n!}Y_i^n$, we deduce that $M$ is $a$-small.
   
   \item The (2) follows from the fact that $\prod_{i=1}^d\exp(-\theta_iY_i)\in \GL_l(\Rinfp\za\rho Y_1,\dots,\rho Y_d\ya)$ and the definition of $\Gamma$-action on $M$.
  \end{enumerate}
\end{proof}
 
 Finally, we complete the proof of Theorem \ref{local Simpson}.
 \begin{proof}(of Theorem \ref{local Simpson})
 
   The (1) was given in Corollary \ref{local representation to Higgs}. The (2) was proved in Proposition \ref{local Higgs to representation}. The equivalence part of (3) follows from Corollary \ref{local representation to Higgs} (2) (as $\theta_i$'s act via $ -\log\gamma_i$'s) together with Proposition \ref{local Higgs to representation} (2) (as  $\gamma_i$'s act via $\exp(-\theta_i)$'s). An elementary linear algebra shows that the equivalence preserves tensor products and dualities. So we only need to prove the ``moreover'' part of (4). 
   
   Let $M$ be an $a$-small representation of $\Gamma$ over $\Rinfp$ and $(H,\theta_H)$ be the corresponding Higgs module over $R^+$. By Corollary \ref{overconvergent resolution}, we have quasi-isomorphisms of complexes over $\Rinf$,
   \[M[\frac{1}{p}]\xrightarrow{\simeq} {\rm HIG}(M\otimes_{\Rinfp}S_{\infty}^{\dagger},\Theta_M)\simeq {\rm HIG}(H\otimes_{R^+}S_{\infty}^{\dagger},\Theta_H).\]
   Applying $\RGamma(\Gamma,\cdot)$, we get a quasi-isomorphism
   \[\RGamma(\Gamma,M[\frac{1}{p}])\to \RGamma(\Gamma,{\rm HIG}(H\otimes_{R^+}S_{\infty}^{\dagger},\Theta_H)).\]
   However, it follows from Proposition \ref{trivial representation} that \[\RGamma(\Gamma,S_{\infty}^{\dagger})\simeq R[0].\] 
   So we get 
   \[\RGamma(\Gamma,{\rm HIG}(H\otimes_{R^+}S_{\infty}^{\dagger},\Theta_H))\simeq {\rm HIG}(H[\frac{1}{p}],\theta_H).\]
   Therefore, we conclude the desired quasi-isomorphism
   \[\RGamma(\Gamma,M[\frac{1}{p}])\simeq{\rm HIG}(H[\frac{1}{p}],\theta_H). \]
 \end{proof}
   Finally, it is worth pointing out that all results in Theorem \ref{local Simpson} still hold for $\widehat S_{\infty,\rho_k}^+$ instead of $S_{\infty}^{\dagger,+}$ except the ``moreover'' part of (4) for the sake that ${\rm HIG}(\widehat S_{\infty,\rho_k}^+[\frac{1}{p}],\Theta)\neq \Rinf[0]$ and $\RGamma(\Gamma,\widehat S_{\infty,\rho_k}^+[\frac{1}{p}])\neq R[0]$. For the further use, we give the following proposition.
   
   \begin{prop}\label{Prop-Bigger ring}
     Keep notations as in Theorem \ref{local Simpson}.
     \begin{enumerate}
       \item  Let $M$ be an $a$-small $\Rinfp$-representation of $\Gamma$ of rank $l$. Then $H(M)=(M\otimes_{\Rinfp}\widehat S_{\infty,\rho_k}^+)^{\Gamma}$ and $\theta_{H(M)}$ is the restriction of $\Theta_M$ to $H(M)$.
       
       \item  Let $(H,\theta_H)$ be an $a$-small Higgs module of rank $l$ over $R^+$. Then $M(H,\theta_H) = (H\otimes_{R^+}\widehat S_{\infty,\rho_k}^+)^{\Theta_H=0}$.
       
       \item Let $M$ be an $a$-small $\Rinfp$-representation of $\Gamma$ and $(H,\theta_H)$ be the corresponding Higgs module. Then there is a canonical $\Gamma$-equivariant isomorphism of Higgs complexes
       \[{\rm HIG}(H\otimes_{R^+}\widehat S_{\infty,\rho_k}^+,\Theta_{H})\to {\rm HIG}(M\otimes_{\Rinfp}\widehat S_{\infty,\rho_k}^+,\Theta_M).\]
     \end{enumerate}
   \end{prop}
   \begin{proof}
         By Corollary \ref{complete Hyodo}, we have a $\Gamma$-equivariant decomposition
         \[\widehat S_{\infty,\rho_k}^+ = \widehat \oplus_{\underline \alpha\in (\bN\cap[0,1))^d}R^+\za\rho_k Y_1,\dots,\rho_kY_d\ya\underline T^{\underline \alpha}.\]
         Let $N$ be the $a$-small $R^+$-representation of $\Gamma$ corresponding to $M$ in the sense of Theorem \ref{Strong Decompletion}. Then $M=N\otimes_{R^+}\Rinfp$.
     \begin{enumerate}
         \item Thanks to Lemma \ref{error term}, we have
         \[(M\otimes_{\Rinfp}\widehat S_{\infty,\rho_k}^+)^{\Gamma} = (N\otimes_{R^+}R^+\za\rho_k Y_1,\dots,\rho_kY_d\ya)^{\Gamma}.\]
         Since $a>r$, it is automatic that $\nu_p(\rho_k)<a+\nu_p(\rho_k)-r$. So (1) is a consequence of Proposition \ref{principal term}.
         
         \item This follows from the proof of Proposition \ref{local Higgs to representation} (1) directly (for the sake that $\nu_p(\rho_k)<a+\nu_p(\rho_k)-r$).
         
         \item This follows from (1), (2) and Theorem \ref{local Simpson} (4) via the base-change along $S_{\infty}^{\dagger,+}\to \widehat S_{\infty,\rho_k}^{+}$.
     \end{enumerate}
   \end{proof}

\section{A $p$-adic Simpson correspondence}\label{Sec 5} 

\subsection{Statement and preliminaries}
  Now, we want to globalise the local Simpson correspondence established in the last section for a liftable smooth formal scheme $\frakX$. We fix such an $\frakX$ together with an $A_2$-lifting $\widetilde \frakX$. Then we have the corresponding integral Faltings' extension $\calE^+$ and overconvergent period sheaf $\OC^{\dagger,+}$. Let $X$ be the rigid analytic generic fibre of $\frakX$ and $\nu:X_{\proet}\to\frakX_{\et}$ be the projection of sites. Throughout this section, we assume $r = \frac{1}{p-1}$.

\begin{dfn}\label{Dfn-small generalised representation}
  Assume $a\geq r$. By an {\bf $a$-small generalised representation} of rank $l$ on $X_{\proet}$, we mean a sheaf $\calL$ of locally finite free $\OX$-modules of rank $l$ which admits a $p$-complete sub-$\OXp$-module $\calL^+$ such that there is an \'etale covering $\{\frakX_i\to\frakX\}_{i\in I}$ and rationals $b_i>b>a$ such that for any $i$, 
      \[(\calL^+/p^{b_i+\nu_p(\rho_k)})^{\rm al}_{\mid X_i} \cong ((\OXp/p^{b_i+\nu_p(\rho_k)})^l)^{\rm al}_{\mid X_i}\]
      is an isomorphism of $(\OX^{+\rm al}/p^{b_i+\nu_p(\rho_k)})_{\mid X_i}$-modules, where $\OX^{+\rm al}$ is the almost integral structure sheaf\footnote{This is the presheaf on $X_{\proet}$ sending each affinoid perfectoid space $U=\Spa(R,R^+)$ to the almost $\calO_{\Cp}$-module $R^{+{\rm al}}$ in the sense of \cite[Section 4]{Sch1}. Since $X_{\proet}$ admits a basis of affinoid perfectoid spaces, the proof of \cite[Proposition 7.13]{Sch1} shows that $\OX^{+{\rm al}}$ is a sheaf.} and $X_i$ denotes the rigid analytic generic fibre of $\frakX_i$.
\end{dfn}
\begin{dfn}\label{Dfn-small Higgs bundle}
   Assume $a\geq r$. By an {\bf $a$-small Higgs bundle} of rank $l$ on $\frakX_{\et}$, we mean a sheaf $\calH$ of locally finite free $\calO_{\frakX}[\frac{1}{p}]$-modules of rank $l$ together with an $\calO_{\frakX}[\frac{1}{p}]$-linear operator $\theta_{\calH}:\calH\to \calH\otimes_{\calO_{\frakX}}\widehat \Omega^1_{\frakX}(-1)$ satisfying $\theta_{\calH}\wedge\theta_{\calH} = 0$ such that it admits a $\theta_{\calH}$-preserving $\calO_{\frakX}$-lattice $\calH^+$ (i.e. $\calH^+\subset \calH$ is a subsheaf of locally free $\calO_{\frakX}$-modules with $\calH^+[\frac{1}{p}]=\calH$) satisfying the condition
   \[\theta_{\calH}(\calH^+)\subset p^{b+\nu_p(\rho_k)}\calH^+\otimes_{\calO_{\frakX}}\widehat \Omega^1_{\frakX}(-1)\]
   for some $b>a$.
\end{dfn}
For any $a$-small generalised representation, define 
\[
\Theta_{\calL} = \id_{\calL}\otimes\Theta:\calL\otimes_{\OX}\OC^{\dagger}\to \calL\otimes_{\OX}\OC^{\dagger}\otimes_{\calO_{\frakX}}\widehat \calO_{\frakX}^1(-1).
\]
Then $\Theta_{\calL}$ is a Higgs field on $\calL\otimes_{\OX}\OC^{\dagger}$. Denote the induced Higgs complex by ${\rm HIG(\calL\otimes_{\OX}\OC^{\dagger},\Theta_{\calL})}$. For any $a$-small Higgs field $(\calH,\theta_{\calH})$, put
\[
\Theta_{\calH} = \theta_{\calH}\otimes\id+\id_{\calH}\otimes\Theta:\calH\otimes_{\calO_{\frakX}}\OC^{\dagger}\to \calH\otimes_{\calO_{\frakX}}\OC^{\dagger}\otimes_{\calO_{\frakX}}\widehat \Omega^1_{\frakX}(-1).
\]
Then $\Theta_{\calH}$ is a Higgs field on $\calH\otimes_{\calO_{\frakX}}\OC^{\dagger}$. Denote the induced Higgs complex by ${\rm HIG}(\calH\otimes_{\calO_{\frakX}}\OC^{\dagger},\Theta_{\calH})$.
Then our main theorem is the following $p$-adic Simpson correspondence.
\begin{thm}[$p$-adic Simpson correspondence]\label{p-adic Simpson}
  Keep notations as above.
  \begin{enumerate}
       \item For any $a$-small generalised representation $\calL$ of rank $l$ on $X_{\proet}$, $\rR\nu_*(\calL\otimes_{\OX}\OC^{\dagger})$ is discrete. Denote $\calH(\calL):=\nu_*(\calL\otimes_{\OX}\OC^{\dagger})$ and $\theta_{\calH(\calL)} = \nu_*\Theta_{\calL}$. Then $(\calH(\calL),\theta_{\calH(\calL)})$ is an $a$-small Higgs bundle of rank $l$.
      
      \item For any $a$-small Higgs bundle $(\calH,\theta_{\calH})$ of rank $l$ on $\frakX_{\et}$, put 
      \[
      \calL(\calH,\theta_{\calH}) = (\calH\otimes_{\calO_{\frakX}}\OC^{\dagger})^{\Theta_{\calH}=0}.
      \]
      Then $\calL(\calH)$ is an $a$-small generalised representation of rank $l$.
      
      \item The functor $\calL\mapsto(\calH(\calL),\theta_{\calH(\calL)})$ induces an equivalence from the category of $a$-small generalised representations to the category of $a$-small Higgs bundles, whose quasi-inverse is given by $(\calH,\theta_{\calH})\mapsto \calL(\calH,\theta_{\calH})$. The equivalence preserves tensor products and dualities and identifies the Higgs complexes
      \[
      {\rm HIG}(\calL\otimes_{\OX}\OC^{\dagger},\Theta_{\calL})\simeq {\rm HIG}(\calH(\calL)\otimes_{\calO_{\frakX}}\OC^{\dagger},\Theta_{\calH(\calL)}).
      \]
      
      \item Let $\calL$ be an $a$-small generalised representation with associated Higgs bundle $(\calH,\theta_{\calH})$. Then there is a canonical quasi-isomorphism
      \[\rR\nu_*(\calL)\simeq{\rm HIG}(\calH,\theta_{\calH}),\]
      where ${\rm HIG}(\calH,\theta_{\calH})$ is the Higgs complex induced by $(\calH,\theta_{\calH})$. In particular, $\rR\nu_*(\calL)$ is a perfect complex of $\calO_{\frakX}[\frac{1}{p}]$-modules concentrated in degree $[0,d]$, where $d$ denotes the dimension of $\frakX$ relative to $\calO_{\Cp}$.
      
      \item Assume $f:\frakX\to\frakY$ is a smooth morphism between liftable smooth formal schemes over $\calO_{\Cp}$. Let $\widetilde \frakX$ and $\widetilde \frakY$ be the fixed $A_2$-liftings of $\frakX$ and $\frakY$, respectively. Assume $f$ lifts to an $A_2$-morphism $\widetilde f: \widetilde \frakX\to\widetilde \frakY$, then the equivalence in (3) is compatible with the pull-back along $f$.
  \end{enumerate}
\end{thm}
\begin{rmk}\label{generalization}
   Assume $\calL$ is a sheaf of locally free $\OX$-modules which becomes $a$-small after a finite \'etale base-change $f:\frakY\to \frakX$. By \'etale descent, the $\rR\nu_*(\calL\otimes_{\OX}\OC^{\dagger})$ is well-defined and discrete. The $\nu_*(\calL\otimes_{\OX}\OC^{\dagger})$ is a Higgs bundle which becomes $a$-small Higgs bundle via pull-back along $f$. Conversely, if $(\calH,\theta_{\calH})$ is a Higgs bundle on $\frakX$ which becomes $a$-small after taking pull-back along a finite \'etale morphism $f$, by pro-\'etale descent for $\OX$-bundles, $(\calH\otimes_{\calO_{\frakX}}\OC^{\dagger})^{\Theta_{\calH}=0}$ is a well-defined $\OX$-bundle. Also, it becomes $a$-small via the pull-back along $f$. Therefore, one can establish a $p$-adic Simpson correspondence in this case.
\end{rmk}
\begin{rmk}\label{compare with Faltings}
  Assume $\frakX$ comes from a smooth formal scheme $\frakX_0$ over $\Zp$ and admits an $A_2$-lifting $\widetilde \frakX$. Note that Faltings used Breuil-Kisin twist to define Higgs fields \cite[Definition 2]{Fal2}  while we use Tate twist, so our smallness conditions on Higgs fields differ from his by a multiplication of $(\zeta_p-1)$. By Proposition \ref{Compare obstruction}, after choosing a covering $\{\frakX_i\to\frakX\}_{i\in I}$, the cocycle $\{\theta_{ij}\}_{i,j\in I}$ corresponding to the intergal Faltings' extension is exactly the one used in \cite[Section 4]{Fal2}. Note that locally we define Higgs fields by $\theta = -\log\gamma$ (Corollary \ref{local representation to Higgs}) while Faltings defined $\theta = \log \gamma$ (\cite[Remark(ii)]{Fal2}). So our construction is compatible with \cite{Fal2} up to a sign on Higgs fields.
\end{rmk}
\begin{rmk}\label{compare with Liu-Zhu}
  If $\frakX$ comes from a smooth formal scheme $\frakX_0$ over $\calO_k$ and $\widetilde \frakX$ is the base-change of $\frakX_0$ along $\calO_k\to A_2$. Let $\OC^{\dagger}$ be the associated overconvergent period sheaf. By its construction, there is a natural inclusion $\OC\hookrightarrow\OC^{\dagger}$. Now assume $\bL$ is a $\Zp$-local system on $\frakX_{\et}$ and $\calL = \bL\otimes_{\Zp}\OX$ is the corresponding $\OX$-bundle on $X_{\proet}$. Since the resulting Higgs field is nilpotent by  \cite[Theorem 2.1]{LZ}, it can be seen from the proof of Theorem \ref{p-adic Simpson} that the morphism
  \[\nu_*(\calL\otimes_{\OX}\OC)\to\nu_*(\calL\otimes_{\OX}\OC^{\dagger})\]
  is an isomorphism. So our construction is compatible with the work of \cite{LZ} in this case.
\end{rmk}

We do some preparations before proving Theorem \ref{p-adic Simpson}.

\begin{lem}\label{almost purity}
  Let $U\in X_{\proet}$ be affinoid perfectoid and $\calM^+$ be a sheaf of $p$-torison free $\OXp$-modules satisfying one of the following conditions:
  
      $(a)$~ $\calM^+_{\mid U}$ is a sheaf of free $\widehat \calO_{X\mid U}^+$-modules.
      
      $(b)$~ $\calM^+$ is $p$-complete and there is an almost isomorphism 
      \[(\calM_{\mid U}^+/p^c)^{\rm al}\cong ((\widehat \calO^+_{X\mid U}/p^c)^r)^{\rm al}\]
      for some $c>0$.
      
  Then the following assertions are true:
  \begin{enumerate}
  \item For any $i\geq 1$ and $a>0$, $H^i(U,\calM^+)^{\rm al}\cong H^i(U,\calM^+/p^{a})^{\rm al}=0$.
  
  \item For any $b>a>0$, the image of $(\calM^+/p^b)(U)$ in $(\calM^+/p^a)$ is $\calM^+(U)/p^a$.
  
  \item Put $\hat \calM^+ = \varprojlim_n\calM^+/p^n$. Then $\hat \calM^+(U) = \varprojlim_n\calM^+(U)/p^n$ and for any $i\geq 1$, $H^i(U,\hat \calM^+)^{\rm al} = 0$.
  \end{enumerate}
\end{lem}
\begin{proof}  
By \cite[Lemma 4.10]{Sch2}, both $(1)$ and $(2)$ hold for free $\OXp$-modules. So we only focus on $\calM^+$'s satisfying the second condition.
\begin{enumerate}
  \item It is enough to show that for any $i\geq 1$, $H^i(U,\calM^+)^{\rm al} = 0$. Granting this, the rest can be deduced from the long exact sequence induced by
  \[0\to\calM^+\xrightarrow{\times p^a}\calM^+\to\calM^+/p^{a}\to 0.\]
  
  Since $(\calM_{\mid U}^+/p^c)^{\rm al}\cong ((\widehat \calO^+_{X\mid U}/p^c)^r)^{\rm al}$, by \cite[Lemma 4.10(v)]{Sch2}, we deduce that 
 $H^i(U,\calM^+/p^c)^{\rm al} = 0$
  for any $i\geq 1$.
  Consider the exact sequence
  \[0\to\calM^+/p^c\xrightarrow{p^{(n-1)c}}\calM^+/p^{nc}\to\calM^+/p^{(n-1)c}\to 0.\]
  By induction on $n$, we see that for any $i\geq 1$,
  $H^i(U,\calM^+/p^{nc})^{\rm al} = 0$. Now, the desired result follows from \cite[Lemma 3.18]{Sch2}.
  
  \item Consider the commutative diagram
  \[\xymatrix@C=0.45cm{
    0\ar[r]&\calM^+\ar[r]^{p^b}\ar[d]^{\times p^{b-a}}&\calM^+\ar[r]\ar[d]&\calM^+/p^b\ar[r]\ar[d]&0\\
    0\ar[r]&\calM^+\ar[r]^{p^a}&\calM^+\ar[r]&\calM^+/p^a\ar[r]&0.
  }
  \]
  Then by (1), we get the following commutative diagram
  \[\xymatrix@C=0.45cm{
    0\ar[r]&\calM^+(U)/p^b\ar[r]\ar[d]&(\calM^+/p^b)(U)\ar[r]^{\delta_b}\ar[d]&H^1(U,\calM^+)\ar[r]\ar[d]^{\times p^{b-a}}&0\\
    0\ar[r]&\calM^+(U)/p^a\ar[r]&(\calM^+/p^a)(U)\ar[r]^{\delta_a}&H^1(U,\calM^+)\ar[r]&0.
  }
  \]
  Since the multiplication by $p^{b-a}$ is zero on $H^1(U,\calM^+)$, the image of $(\calM^+/p^b)(U)$ in $(\calM^+/p^a)(U)$ is contained in the kernel of $\delta_a$. In other words, $(\calM^+/p^b)(U)$ takes values in $\calM^+(U)/p^a$. Now, the result follows.
  
  \item When $\calM^+$ is $p$-complete, there is nothing to prove. Now, assume $\calM^+$ is a free $\OXp$-module. The first part follows from (2) and the second part follows from the same argument used in (1).
 \end{enumerate}
\end{proof}
\begin{rmk}
  In this paper, we say a module (or a sheaf of $\OXp$-modules) $M$ is $p$-complete, if $M\cong \rR\lim_nM\otimes_{\Zp}^L\Zp/p^n$. This is different from that $M = \lim_nM/p^n$ in general. However, as mentioned in the paragraph below \cite[Lemma 4.6]{BMS2}, if $M$ has bounded $p^{\infty}$-torsion; that is, $M[p^{\infty}] = M[p^N]$ for some $N\geq 0$, then saying $M$ is $p$-complete amounts to saying $M = \lim_nM/p^n$. Indeed, in this case, the pro-systems $\{M\otimes^L_{\Zp}\Zp/p^n\}_{n\geq 0}$ and $\{M/p^n\}_{n\geq 0}$ are pro-isomorphic. So we obtain that
  \[\rR\lim_nM\otimes_{\Zp}^L\Zp/p^n\simeq \rR\lim_n M/p^n.\]
\end{rmk}
\begin{lem}\label{take section}
  Assume $\frakX = \Spf(R^+)$ is small. Define $X_{\infty},\Rinfp$ as before. Let $\calL^+$ be a sheaf of $p$-complete and $p$-torsion free $\OXp$-modules such that 
  \[(\calL^+/p^{a})^{\rm al} \cong ((\OXp/p^{a})^l)^{\rm al}\]
  for some $a>0$. Put $M = \calL^+(X_{\infty})$,then 
  \begin{enumerate}
      \item $M$ is a finite free $\Rinfp$-module of rank $l$.
      
      \item For any $0<b<a$, there is a $\Gamma$-equivariant isomorphism $M/p^b\cong(\Rinfp/p^b)^l$.
  \end{enumerate}
\end{lem}
\begin{proof}
  By Lemma \ref{almost purity}, we have $\Gamma$-equivariant almost isomorphisms
  \begin{equation}\label{Equ-almost isomorphisms}
      M/p^{a}\xrightarrow{\approx}(\calL^+/p^a)(X_{\infty})\approx(\OXp/p^a)^l(X_{\infty})\xleftarrow{\approx}(\Rinfp/p^a)^l.
  \end{equation}
  In particular, we get an almost isomorphism $M/p^a\approx(\Rinfp/p^a)^l.$ Denote $e_1,\dots,e_l$ the standard basis of $(\Rinfp)^l$.
\begin{enumerate}
  \item As mentioned in the paragraph after \cite[Definition 2.2]{Sch2}, for any $\epsilon\in\bQ_{>0}$, one can find $\calO_{\Cp}$-morphisms
  \[f:M/p^a\to(\Rinfp/p^a)^l\]
  and 
  \[g:(\Rinfp/p^a)^l\to M/p^a\]
  such that $f\circ g = p^{\epsilon}$ and $g\circ f = p^{\epsilon}$. In particular, the image of $g$ is $p^{\epsilon}M/p^a$ and the kernel of $g$ is killed by $p^{\epsilon}$.
  
  For any $i$, choose $x_i\in M$ such that 
  \[x_i\equiv g(e_i) \mod p^aM.\]
  Then $x_i$'s generate 
  \[p^{\epsilon}M/p^a\cong M/p^{a-\epsilon}.\]
  We claim $x_i$'s are linear independent over $\Rinfp/p^{a-\epsilon}$. Granting this, we see $M/p^{a-\epsilon}$ is a finite free $\Rinfp/p^{a-\epsilon}$-module. Since $M$ is $p$-torsion free and $p$-complete by Lemma \ref{almost purity} (3), by choosing $\epsilon<a$, we deduce that $M$ is finite free of rank $l$ as desired.
  
  So we are reduced to proving the claim. Assume $\lambda_i\in \Rinfp$ such that $\sum_{i=1}^l\lambda_ix_i\in p^aM$, i.e. $g(\sum_{i=1}^l\lambda_ie_i)\in p^aM$. So $\sum_{i=1}^l\lambda_ie_i\in\Ker(g)$ and thus is killed by $p^{\epsilon}$. In other words, $p^{\epsilon}\sum_{i=1}^l\lambda_ie_i\in p^a(\Rinfp)^l$. This forces $\lambda_i\in p^{a-\epsilon}\Rinfp$ for any $i$. So we are done.
  
  \item By \cite[Proposition 4.4]{Sch1}, the almost isomorphism $M/p^a\approx (\Rinfp/p^a)^l$ induces an isomorphism 
  \[\iota: \frakm_{\Cp}\otimes_{\calO_{\Cp}}(\Rinfp/p^a)^l\rightarrow\frakm_{\Cp}\otimes_{\calO_{\Cp}}M/p^a.\]
  Since (\ref{Equ-almost isomorphisms}) is $\Gamma$-equivariant, so is $\iota$. Since $\frakm_{\Cp}$ is flat over $\calO_{\Cp}$, this amounts to a $\Gamma$-equivariant isomorphism
  \[h:(\frakm_{\Cp}\Rinfp/p^a\frakm_{\Cp}\Rinfp)^l\to\frakm_{\Cp}M/p^a\frakm_{\Cp}M.\]
  Now, for any $\epsilon>0$, choose $x_{i,\epsilon}\in \frakm_{\Cp}M$ such that for any $i$,
  \[x_{i,\epsilon}\equiv h(p^{\epsilon}e_i)\mod p^aM.\]
  Note that $x_{i,\epsilon}$ is unique modulo $p^aM$. So for $0<\epsilon'<\epsilon$, we have \[p^{\epsilon-\epsilon'}x_{i,\epsilon'}\equiv x_{i,\epsilon}\mod p^aM.\]
  Assume $\epsilon<a$, we see that $p^{\epsilon -\epsilon'}$ divides $x_{i,\epsilon}$ for any $\epsilon'$. By \cite[Lemma 8.10]{BMS1}, $R^+$ is a topologically free $\calO_{\Cp}$-module, then so is $\Rinfp$. As we have seen that $M$ is a finite free $\Rinfp$-module, it is also topologically free over $\calO_{\Cp}$. This forces that $x_{i,\epsilon}$ is divided by $p^{\epsilon}$. So we may assume $x_{i,\epsilon} = p^{\epsilon}y_{i,\epsilon}$ for some $y_{i,\epsilon}\in M$. By construction, $y_{i,\epsilon}$ is unique modulo $p^{a-\epsilon}M$.
 
 Now define $H_{\epsilon}:(\Rinfp/p^{a-\epsilon})^l\to M/p^{a-\epsilon}$ by sending $e_i$ to $y_{i,\epsilon}$. By construction of $H_{\epsilon}$, we see that it is the unique $\Rinfp$-morphism from $(\Rinfp/p^{a-\epsilon})^l$ to $M/p^{a-\epsilon}$ whose restriction to $(\frakm_{\Cp}\Rinfp/p^{a-\epsilon})^l$ coincides with $h$.
 
 We need to show $H_{\epsilon}$ is an isomorphism. However, since $M$ is also finite free, after interchanging $M$ and $(\Rinfp)^l$ and proceeding as above, we get a unique $G_{\epsilon}:  M/p^{a-\epsilon}\to(\Rinfp/p^{a-\epsilon})^l$, whose restriction to $\frakm_{\Cp}M/p^{a-\epsilon}$ coincides with $h^{-1}$. Now, the similar argument shows that $H_{\epsilon}\circ G_{\epsilon} =\id$ and 
 $G_{\epsilon} \circ H_{\epsilon} =\id$. So $H_{\epsilon}$ is an isomorphism.
 
 Finally, since $h$ is $\Gamma$-equivariant, by the uniqueness of $H_{\epsilon}$, we deduce that $H_{\epsilon}$ is also $\Gamma$-equivariant. Since $\epsilon$ is arbitrary, we are done.
 \end{enumerate}
\end{proof}
  The following corollary is a special case of Proposition \ref{take section}.
\begin{cor}\label{take section-II}
  Assume $\frakX = \Spf(R^+)$ is small affine.
  Let $\calL$ be an $a$-small generalised representation with a sub-$\OXp$-sheaf $\calL^+$ satisfying $(\calL^+/p^{b+\nu_p(\rho_k)})^{\rm al}\cong ((\OXp/p^{b+\nu_p(\rho_k)})^l)^{\rm al}$ for some $b>a$. Then $\calL^+(X_{\infty})$ is a $b'$-small $\Rinfp$-representation of $\Gamma$ for any $a<b'<b$.
\end{cor}

\begin{lem}\label{reduce to local case}
  Assume $\frakX = \Spf(R^+)$ is affine small. Let $\calL^+$ be a sheaf of $p$-complete and $p$-torsion free $\OXp$-modules such that 
  \[(\calL^+/p^c)^{\rm al} \cong ((\OXp/p^c)^l)^{\rm al}\]
  for some $c>0$.
  Then for any $\calP^+\in \{\calO\bC^+_{\rho}, \calO\widehat \bC^+_{\rho}, \calO\bC^{\dag,+}\}$ and for each $i\geq 0$, the natural map
  \[H^i(\Gamma,(\calL^+\otimes_{\OXp}\calP^+)(X_{\infty}))\to H^i(X_{\proet}/\frakX,\calL^+\otimes_{\OXp}\calP^+)\]
  is an almost isomorphism. Moreover, when $i=0$, it is an isomorphism.
\end{lem}
\begin{proof}
  The proof is similar to \cite[Lemma 5.6]{Sch2} and \cite[Lemma 2.7]{LZ}.
  Denote $X_{\infty}^{m/X}$ the $m$-fold fibre product of $X_{\infty}$ over $X$. As $X_{\infty}$ is a Galois cover of $X$ with Galois group $\Gamma$, we have $X_{\infty}^{m/X}\simeq X_{\infty}\times\Gamma^{m-1}$.
  Note that $\OXp/p^c$ comes from the \'etale sheaf $\calO_X^+/p^c$ on $X_{\et}$ and that $(\calL^+/p^c)^{\rm al} \cong ((\OXp/p^c)^l)^{\rm al}$. By \cite[Lemma 3.16]{Sch2}, for any $i\geq 0$ and $m\geq 1$, we have almost isomorphisms
  \[\Hom_{\cts}(\Gamma^{m-1},H^i(X_{\infty},\calL^+\otimes_{\OXp}\calP^+/p^c))\to H^i(X_{\infty}^{m/X},\calL^+\otimes_{\OXp}\calP^+/p^c) .\]
  By induction on $n$, we have almost isomorphisms
  \[\Hom_{\cts}(\Gamma^{m-1},H^i(X_{\infty},\calL^+\otimes_{\OXp}\calP^+/p^{nc}))\to H^i(X_{\infty}^{m/X},\calL^+\otimes_{\OXp}\calP^+/p^{nc}) ,\]
  for any $n\geq 1$. By letting $n$ go to $+\infty$, we get almost isomorphisms
  \begin{equation*}
      \Hom_{\cts}(\Gamma^{m-1},H^i(X_{\infty},\calL^+\otimes_{\OXp}\calP^+))\to H^i(X_{\infty}^{m/X},\calL^+\otimes_{\OXp}\calP^+)
  \end{equation*}
  for $\calP^+\in\{\calO\bC_{\rho}^{+,\leq r},\calO\widehat \bC^+_{\rho}\}$, where $\calO\bC_{\rho}^{+,\leq r}$ denotes the sub-sheaf of
  \[ \calO\bC^+_{\rho}\cong \OXp[\rho Y_1,\dots, \rho Y_d]\]
  consisting of polynomials of degrees $\leq r$. By the coherence of restricted pro-\'etale topos, $H^i(X_{\infty}^{m/X},-)$ commutes with direct limits for all $i$. Since $\calO\bC^+_{\rho} = \cup_{r\geq 0}\calO\bC^{+,\leq r}_{\rho}$, we also get desired almost isomorphisms for $\calP^+=\calO\bC^+_{\rho}$. A similar argument also works for $\calP^+=\calO\bC^{\dagger,+} = \cup_{\rho,\nu_p(\rho)>\nu_p(\rho_k)}\calO\widehat \bC^+_{\rho}$.
  Further more, when $i = 0$, since both sides are $\frakm_{\Cp}$-torsion free, so we get injections in this case.
  
  Now applying Cartan-Leray spectral sequence to the Galois cover $X_{\infty}\rightarrow X$ and using Lemma \ref{almost purity}, we conclude that the map
  \[H^i(\Gamma_{\infty}, (\calL^+\otimes_{\OXp}\calP^+)(X_{\infty}))\rightarrow H^i(X_{\proet}/\frakX,\calL^+\otimes_{\OXp}\calP^+)\]
  is an almost isomorphism for every $i\geq 0$.
  
  For $i = 0$, we know $H^0(X_{\proet}/\frakX,\calL^+\otimes_{\OXp}\calP^+)$ is the $(0,0)$-term of Cartan-Leray spectral sequence at the $\rE_2$-page, which is the kernel of the map
  \[(\calL^+\otimes_{\OXp}\calP^+)(X_{\infty})\rightarrow (\calL^+\otimes_{\OXp}\calP^+)(X^{2/X}_{\infty}).\]
  On the other hand, $H^0(\Gamma, (\calL^+\otimes_{\OXp}\calP^+)(X_{\infty}))$ is the kernel of the map 
  \[(\calL^+\otimes_{\OXp}\calP^+)(X_{\infty})\rightarrow \Hom_{\cts}(\Gamma,(\calL^+\otimes_{\OXp}\calP^+)(X_{\infty})).\]
  So the result follows from the injectivity of the map
  \[\Hom_{\cts}(\Gamma,(\calL^+\otimes_{\OXp}\calP^+)(X_{\infty}))\to(\calL^+\otimes_{\OXp}\calP^+)(X^{2/X}_{\infty}). \]
\end{proof}

\subsection{Proof of Theorem \ref{p-adic Simpson}}
  Now we are prepared to prove Theorem \ref{p-adic Simpson}.

  \begin{enumerate}
    \item Let $\calL$ be an $a$-small generalised representation of rank $l$ and $\calL^+$ be the sub-$\OXp$-sheaf as described in Definition \ref{Dfn-small generalised representation}. Denote $\calH^+: = \nu_*(\calL^+\otimes_{\OXp}\OC^{\dagger,+})$. It suffices to show that $\rR^i\nu_*(\calL^+\otimes_{\OXp}\OC^{\dagger,+})$ is $p^{\infty}$-torsion for any $i\geq 1$ and that $\calH^+$ satisfies conditions in Definition \ref{Dfn-small Higgs bundle}. Let $b>a$ and $\{\frakX_i\to\frakX\}_{i\in I}$ be as in Definition \ref{Dfn-small generalised representation}. Since the problem is local on $\frakX_{\et}$, we are reduced to showing that for any $i\in I$, if we write $\frakX_i = \Spf(R_i^+)$, then $H^n(X_{\proet}/\frakX_i, \calL^+\otimes_{\OXp}\OC^{\dagger,+})$ is $p^{\infty}$-torsion for any $n\geq 1$ and is a $b_i$-small Higgs module over $R_i^+$ for $n=0$ in the sense of Definition \ref{Dfn-Higgs module} for some $b_i>b$. So we only need to deal with the case for $\frakX$ small affine.
    
    Now we may assume $\frakX = \Spf(R^+)$ is affine small itself and that \[(\calL^+/p^{b'})^{\rm al}\cong ((\OXp)^l/p^{b'})^{\rm al}\]
    for some $b'>b$. Let $X_{\infty},\Rinfp$ and $\Gamma$ be as before. By Lemma \ref{reduce to local case}, the natural morphism
    \[H^i(\Gamma,\calL^+(X_{\infty})\otimes_{\Rinfp}S_{\infty}^{\dagger,+})\to H^i(X_{\proet}/\frakX,\calL^+\otimes_{\OXp}\OC^{\dagger,+})\]
    is an almost isomorphism for $i\geq 1$ and is an isomorphism for $i=0$. So we are reduced to showing $\RGamma(\Gamma,\calL^+(X_{\infty})\otimes_{\Rinfp}S_{\infty}^{\dagger})$ is discrete after inverting $p$ and $H^0(\Gamma,\calL^+(X_{\infty})\otimes_{\Rinfp}S_{\infty}^{\dagger})$ is a $b''$-small Higgs module for some $b''>b$.
    
    However, by Corollary \ref{take section-II}, $\calL^+(X_{\infty})$ is a $b''$-small $\Rinfp$-representation of $\Gamma$ for some fixed $b''>b$. So the result follows from Theorem \ref{local Simpson} (1).
    
    \item Let $(\calH,\theta_{\calH})$ be an $a$-small Higgs bundle of rank $l$ and $\calH^+$ be the $\calO_{\frakX}$-lattice as described in Definition \ref{Dfn-small Higgs bundle}. Fix an $a'$ satisfying $a<a'<b$. Denote $\calL^+ = (\calH^+\otimes_{\calO_{\frakX}}\OC^{\dagger,+})^{\Theta_{\calH} = 0}$. Then it is a subsheaf of $\calL=(\calH\otimes_{\calO_{\frakX}}\OC^{\dagger})^{\Theta_{\calH}=0}$ and hence $p$-torison free. We claim that the inclusion $\calO\bC^{\dagger,+}\to\calO\widehat \bC^+_{\rho_k}$ induces a natural isomorphism
    \[(\calH^+\otimes_{\calO_{\frakX}}\calO\bC^{\dagger,+})^{\Theta_{\calH}=0}\to(\calH^+\otimes_{\calO_{\frakX}}\calO\widehat \bC_{\rho_k}^+)^{\Theta_{\calH}=0}.\]
    Indeed, this is a local problem and therefore follows from Proposition \ref{Prop-Bigger ring}. As $\calH^+\otimes_{\calO_{\frakX}}\calO\widehat \bC^+_{\rho_k}$ is $p$-complete, by continuity of $\Theta_{\calH}$, so is $\calL^+$.
    It remains to prove that $\calL^+$ is locally almost trivial modulo $p^{a'+\nu_p(\rho_k)}$.
    
    Assume $\frakX = \Spf(R^+)$ is small affine and let $X_{\infty},\Rinfp$ and $\Gamma$ be as before. Shrinking $\frakX$ if necessary, we may assume $(\calH^+,\theta_{\calH})$ is induced by a $b'$-small Higgs module over $R^+$ for some $b'>a'$. Then by Theorem \ref{local Simpson}, $\calL^+(X_{\infty})$ is a $b'$-small $\Rinfp$-representation of $\Gamma$.
    
    Let us go back to the global case. Choose an \'etale covering $\{\frakX_i\to \frakX\}$ of $\frakX$ by small affine $\frakX_i = \Spf(R_i^+)$'s such that on each $\frakX_i$, $(\calH^+,\theta_{\calH^+})$ is induced by a $b_i$-small Higgs module over $R_i^+$ for some $b_i>a'$. Denote by $X_{i,\infty}$ the corresponding ``$X_{\infty}$'' for $\frakX_i$ instead of $\frakX$. As above, we have 
    \[\calL^+(X_{i,\infty})/p^{b_i}\cong (\OXp(X_{i,\infty})/p^{b_i})^l.\]
    Therefore, by the proof of \cite[Lemma 4.10(i)]{Sch2}, we get an almost isomorphism
    \[(\calL^+/p^{b_i})^{\rm al}_{\mid X_i}\cong ((\OXp/p^{b_i})^l)^{\rm al}_{\mid X_i}\]
    with $b_i>a'>a$ as desired.
    
    \item Let $\calL$ be an $a$-small generalised representation. There exists a natural morphism of Higgs complexes
    \[\iota:
    {\rm HIG}(\calH(\calL)\otimes_{\calO_{\frakX}}\OC^{\dagger},\Theta_{\calH(\calL)}) \to
      {\rm HIG}(\calL\otimes_{\OX}\OC^{\dagger},\Theta_{\calL})
    \]
    By construction of $(\calH(\calL),\theta_{\calH(\calL)})$, it follows from Theorem \ref{local Simpson} (4) that $\iota$ is an isomorphism. Since $\OC^{\dagger}$ is a resolution of $\OX$ by Theorem \ref{period sheaf}, we see that $\calL(\calH(\calL),\theta_{\calH(\calL)}) = \calL$. 
    The isomorphism 
    \[(\calH,\theta_{\calH})\to (\calH(\calL(\calH)),\theta_{\calH(\calL(\calH))})\]
    can be deduced in a similar way. So we get the equivalence as desired.
    
    It remains to show the equivalence preserves products and dualities. But this is a local problem, so we are reduced to Theorem \ref{local Simpson} (3).
    
    \item This follows from the same arguments in the proof of Theorem \ref{local Simpson} (4). Indeed, combining Theorem \ref{period sheaf} and the item $(3)$, we have a quasi-isomorphism 
    \[\calL\to {\rm HIG}(\calL\otimes_{\OX}\calO\bC^{\dagger},\Theta_{\calL}) \simeq {\rm HIG}(\calH\otimes_{\calO_{\frakX}}\calO\bC^{\dagger},\Theta_{\calH}).\]
    On the other hand, it follows from $(1)$ that there exists a quasi-isomorphism
    \[\rR\nu_*({\rm HIG}(\calH\otimes_{\calO_{\frakX}}\calO\bC^{\dagger},\Theta_{\calH}))\simeq {\rm HIG}(\calH,\theta_{\calH}).\]
    So we get a quasi-isomorphism 
    \[\rR\nu_*(\calL)\simeq{\rm HIG}(\calH,\theta_{\calH})\]
    as desired.
    
    \item
    Since $f:\frakX\to\frakY$ admits an $A_2$-lifting $\tilde f$, by Proposition \ref{relative resolution}, we get a morphism $f^*\calO\bC_Y^{\dagger}\to \calO\bC_X^{\dagger}$ which is compatible with Higgs fields.
    
    Assume $(\calH,\theta_{\calH})$ is an $a$-small Higgs field on $\frakY_{\et}$. Denote by $(f^*\calH,f^*\theta_{\calH})$ its pull-back along $f$.
    By (3), we get the following isomorphisms, which are compatible with Higgs fields:
 \begin{equation*}
    \begin{split}
        \calL(f^*\calH,f^*\theta_{\calH})\otimes_{\OX}\OC_X^{\dagger}  &\cong f^*\calH\otimes_{\calO_{\frakX}}\OC_X^{\dagger}\\
         &\cong f^*(\calH\otimes_{\calO_{\frakY}}\OC_Y^{\dagger})\otimes_{f^*\OC_Y^{\dagger}}\OC_X^{\dagger}\\
         &\cong f^*(\calL(\calH,\theta_{\calH})\otimes_{\widehat \calO_Y}\calO\bC^{\dagger}_Y)\otimes_{f^*\calO\bC_Y^{\dagger}}\calO\bC_X^{\dagger}\\
         &\cong f^*\calL(\calH,\theta_{\calH})\otimes_{\OX}\calO\bC^{\dagger}_X.
    \end{split}
    \end{equation*}
    After taking kernels of Higgs fields, we obtain that 
    \[\calL(f^*\calH,f^*\theta_{\calH})\cong f^*\calL(\calH,\theta_{\calH}).\]
    So the functor $(\calH,\theta_{\calH})\to \calL(\calH,\theta_{\calH})$ in (2) is compatible with the pull-back along $f$. But we have shown it is an equivalence, so its quasi-inverse must commute with the pull-back along $f$. This completes the proof.
  \end{enumerate}
 \begin{cor}\label{proet cohomology}
   Assume $\frakX$ is a liftable proper smooth formal scheme of relative dimension $d$ over $\calO_{\Cp}$.
   For any small generalised representation $\calL$, $\RGamma(X_{\proet},\calL)$ is concentrated in degree $[0,2d]$, whose cohomologies are finite dimensional $\Cp$-spaces.
 \end{cor}
 \begin{proof}
   Since we have assumed $\frakX$ is proper smooth, this follows from Theorem \ref{p-adic Simpson} (4) directly.
 \end{proof}
 \begin{rmk}
  Except the item (4), all results in Theorem \ref{p-adic Simpson} are still true by using $\calO\widehat \bC_{\rho_k}^+$ instead of $\OC^{\dagger,+}$. 
 \end{rmk}
 \begin{rmk}
   In Corollary \ref{proet cohomology}, one can also deduce that $\RGamma(X_{\proet},\calL)$ is concentrated in degree $[0,2d]$ when $\frakX$ is just quasi-compact of relative dimension $d$ over $\calO_{\Cp}$. Indeed, in this case, we have 
   \[\RGamma(X_{\proet},\calL)\simeq \RGamma(\frakX_{\et},{\rm HIG}(\calH,\theta_{\calH}))\simeq \RGamma(X_{\et},{\rm HIG}(\calH,\theta_{\calH})\otimes_{\calO_{\frakX}}\calO_{X_{\et}}),\]
   where ${\rm HIG}(\calH,\theta_{\calH})\otimes_{\calO_{\frakX}}\calO_{X_{\et}}$ denotes the induced Higgs complex on $X_{\et}$. On the other hand, by \'etale descent, the category of \'etale vector bundles on $X_{\et}$ is equivalent to the category of analytic vector bundles on $X_{\rm an}$, where $X_{\rm an}$ denotes the analytic site of $X$. So the Higgs complex ${\rm HIG}(\calH,\theta_{\calH})\otimes_{\calO_{\frakX}}\calO_{X_{\et}}$ upgrades to an analytic Higgs complex ${\rm HIG}(\calH_{\rm an},\theta_{\calH})$ such that
   \[{\rm HIG}(\calH_{\rm an},\theta_{\calH})\otimes_{\calO_{X_{\rm an}}}\calO_{X_{\et}} = {\rm HIG}(\calH,\theta_{\calH})\otimes_{\calO_{\frakX}}\calO_{X_{\et}}.\]
   By analytic-\'etale comparison (cf. \cite[Proposition 8.2.3]{FP}), for any coherent $\calO_{X_{\rm an}}$-module $\calM$, there is a canonical quasi-isomorphism
   \[\RGamma(X_{\rm an},\calM)\simeq \RGamma(X_{\et},\calM\otimes_{\calO_{X_{\rm an}}}\calO_{X_{\et}}).\]
   So by considering corresponding spectral sequences of these complexes, we get a quasi-isomorphism
   \[\RGamma(X_{\rm an},{\rm HIG}(\calH_{\rm an},\theta_{\calH}))\simeq\RGamma(X_{\et},{\rm HIG}(\calH,\theta_{\calH})\otimes_{\calO_{\frakX}}\calO_{X_{\et}}) .\]
   Now, the quasi-compactness of $\frakX$ implies that $X$ is a noetherian space. So the result follows from Grothendieck's vanishing theorem (cf. \cite[Th\'eor\`eme 3.6.5]{Gro}) directly. The author thanks anonymous referees for pointing out this.
 \end{rmk}

\section{Appendix}
In Appendix, we prove some elementary facts used in this paper. Throughout this section, we always assume $A$ is a $p$-complete flat $\calO_{\Cp}$-algebra.
\begin{dfn}\label{p-complete module}
  Let $\Lambda = \{\alpha\}_{\alpha\in\Lambda}$ be an index set and $I = \{i_{\alpha}\}_{\alpha}$ be a set of non-negative real numbers indexed by $\Lambda$. Define
  \begin{enumerate}
      \item $A[\Lambda] = \oplus_{\alpha\in\Lambda}A$;
      \item $A\za\Lambda\ya = \varprojlim_mA[\Lambda]/p^mA[\Lambda]$;
      \item $A[\Lambda,I] = \oplus_{\alpha\in\Lambda}p^{i_{\alpha}}A$;
      \item $A\za\Lambda,I\ya = \varprojlim_m(A[\Lambda,I]+p^mA[\Lambda])/p^mA[\Lambda]$;
      \item $A\za\Lambda,I,+\ya = \varprojlim_mA[\Lambda,I]/p^mA[\Lambda,I]$.
  \end{enumerate}
\end{dfn}
\begin{prop}\label{derived vs classical}
  \begin{enumerate}
      \item $A\za\Lambda\ya/A\za\Lambda,I\ya$ is the classical $p$-completion of $A[\Lambda]/A[\Lambda,I]$.
      \item $A\za\Lambda\ya/A\za\Lambda,I,+\ya$ is the derived $p$-completion of $A[\Lambda]/A[\Lambda,I]$.
  \end{enumerate}
\end{prop}
\begin{proof}
  Since $A\za\Lambda,I\ya$ is the closure of $A\za\Lambda,I,+\ya$ in $A\za\Lambda\ya$ with respect to the $p$-adic topology, the item $(1)$ follows from $(2)$ directly. So we are reduced to proving $(2)$.
  
  Consider the short exact sequence
  \[\xymatrix@C=0.45cm{
    0\ar[r] & A[\Lambda,I] \ar[r] & A[\Lambda] \ar[r] & A[\Lambda]/A[\Lambda,I] \ar[r] & 0.
  }\]
  For any $n\geq 0$, we get an exact triangle
  \[A[\Lambda,I]\otimes^L_{\Zp}\Zp/p^n\to A[\Lambda]\otimes^L_{\Zp}\Zp/p^n \to (A[\Lambda]/A[\Lambda,I])\otimes^L_{\Zp}\Zp/p^n\to .\]
  Applying $\rR\lim_n$ to this exact triangle and using $p$-complete flatness of $A$, we get the following exact triangle
  \[A\za\Lambda,I,+\ya[0]\to A\za\Lambda\ya[0] \to K\to ,\]
  where $K$ denotes the derived $p$-completion of $A[\Lambda]/A[\Lambda,I]$. Now, the item $(2)$ follows from the injectivity of the map $A\za\Lambda,I,+\ya\to A\za\Lambda\ya$.
\end{proof}

\begin{rmk}\label{describe of completion}
  For any $(\lambda_{\alpha})_{\alpha\in\Lambda}\in\prod_{\alpha\in\Lambda}A$, we write $\lambda_{\alpha}\xrightarrow{\nu_p}0$, if for any $M>0$ the set $\{\alpha\in\Lambda| \nu_p(\lambda_{\alpha})\leq M\}$ is finite. Then we have
  \[
    A\za\Lambda,I\ya = \{(\lambda_{\alpha})_{\alpha\in\Lambda}| \nu_p(\frac{\lambda_{\alpha}}{p^{i_{\alpha}}})\geq 0\}
  \]
  and
  \[
    A\za\Lambda,I,+\ya = \{(\lambda_{\alpha})_{\alpha\in\Lambda}|\nu_p(\frac{\lambda_{\alpha}}{p^{i_{\alpha}}})\geq 0, \frac{\lambda_{\alpha}}{p^{i_{\alpha}}}\xrightarrow{\nu_p}0\}.
  \]
\end{rmk}
\begin{dfn}\label{equivalent of Basis}
   Assume $M$ is a (topologically) free $A$-module. Let $\Sigma_1$ and $\Sigma_2$ be two subsets of $M$.
   \begin{enumerate}

   \item We write $\Sigma_1 \sim \Sigma_2$, if they (topologically) generate the same sub-$A$-module of $M$

   \item We write $\Sigma_1 \approx \Sigma_2$, if both of them are sets of (topological) basis of $M$. In this case, we also write $M\approx\Sigma_1$ if no ambiguity appears.
   \end{enumerate}
 \end{dfn}
  \begin{prop}\label{basis of free modules}
   Fix $\epsilon,\omega \in\calO_{\Cp}$.
   Let $M$ be a (topologically) free $A$-module with basis $\{x_i\}_{i\geq 0}$. If $N\subset M$ is a submodule such that
   \[N\sim \{ \omega(x_i+i\epsilon x_{i-1})\mid~ i\geq 0\},\]
   where $x_{-1} = 0$, then $N=\omega M$.
 \end{prop}
 \begin{proof}
   Put $y_i = x_i+i\epsilon x_{i-1}$ for all $i$. Then we see that
   \[(y_0, y_1, y_2, y_3, \cdots) = (x_0, x_1, x_2, x_3, \cdots)\cdot X\]
   with
   \[X =
     \left(
       \begin{array}{ccccc}
         1 & \epsilon & 0 & 0 & \cdots \\
         0 & 1 & 2\epsilon & 0 & \cdots \\
         0 & 0 & 1 & 3\epsilon & \cdots \\
         0 & 0 & 0 & 1 & \cdots \\
         \vdots & \vdots & \vdots & \vdots & \ddots \\
       \end{array}
     \right),
   \]
   and that
   \[(x_0, x_1, x_2, x_3, \cdots) = (y_0, y_1, y_2, y_3, \cdots)\cdot Y\]
   with
   \[Y =
     \left(
       \begin{array}{ccccc}
         1 & -\epsilon & 2\epsilon^2 & -6\epsilon^3 & \cdots \\
         0 & 1 & -2\epsilon & 6\epsilon^2 & \cdots \\
         0 & 0 & 1 & -3\epsilon & \cdots \\
         0 & 0 & 0 & 1 & \cdots \\
         \vdots & \vdots & \vdots & \vdots & \ddots \\
       \end{array}
     \right).
   \]
   The $(i,j)$-entry of $Y$ is $\delta_{ij}$ if $i\geq j$ and is $(-\epsilon)^{j-i}\frac{(j-1)!}{(i-1)!}$ if $i<j$. Then the proposition follows from the fact
   $XY = YX =Id$.
 \end{proof}
   The following proposition can be proved in the same way.
 \begin{prop}\label{basis of free modules II}
   Fix $\Theta\in\rM_l(A)$.
   Let $M$ be a (topologically) free $A$-module with basis $\{x_i\}_{i\geq 0}$. Let $N$ be a finite free $R$-module of rank $l$ with a set of basis $\{e_1, \cdots, e_l\}$. For every $1\leq j\leq l$ and $i\geq 0$, put $f_{j,i}\in N\otimes_AM$ satisfying
   \[(f_{1,i}, \cdots, f_{l,i})=(e_1\otimes x_i, \cdots, e_l\otimes x_i)+i(e_1\otimes x_{i-1}, \cdots, e_l\otimes x_{i-1})\Theta,\]
   where $x_{-1} = 0$.
   Then $N\otimes_AM\approx\{f_{j,i}\mid 1\leq j\leq l, i\geq 0\}$.
 \end{prop}


\begin{thebibliography}{99}
 
 \bibitem[AGT16]{AGT}A. Abbes, M. Gros, T. Tsuji: {The $p$-adic Simpson Correspondence}, Annals of Mathematics Studies. 193, 2016.


 \bibitem[Bei12]{Bei}A. Beilinson: {\it p-adic periods and derived de Rham cohomology}, JAMS, Vol. 25, NO. 3, pp. 715-738, July 2012.

 \bibitem[Bha12]{Bha}B. Bhatt: {$p$-adic derived de Rham cohomology}, arxiv:1204.6560v1, 2012.


 \bibitem[BMS18]{BMS1} B. Bhatt, M. Morrow, P. Scholze: {\it Integral $p$-adic Hogde Theory},  Publ. math. de l'IH\'ES 128, pp. 219-395, 2018. 
 
 \bibitem[BMS19]{BMS2}B. Bhatt, M. Morrow, P. Scholze: {\it Topological Hochschild homology and integral $p$-adic Hodge theory}, Publ. math. de l'IH\'ES 129, pp. 199–310, 2019.



 \bibitem[DLLZ19]{DLLZ19}H. Diao, K.-W. Lan, R. Liu, X. Zhu: {\it Logarithmic adic spaces: some foundational results}, arxiv:1912.09836v1, 2019.
 
 \bibitem[DLLZ22]{DLLZ22}H. Diao, K.-W. Lan, R. Liu, X. Zhu: {\it Logarithmic Riemman-Hilbert correspondences for rigid varieties}, to appear in  \href{https://doi.org/10.1090/jams/1002}{JAMS}, 2022.

 \bibitem[DW05]{DW}C. Deninger, A. Werner: {\it Vector bundles on p-adic curves and parallel transport} Ann. Sci. Ecole Norm. Sup. 38, pp. 553–597, 2005.

 \bibitem[Fal88]{Fal1}G. Faltings: {\it $p$-adic Hogde Theory}, J. Am. Math. Soc. 1, pp. 255-299, 1988.

 \bibitem[Fal05]{Fal2}G. Faltings: {\it A $p$-adic Simpson correspondence}, Advances in Mathematics. 198, pp. 847-862, 2005.

 \bibitem[FF19]{FF}L. Fargues, J.-M. Fontaine: {\it Courbes et fibr\'es vectoriels en th\'eorie de Hogde $p$-adique}, Ast\'erisque No.406, 2019.

 \bibitem[Fon82]{Fon}J.-M. Fontaine: {\it Forms diff\'erentielles et modules de Tate des vari\'et\'es ab\'eliennes sur les corps locaux}, Inv. Math. 65, pp. 379-409, 1982.
 
 \bibitem[FP]{FP}J. Frensel, M. van der Put: {\it Rigid analtic geometry and its applications}, Progress in mathematics (Boston, Mass.), v. 218, 2004.
 
 \bibitem[GR03]{GR}O. Gabber, L. Remero: {\it Almost Ring Theory}, Vol 1800 of Lecture Notes in Mathematics. Springer-Verlag, Berlin$\cdot$ Heidelberg$\cdot$ New York, 2003.

 \bibitem[Gro57]{Gro} A. Grothendieck: {\it: Sur quelques points d'alg\`ebre homologique}, Tohoku Math. J. (2) 9, 119-221, 1957.

 \bibitem[Heu20]{Heu} Ben Heuer: {\it Line bundles on rigid spaces in the v-topology}, arXiv: 2012.07918v2, 2020.
 
 \bibitem[HMW21]{HMW} Ben Heuer, Lucas Mann, Annette Werner: {\it The $p$-adic Corlette-Simpson correspondence for abeloids }, arXiv: 2107.09403v2, 2021.
 
 \bibitem[Hy89]{Hy}O. Hyodo: {\it On variantion of Hodge-Tate structures}, Math. Ann. 284, pp. 7-22, 1989.

 \bibitem[Ill71]{Ill1}L. Illusie: {\it Complexe cotangent et d\'eformations I}, Vol 239 of Lecture Notes in Mathematics. Springer-Verlag, Berlin$\cdot$ Heidelberg$\cdot$ New York, 1971.

 \bibitem[Ill72]{Ill2}L. Illusie: {\it Complexe cotangent et d\'eformations II}, Vol 283 of Lecture Notes in Mathematics. Springer-Verlag, Berlin$\cdot$ Heidelberg$\cdot$ New York, 1972.

 \bibitem[LZ17]{LZ}R. Liu, X. Zhu: {\it Rigidty and a Riemman-Hilbert correspondence for $p$-adic local systems}, Inv. Math. 207, pp. 291-343, 2017.

 \bibitem[MT20]{MT}M. Morrow, T. Tsuji: {\it Generalized representations as $q$-connections in integral $p$-adic Hodge theory}, arXiv: 2010.04059v2, 2020.

 \bibitem[OV07]{OV} A. Ogus, V. Vologodsky: {\it Nonabelian Hogde theory in characteristic $p$}, Publ. math. de l'IH\'ES 106, pp. 1–138, 2007.

 \bibitem[Qui70]{Qui} D. Quillen: {On the (co-)homology of commutative rings}, in Applications of Categorical Algebra, Proc. Sympos. Pure Math, Vol. XVII, pp. 65-87, Am. Math. Soc., Providence, 1970.

 \bibitem[Sch12]{Sch1}P. Scholze: {\it Perfectoid spaces}, Publ. math. de l'IH\'ES 116, no. 1, pp. 245-313, 2012.

 \bibitem[Sch13a]{Sch2}P. Scholze: {\it $p$-adic Hodge theory for rigid-analytic varieties},  Forum of Mathematics, Pi, 1, e1, 77 pages, 2013.

 \bibitem[Sch13b]{Sch3}P. Scholze: {\it Perfectoid spaces: a survey}, Current developments in mathematics, 2012, Volume 2012, Issue 1, pp. 193-227.

 \bibitem[Sim92]{Sim}C.T. Simpson: {\it Higgs bundles and local systems}, Publ. math. de l'IH\'ES 75, pp. 5-95, 1992.

 \bibitem[Wei]{Wei}C.A. Weibel: {\it An introduction to homological algebra}, Camb. Stu. in. Adv. Math., 2004.
 
 \bibitem[YZ20]{YZ}J. Yang, K. Zuo: {\it A note on $p$-adic Simpson correspondence}, arXiv:2012.02058v4, 2020.
\end{thebibliography}
\end{document}